\documentclass[sn-mathphys-num,iicol]{sn-jnl}%

\usepackage{graphicx}%
\usepackage{multirow}%
\usepackage{amsmath,amssymb,amsfonts}%
\usepackage{amsthm}%
\usepackage{mathrsfs}%
\usepackage[title]{appendix}%
\usepackage{xcolor}%
\usepackage{textcomp}%
\usepackage{manyfoot}%
\usepackage{booktabs}%
\usepackage{algorithm}%
\usepackage{algorithmicx}%
\usepackage{algpseudocode}%
\usepackage{listings}%

\usepackage{stmaryrd}
\usepackage{tikz}
\usepackage{times}
\usepackage{mathptmx}
\usepackage{supertabular}

\graphicspath{{./figures/}}
\theoremstyle{thmstyleone}%
\newtheorem{theorem}{Theorem}[section]%

\theoremstyle{thmstyletwo}%
\newtheorem{ex}[theorem]{Example}%
\newtheorem{remark}[theorem]{Remark}%

\theoremstyle{thmstylethree}%
\newtheorem{definition}[theorem]{Definition}%
\newtheorem{tec}[theorem]{Principle}

\algnewcommand\Initialization{\item[\textbf{Initialization:}]}%

\let\OLDthebibliography\thebibliography
\renewcommand\thebibliography[1]{
  \OLDthebibliography{#1}
  \setlength{\parskip}{0pt}
  \setlength{\itemsep}{0pt plus 0.3ex}
}

\hypersetup{
pdftitle={A Second-Order TGV Discretization with 90° Rotational Invariance Property},
pdfauthor={Alireza Hosseini, Kristian Bredies},
pdfkeywords={Image processing, total generalized variation, TGV discretization, inverse problem, primal-dual algorithm}
}

\begin{document}

\begin{fmremark}
  fjdkfjd
\end{fmremark}

\title[A Second-Order TGV Discretization]{\texorpdfstring{A Second-Order TGV Discretization with 90$^\circ$ Rotational Invariance Property\textsuperscript{*}}{A Second-Order TGV Discretization with 90° Rotational Invariance Property}}

\author*[1]{\fnm{Alireza} \sur{Hosseini}}\email{hosseini.alireza@ut.ac.ir}

\author[2]{\fnm{Kristian} \sur{Bredies}}\email{kristian.bredies@uni-graz.at}
\equalcont{\emph{Journal of Mathematical Imaging and Vision}, \textbf{67}, 23 (2025). Springer Science+Business Media, DOI:~10.1007/s10851-024-01224-8.}

\affil*[1]{\orgdiv{School of Mathematics, Statistics and Computer Science, College of Science}, \orgname{University of Tehran}, \orgaddress{\street{P.O. Box 14115-175}, \city{Tehran}, \country{Iran}}. Email: \textcolor{blue}{hosseini.alireza@ut.ac.ir}}

\affil[2]{\orgdiv{Department of Mathematics and Scientific Computing}, \orgname{University of Graz}, \orgaddress{\street{Heinrichstra\ss e 36}, \postcode{8010} \city{Graz}, \country{Austria}}. Email: \textcolor{blue}{kristian.bredies@uni-graz.at}}

\abstract{In this work, we propose a new discretization for second-order total generalized variation (TGV) with some distinct properties compared to existing discrete formulations. The introduced model is based on the same design principles as Condat's discrete total variation model (Condat in SIAM J Imaging Sci 10(3):1258--1290, 2017) and shares its benefits, particularly improved solution quality for imaging problems. We propose an algorithm for general discrete inverse problems with second-order TGV using the new discretization. Numerical results obtained with this algorithm for denoising and upscaling demonstrate the advantages of the discretization. Moreover, to assess the invariance properties of the new model, we compare the results of the proposed TGV and the classic discrete TGV for original data and 90$^{\circ}$ rotated versions. Additionally, we provide an algorithm for calculating the TGV value with respect to the new discretization model.}



\maketitle

	\section{Introduction} \label{sec:intro}
	Image reconstruction is a major subject in image and signal processing, applicable in areas such as medical imaging, pattern recognition, and video coding. Various techniques are used for image reconstruction, including spatial filtering \cite{spacial, spacial2}, transform domain filtering \cite{wav1,wav2,wav3}, methods based on partial differential equations \cite{dif1,dif2}, variational methods \cite{upwind, Chambol3, condat1, Shanon, bredies}, and machine learning approaches such as deep learning \cite{deep1,deep2}
	and linear regression \cite{reg}. In this paper, we contribute to variational methods in order to make progress in this area. In particular, a new kind of discrete variational model is proposed to solve image processing tasks. The proposed model is associated with a new discretization of the so-called second-order total generalized variation (TGV) \cite{bredies}.\\
	In imaging problems, it is common to solve inverse problems. Generally, solving an inverse problem amounts to solving an equation of the form
	$$z=G(u),$$
	where $u$ is the initial ``perfect'' image in a continuous domain (e.g., $u\in L^1_{loc}(\Omega)$, for $\Omega\subset\mathbb{R}^d$ domain), $G$ is a forward operator such as blurring, sampling, or more generally, some linear operator, and $z$ is the measured data. The problem is thus reconstructing $u$ from the given data $z$. Due to the ill-posed nature of many inverse problems, regularization is necessary. Tikhonov regularization is a common approach, formulated as an optimization problem of the form
	\begin{equation}\label{OP}
		\min_u \ \mathcal{F}(u)+\mathcal{R}(u),
	\end{equation}
	where $\mathcal{F}$ represents the data fidelity and $\mathcal{R}$ is the regularization functional. The most common fidelity term is of the form
	$$\mathcal{F}(u)=\frac{1}{2}\|G(u)-z\|^2,$$
	where $\|\cdot\|$ is a given norm.
	The regularization functional $\mathcal{R}$ is commonly adapted for imaging problems and the associated applications such as medical imaging, and machine vision. Standard Tikhonov regularization approaches
	will usually consider quadratic forms such as $\mathcal{R}(u)=\frac{1}{2}\|u\|_2^2$ or $\frac{1}{2}\|\nabla u\|_2^2.$ However, it is shown in \cite{tv}, that for denoising problems, $\mathcal{R}(u)=\frac{1}{2}\|u\|_2^2$, often does not provide adequate spatial regularization for imaging problems. Therefore, this is an inadequate choice, since all
	natural images admit a lot of spatial regularity. On the other hand, the second
	case ($\mathcal{R}(u)=\frac{1}{2}\|\nabla u\|_2^2$), normally imposes too much spatial regularization.\\
	In a pioneering work, Rudin et al. \cite{ROF} introduced  ``Total Variation'' (TV) as a regularizer for inverse problems in imaging. This model is straightforward, easy to discretize, and yields reliable numerical results for imaging problems. For instance, it can achieve acceptable regularization for denoising problems, although some artifacts, known as staircasing artifacts, may still remain. Due to these limitations, the TV model has been generalized through the introduction of ``Total Generalized Variation'' (TGV) \cite{bredies}. This generalization defines the $k$-th order TGV ($\text{TGV}^k$) for $k\geq 1$, where $\text{TGV}^1$ corresponds to TV up to a positive factor.\\
	The second-order TGV ($\text{TGV}^2$) is the most commonly used TGV model for imaging problems, exhibiting superior performance for piecewise smooth images compared to TV. Notably, the typical artifacts associated with the TV model are not observed in the results obtained with TGV regularization.\\
	The continuous definition of total variation (see Section 2 for the definition of TV) has the desirable property of isotropy, meaning a rotation of an image in the plane does not change the TV value. It is natural to expect the discretized form of total variation to be rotation-invariant, at least for rotations of any integer multiple of 90 degrees. However, despite its name, isotropic TV does not possess this isotropy property. Efforts have been made to improve on this, resulting in the introduction of other versions of discrete TV with different properties, such as upwind TV  \cite{upwind}, which is about the discrete coarea formula, Shannon discrete TV \cite{Shanon}, Condat's TV which attempts to improve isotropy \cite{condat1}, and an approximation of TV using nonconforming P1 (Crouzeix--Raviart) finite
	elements \cite{finit}.\\
	Condat \cite{condat1}, proposed a new discretization of TV ($\text{TV}_c$), which is inspired by the dual formulation of isotropic TV with additional constraints from domain conversion operators. $\text{TV}_c$ is shown to be invariant to 90-degree rotations and exhibits better performance in imaging problems such as denoising and upscaling compared to classic discrete isotropic TV. Since our new proposed discrete TGV model is inspired by Condat's discretization, we give a brief explanation of this model in the course of this paper.\\
	Additionally, a discretization approach for second-order TGV is presented in \cite{bredies}, which is a straightforward generalization of classic isotropic TV obtained by discretizing the dual formulation of the continuous $\text{TGV}^2$ functional (see Section 2 for the definition of $\text{TGV}^2$). This approach is referred to as classic discrete TGV. {Currently, there are some other sophisticated discretization strategies available in the literature, designed for different aims. Shannon TGV \cite{Hosseini} is a recent model that generalizes Shannon TV. This approach aims to reduce aliasing effects that appear in reconstructed images at sub-pixel levels after upscaling while preserving the ability of TGV to reduce staircases. Another recent discrete TGV approach \cite{Lukas} introduces a discrete version of the second-order TGV-seminorm for piecewise constant functions on general triangular meshes. It is based on lowest-order discontinuous Lagrange and Raviart--Thomas finite elements, allowing TGV to be applied to data structures other than regular grids.} \\
	This paper contributes by designing a new discretization of  $\text{TGV}^2$ in two dimensions with favorable properties compared to existing standard discretization approaches. For TGV, the standard discretization is via finite-difference operators, which have known drawbacks such as lack of rotational invariance, even for grid-preserving 90-degree rotations. The design of approaches with more favorable properties remains an open problem, and this paper aims to fill this gap by generalizing the state-of-the-art strategy proposed by Condat \cite{condat1}.\\
	The rest of the paper is organized as follows: Section 2 provides a review of TV and TGV functionals, including definitions and existing discretizations. In Section 3, the new discrete second-order TGV is designed, introducing staggered grid sets and elementary operators. New difference operators are proposed, and together with domain conversion operators, the mathematical model of the new discrete second-order TGV is formulated. Section 4 explains some basic invariance properties of the proposed model. Section 5 proposes numerical algorithms for solving some discrete inverse problems in image processing and compares the results to classic  discrete TV, Condat's discrete TV, {Shannon TGV,} and classic discrete TGV. Furthermore,
	the rotation invariance of the newly proposed discrete TGV with respect to integer multiples of 90-degree rotations is illustrated numerically and compared to the existing classic discrete TGV model \cite{bredies}.
	\section{TV, TGV and Their Discretizations}
	In the following, we review TV and TGV functionals including two well-known TV discretization models: isotropic TV, Condat's TV, and the classic TGV discretization.
	\subsection{Total Variation} \label{sec:tv}
	The total variation (TV) is a prevalent functional frequently employed to regularize ill-posed  inverse problems in imaging.
	In continuous domains, the total variation of $u\in L^1_{loc}(\Omega), \Omega\subset \mathbb{R}^d$ domain is defined by
	\begin{multline}\label{j}
		\text{TV}(u)=\sup\biggl\{-\int_{\Omega}u\cdot \text{div} \phi\ dX: \phi\in C_c^{1}(\Omega,\mathbb{R}^d), \\  |\phi(X)|\leq 1\ \forall X\in\Omega\biggr\}.
	\end{multline}
	For $u\in C^1(\Omega)$ (or $W^{1,1}(\Omega)$), it can be verified that
	\begin{equation}\label{CDJ}
		\text{TV}(u)=\int_{\Omega}|\nabla u|\ dX.
	\end{equation}
	In this paper, (\ref{CDJ}) and (\ref{j}), are referred to the primal and the dual formulation of the TV, respectively.
	\subsection{Second-Order TGV} \label{sec:tgv}
	Total generalized variation of order $k$ ($\text{TGV}^k, k\in\mathbb{N})$, is a regularization functional introduced in \cite{bredies}. For theoretical aspects, see \cite{brediestgv11,tgv5,brediestgv12,brediestgv13,tgv4,brediestgv14}, and for applications of TGV, refer to \cite{brediestgv15,brediestgv16,brediestgv17,brediestgv18,brediestgv19,tgv2,tgv3}.  It generalizes the total variation functional TV (\ref{j}) (in the sense of $\text{TGV}^1=\alpha_0\text{TV}$). When $k\geq 2,$ $\text{TGV}^k$ exhibits exceptional properties, such as attenuating artifacts, especially staircase artifacts in imaging problems, which are common with TV regularization.
	The second-order TGV in the continuous domain is defined by:
	\begin{multline}\label{tgv1}
		\text{TGV}_{\alpha}^2(u)=\sup \ \biggl\{\int_{\Omega}u\cdot\text{div}^2 v\ dX: v\in C_c^{2}(\Omega, \text{Sym}^{2}(\mathbb{R}^d)),  \\ \|\text{div}^l v\|_{\infty}\leq \alpha_l, l=0,1 \biggr\},
	\end{multline}
	for $u\in L^1_{loc}(\Omega)$, $\Omega\subset\mathbb{R}^d$ domain, $\alpha_0>0,\alpha_1>0$. In the paper, we focus on the two dimensional setting, i.e., $d=2.$ If $u\in C^2(\Omega),$ this definition can be rewritten to:
	\begin{align} \notag
		\text{TGV}^2_{\alpha}(u)&= \min_w \ \alpha_1\int_{\Omega}|\nabla u-w|\ dX+\alpha_0\int_{\Omega}|\epsilon(w)|\ dX, \\
		\epsilon(w) &= \left(
		\begin{array}{cc}
			\frac{\partial w_1}{\partial x}                                           & \frac{\frac{\partial w_1}{\partial y}+\frac{\partial w_2}{\partial x}}{2} \\ \label{tgv4}
			\frac{\frac{\partial w_1}{\partial y}+\frac{\partial w_2}{\partial x}}{2} & \frac{\partial w_2}{\partial y}
		\end{array}
		\right),
	\end{align}
	where $X = (x,y)$.
	Referring to \cite{bredies}, $C_c^{2}(\Omega, \text{Sym}^{2}(\mathbb{R}^2))$ is the set of symmetric $\mathbb{R}^{2\times 2}$ matrices whose components belong to $C_c^2(\Omega)$.
	In this paper, (\ref{tgv4}) and (\ref{tgv1}), are referred to the primal and the dual formulation of the second-order TGV, respectively.
	\begin{remark}
		In the above $\text{TGV}$
		definition, for $d=2$, $\Omega\subset\mathbb{R}^2, u:\Omega\rightarrow \mathbb{R}, u\in L_{loc}^1(\Omega)$, $v=\left(\begin{array}{cc}
			v_1 & v_3 \\
			v_3 & v_2
		\end{array}\right)$, $v_1,v_2,v_3\in C_c^2(\Omega)$.
		From the definition of the $\text{divergence}$ in \cite{bredies}, we have
		\begin{equation}\label{divd3}
			\text{Div}\ v=\left(\begin{array}{c}
				\frac{\partial v_1}{\partial x}+\frac{\partial v_3}{\partial y} \\
				\frac{\partial v_2}{\partial y}+\frac{\partial v_3}{\partial x}
			\end{array}
			\right).
		\end{equation}
		Moreover, for $w=(w_1,w_2), w_1, w_2\in C_c^1(\Omega)$,
		\begin{equation}
			\text{div}\ w= \frac{\partial w_1}{\partial x}+ \frac{\partial w_2}{\partial y},
		\end{equation}
		From the fact that $\text{div}^2 v=\text{div}(\text{Div} \ v),$ and the definition of $\text{div}\ v$, it can be easily seen that
		$$\text{div}^2 v=\frac{\partial^2 v_1}{\partial x^2 }+\frac{\partial^2 v_2}{\partial y^2 }+2\frac{\partial^2 v_3}{\partial x\partial y}.$$
	\end{remark}
	To discretize TV and TGV, we need some forward difference operators.  In the following, let $A=\{1,\ldots, N_1\}\times\{1,\ldots, N_2\}.$ We define the discrete operators:\\
	$D_{x+}:\mathbb{R}^{N_1\times N_2}\rightarrow \mathbb{R}^{N_1\times N_2}$ as discrete approximation of the partial derivative with respect to direction $x$:\\
	\begin{equation}\label{dxx+}
		(D_{x+}u)(n_1,n_2)=\left\{ \begin{array}{l}
                  u(n_1+1,n_2) -u(n_1,n_2), \\
                   \quad \ \ (n_1,n_2)\in A\setminus (\{N_1\}\times\{1,2,\ldots, N_2\}), \\
			0, \ \ \ \text{else},
		\end{array}
		\right.
	\end{equation}
	assuming  homogeneous discrete Neumann boundary conditions
	on $\{N_1\}\times\{1,2,\ldots, N_2\},$ i.e.,   $$u(N_1+1,n_2)=u(N_1,n_2),  \quad n_2=1,\ldots, N_2.$$ 
	$D_{y+}:\mathbb{R}^{N_1\times N_2}\rightarrow \mathbb{R}^{N_1\times N_2}$ as discrete approximation of the partial derivative with respect to direction $y$:\\
	\begin{equation}\label{dyy+}
		(D_{y+}u)(n_1,n_2)=\left\{ \begin{array}{l}
                  u(n_1,n_2+1)-u(n_1,n_2), \\
                  \quad \ \ (n_1,n_2)\in A\setminus (\{1,2,\ldots, N_1\}\times \{N_2\}), \\
			0,  \ \ \ \text{else},
		\end{array}
		\right.
	\end{equation}
	assuming  homogeneous discrete Neumann boundary conditions
	on $\{1,2,\ldots, N_1\}\times \{N_2\},$ i.e.,   $$u(n_1,N_2+1)=u(n_1,N_2), \quad n_1=1,\ldots, N_1.$$
	Moreover, we define
	$\mathcal{D}=(D_{x+}, D_{y+})$.
	\subsection {Classic Discrete TV (Isotropic TV) as a Discrete TV that is not isotropic Literally} \label{sec:iso_tv}
	The classic discrete TV, also called isotropic TV, for a discrete image $u\in\mathbb{R}^{N_1\times N_2}$ is defined by:
	\begin{equation}\label{ITV}
		\text{TV}_i(u)=\sum_{n_1=1}^{N_1}\sum_{n_2=1}^{N_2} \sqrt{(D_{x+}u)(n_1,n_2)^2+(D_{y+}u)(n_1,n_2)^2}.
	\end{equation}
	This definition is inspired from the primal formulation of $\text{TV}$ for smooth functions (\ref{CDJ}), with replacing the differential operators $\frac{\partial u}{\partial x}$ and $\frac{\partial u}{\partial y}$ by finite-difference operators $D_{x+}$ and $D_{y+}$ (see (\ref{dxx+}) and (\ref{dyy+})). {Already being considered in the seminal paper \cite{ROF}, it evolved to the standard and most popular choice for discrete TV. Reasons for that may, on the one hand, be its simplicity and, on the other hand, be the availability of efficient and widespread computational algorithms that {are} based on this discretization (see, for instance, \cite{chambollealg,pridua}).}
	It is easy to see that classic discrete TV has a dual form which can be implemented by the following optimization problem:
	\begin{align} \notag
			\text{TV}_i(u) &=\max_{v\in {(\mathbb{R}^2)}^{N_1\times N_2}}\langle \mathcal{D}u,v\rangle \\ \notag &=\max_{v\in {(\mathbb{R}^2)}^{N_1\times N_2}}\sum_{n_1=1}^{N_1}\sum_{n_2=1}^{N_2} (D_{x+}u)(n_1,n_2)\cdot v_1(n_1,n_2) \\[-0.5em] \notag & \qquad\qquad\qquad\qquad\qquad +(D_{y+}u)(n_1,n_2)\cdot v_2(n_1,n_2), \\[0.25em] \notag
                                       &\text{subject\ to}\quad |v(n_1,n_2)|\leq 1, \\ \label{dis}
          & \qquad\qquad\quad \forall(n_1,n_2)\in \{1,\ldots, N_1\}\times\{1,\ldots, N_2\},
	\end{align}
	where, for $v\in\mathbb{R}^2,$ $|v|=\sqrt{v_1^2+v_2^2}.$
	Surprisingly, this discretization is not invariant with respect to $90^{\circ}$ rotations. In other words, if $\mathcal{R}_{90}u$ is the $90^\circ$ rotated version of a discrete image $u,$ then generally $\text{TV}_i(u)\neq \text{TV}_i(\mathcal{R}_{90}u).$ {For example, if} $u\in\mathbb{R}^{2\times2}, u(1,1)=1, u(n_1,n_2)=0, n_1,n_2=1,2, (n_1,n_2)\neq (1,1),$ then it is easy to see that $\text{TV}_i(u)=\sqrt{2}$ and $\text{TV}_i(\mathcal{R}_{90}u)=2.$ For a general angle $\theta$ using the rotation operator $\mathcal{R}_{\theta},$ we also observe non-invariance. However, { a question arises:} what kind of rotation invariance can we expect from a discretization? In the sequel, we review Condat's discrete TV,  a modified model designed {through} grid domain conversions. This discretization is exact up to numerical precision with respect to $ 90^\circ$ rotations and {provides} a better approximation for the other rotation angles {compared} to $\text{TV}_i.$
	\subsection{Condat's TV as a More Isotropic Discrete TV } \label{sec:condat_tv}
	Based on the dual formulation of the continuous TV (\ref{j}) (whose discretization is given in (\ref{dis})) and introducing three linear operators $L_{\bullet}, L_{\leftrightarrow}$ and $L_{\updownarrow}$ over $(\mathbb{R}^2)^{N_1\times N_2}$, Condat \cite{condat1} proposed a new discretization of the total variation.  {Consider dual version of isotropic TV (\ref{dis}), and its constraint 
		\begin{multline}\label{CTT}
			\sqrt{v_1(n_1,n_2)^2+v_2(n_1,n_2)^2}\leq 1,\\ \forall(n_1,n_2)\in \{1,\ldots, N_1\}\times\{1,\ldots, N_2\}.
		\end{multline}
		where  $v_1$ is located at the same grid domain as $D_{x+}u,$ and $v_2$ is located at the same grid domain as $D_{y+}u$. If we accept that the location of difference of two objects is in the middle of them, then, it is easy to see that the grid locations of  $D_{x+}u$ and $D_{y+}u$ are different. Condat \cite{condat1}, suggests a strategy, to impose a unification  to the locations of $v_1,$ and $v_2,$ in constraint (\ref{CTT}), in the way that guarantee invariant property with respect to $90^\circ$ rotations. Three converting operators to convert $v_1,$ and $v_2$ to three different grid domains are defined; $L_\bullet$ (converts signals to the centers of pixels), $L_\leftrightarrow$ (converts signals to the middles of vertical edges of pixels) and $L_\updownarrow$ (converts signals to the middles of horizontal edges of pixels). Objective function of Condat's 
		model is the same of isotropic TV (\ref{dis}), but instead of (\ref{CTT}), the following three constraints are introduced:
		\begin{multline}\label{CTT2}
			\sqrt{(L_\star v)_1(n_1,n_2)^2+(L_\star v)_2(n_1,n_2)^2}\leq 1,\\ \forall(n_1,n_2)\in \{1,\ldots, N_1\}\times\{1,\ldots, N_2\},\ \star\in \{\bullet, \leftrightarrow, \updownarrow \}.
		\end{multline}
		Consequently, Condat-TV reads as }
	\begin{align}\notag
          \displaystyle
          \text{TV}_c(u) &= \max_{v\in {(\mathbb{R}^2)}^{N_1\times N_2}}\langle \mathcal{D}u,v\rangle \\ \notag &=\max_{v\in {(\mathbb{R}^2)}^{N_1\times N_2}}\sum_{n_1=1}^{N_1}\sum_{n_2=1}^{N_2} (D_{x+}u)(n_1,n_2)\cdot v_1(n_1,n_2) \\ \notag & \qquad\qquad\qquad\qquad\quad  +(D_{y+}u)(n_1,n_2)\cdot v_2(n_1,n_2), \\ \notag
			& \text{subject\ to}\quad |L_{\bullet}v(n_1,n_2)|\leq 1, \  |L_{\leftrightarrow}v(n_1,n_2)|\leq 1,\\ \notag  & \phantom{\text{subject\ to}\quad} |L_{\updownarrow}v(n_1,n_2)|\leq 1, \\ \label{condat} & \qquad \forall(n_1,n_2)\in \{1,\ldots, N_1\}\times\{1,\ldots, N_2\},
	\end{align}
	where
	\begin{align} \notag
			(L_{\bullet} v)_1 (n_1,n_2) &= \frac{1}{2} \big( v_1(n_1,n_2) + v_1(n_1-1,n_2) \big),\\ \notag
			(L_{\bullet} v)_2 (n_1,n_2) &= \frac{1}{2}\big( v_2(n_1,n_2) + v_2(n_1,n_2-1) \big),\\ \notag
          (L_{\leftrightarrow} v)_1 (n_1,n_2) &= \frac{1}{4} \big(  v_1(n_1,n_2) + v_1(n_1-1,n_2) \\ \notag
          & \qquad + v_1 (n_1,n_2+1) + v_1(n_1-1,n_2+1)  \big),\\ \notag
			(L_{\leftrightarrow} v)_2(n_1,n_2) &= v_2(n_1,n_2), \\ \notag
			(L_{\updownarrow} v)_1 (n_1,n_2) &= v_1(n_1,n_2), \\ \notag
			(L_{\updownarrow}v)_2 (n_1,n_2) &= \frac{1}{4}
                                                          \big( v_2(n_1,n_2) + v_2(n_1,n_2-1)  \\
          & \label{L operators} \qquad + v_2 (n_1+1,n_2) + v_2(n_1+1,n_2-1)   \big),
	\end{align}
	{with  zero boundary condition assumptions, that is,
		for any $(n_1,n_2)\in \{1,2,\ldots,N_1\}\times\{1,2,\ldots,N_2\}$, $ v_1(0,n_2)=v_2(n_1,0)=v_1(N_1,n_2)=v_2(n_1,N_2)=0$. We also set
		$(L_{\leftrightarrow} v)_1(N_1,n_2)= (L_{\leftrightarrow}v)_2(N_1,n_2)=0,$ $(L_{\updownarrow} v)_1(n_1,N_2)= (L_{\updownarrow}v)_2(n_1,N_2)=0$.}\\
	{Given that the dual variable is bounded everywhere in the continuous definition (\ref{dis}), the insertion of three constraints using the proposed linear operators imposes boundedness on the dual variables on a grid three times denser than the pixel grid. Condat's TV is a discrete TV with isotropy properties; in other words, after rotating the image by any integer multiple of $90^\circ$, the TV value remains unchanged up to numerical precision. For other rotation angles, a better approximation can be obtained compared to isotropic TV. Additionally, this model performs well in removing noise and reconstructing edges compared to isotropic TV.\\
		The following section discusses the generalization of this idea to design a discrete second-order total generalized variation with the same rotational invariance properties.}
	\subsection{Classic Discrete TGV (Discretization of (\ref{tgv4}))}\label{sec25}
	Classic discretization of the second-order TGV \cite{bredies} is briefly explained in the following. Assume that $u\in \mathbb{R}^{N_1\times N_2}$ is a two-dimensional  $N_1\times N_2$ pixels image with homogeneous discrete Neumann boundary conditions, that is $u(n_1, N_2+1)=u(n_1,N_2), u(n_1,0)=u(n_1,1),$ for $1\leq n_1\leq N_1$ and  $u(N_1+1, n_2)=u(N_1,n_2), u(0,n_2)=u(1,n_2),$ for $1\leq n_2\leq N_2.$
	In the following, based on forward operators introduced in (\ref{dxx+}) and (\ref{dyy+}),  by enforcing a discrete Gauss--Green theorem,  backwards operators are defined as well. Consequently, all discrete operators for designing discrete TGV models {are} obtained.	\\
	It is natural to define $\mathcal{D}=(D_{x+}\  D_{y+}):\mathbb{R}^{N_1\times N_2}\rightarrow (\mathbb{R}^2)^{N_1\times N_2}$, as a discretization of the gradient operator $\nabla$ appearing in (\ref{tgv4}). Now, we define backwards difference operators
	$D_{x-}, D_{y-}:\mathbb{R}^{N_1\times N_2}\rightarrow \mathbb{R}^{N_1\times N_2}.$
	Note that from the theory of linear operators (the discrete Gauss--Green theorem),
	for any $u\in \mathbb{R}^{N_1\times N_2}$, $v\in (\mathbb{R}^2)^{N_1\times N_2}$, we have:
	\begin{equation}\label{adj1} \langle \mathcal{D}u,v\rangle=\langle u,\mathcal{D}^*v\rangle.
	\end{equation}
	On the other hand, for $u \in C_0^1(\Omega,\mathbb{R})$ and $v \in C_0^1(\Omega,\mathbb{R}^2)$, $\Omega\subset \mathbb{R}^2$ domain,  we have: %
	\begin{equation}\label{adj2}
		\langle \nabla u, v\rangle= \int_{\Omega}\nabla u\cdot v\  dxdy=-\int_{\Omega} u\cdot\text{div}\ v\ dxdy=-\langle u,\text{div}\ v\rangle,
	\end{equation}
	that is, $\nabla^*=-\text{div}.$ As $\mathcal{D}$ is a discrete approximation of $\nabla$, we can define the divergence operator on $ (\mathbb{R}^2)^{N_1\times N_2}$ by
	$\text{div}=-\mathcal{D}^*$. From (\ref{adj1}), we get
	\begin{equation}\label{adj3}
		-\text{div}\ v=\mathcal{D}^*v=-D_{x-}v_1-D_{y-}v_2,
	\end{equation}
	where for $w\in (\mathbb{R})^{N_1\times N_2}$, the backwards operator $D_{x-}$ in the $x$-direction is given as 
        \begin{multline*}
          (D_{x-}w)(n_1,n_2)=w(n_1,n_2)-w(n_1-1,n_2),\\  (n_1,n_2) \in A\setminus (\{1,N_1\}\times\{1,\ldots, N_2\}).
        \end{multline*}
	With homogeneous Neumann boundary conditions and the definition of the adjoint operator we get
        \begin{align*} 
          (D_{x-}w)(1,n_2) &= w(1,n_2),\\ \notag (D_{x-}w)(N_1,n_2) &= -w(N_1-1,n_2), \\ n_2&= 1,\ldots, N_2.
        \end{align*}
	Similarly, the backwards operator $D_{y-}$ in the $y$-direction is given as:
        \begin{multline*}
          (D_{y-}w)(n_1,n_2)=w(n_1,n_2)-w(n_1,n_2-1),\\ (n_1,n_2)\in A\setminus (\{1, \ldots N_1\}\times\{1, N_2\}),
        \end{multline*}
	and
        \begin{align*}
          (D_{y-}w)(n_1,1)&=w(n_1,1),\\ \notag (D_{y-}w)(n_1,N_2)&=-w(n_1,N_2-1),\\ n_1&=1,\ldots, N_1.
        \end{align*}
	The classic discretization of TGV is the discretization of the continuous version (\ref{tgv4}) as follows:
	\begin{align} \notag
			&\text{TGV}_{\alpha}^2(u) =\min_{w=(w_1,w_2)\in(\mathbb{R}^2)^{N_1\times N_2}} \\ \notag & \quad\ \  \alpha_1\sum_{n_1,n_2}\sqrt{
                                                                                                            \begin{aligned}
                                         & ((D_{x+}u)(n_1,n_2)-w_1(n_1,n_2))^2 \\ & \quad +((D_{y+}u)(n_1,n_2)-w_2(n_1,n_2))^2
                            \end{aligned}
                            } \\ \label{tgv5}
			& \quad +\alpha_0
          \sum_{n_1,n_2}\sqrt{
                          \begin{aligned}
                            &(D_{x+}w_1)(n_1,n_2)^2+(D_{y+}w_2)(n_1,n_2)^2 \\ & \quad +\frac{1}{2}[(D_{x+}w_2)(n_1,n_2)+(D_{y+}w_1)(n_1,n_2)]^2,
                          \end{aligned}
                          }
        \end{align}
	where $\alpha_0 > 0$, $\alpha_1 > 0$.
	The previously introduced operators are sufficient to model this optimization problem.
	\subsubsection{Dual Form of the Classic Discrete Second-Order TGV}
	Here, {we present} the dual formulation of the second-order discrete TGV (\ref{tgv5}).  The operator $\text{div} $, operating on $v\in (\mathbb{R}^2)^{N_1\times N_2}$ is defined in (\ref{adj3}).
	However, we also need the discrete operator $\text{Div}$ which operates on $v=\left(\begin{array}{cc}
		v_1 & v_3 \\
		v_3 & v_2
	\end{array}\right)\in S(\mathbb{R}^4)^{N_1\times N_2},$ where $S(\mathbb{R}^4)^{N_1\times N_2}$ is the set of symmetric second-order tensor fields in $(\mathbb{R}^4)^{N_1\times N_2}.$
	From the definition of $\text{Div}$ in (\ref{divd3}), we can define
	\begin{equation}\label{div3D}
		\text{Div}: S(\mathbb{R}^4)^{N_1\times N_2}\rightarrow (\mathbb{R}^2)^{N_1\times N_2},\  \text{Div}\ v= \left(\begin{array}{c}
			D_{x-}v_1 +D_{y-}v_3 \\
			D_{x-}v_3+D_{y-}v_2
		\end{array}
		\right).
	\end{equation}
	It can now be checked that $\text{Div}$ according to (\ref{div3D}) is the negative adjoint of $\mathcal{E},$ where
	$\mathcal{E}:u\in (\mathbb{R}^2)^{N_1\times N_2}\rightarrow S(\mathbb{R}^4)^{N_1\times N_2}$ is given by  $$\mathcal{E} u= \left(\begin{array}{cc}
		D_{x+}u_1                        & \frac{1}{2}(D_{y+}u_1+D_{x+}u_2) \\
		\frac{1}{2}(D_{y+}u_1+D_{x+}u_2) & D_{y+}u_2
	\end{array}
	\right).$$
	Furthermore, from the fact that for $v\in C_c^{2}(\Omega, \text{Sym}^{2}(\mathbb{R}^d))$, $\text{div}^2 v=\text{div}(\text{Div}\ v)$, we can concatenate the operations in (\ref{adj3}) and (\ref{div3D}) to obtain a discrete second-order divergence as follows:
	\begin{align}\notag
			\text{div}^2 &: S(\mathbb{R}^4)^{N_1\times N_2}\rightarrow \mathbb{R}^{N_1\times N_2}, \\ \notag
			\text{div}^2 v &= \text{div}\ \text{Div}\ v = D_{x-}D_{x-} v_1+D_{y-}D_{y-}v_2 \\ \label{div2} & \qquad\qquad\qquad +[D_{x-}D_{y-}+D_{y-}D_{x-}]v_3.
	\end{align}
	For the adjoint operator of $\text{div}^2$, which is the second derivative, i.e.,  $\mathcal{D}^2=\mathcal{E}\mathcal{D}$, it is easy to see that
	\begin{align}\notag
		u &\in (\mathbb{R})^{N_1\times N_2},\\\label{D2} \mathcal{D}^2 u &= \left(\begin{array}{cc}
			 D_{x+}D_{x+}u                            & \!\!\!\!\!\!\!\!\!\!\! \frac{1}{2}(D_{y+}D_{x+}u+D_{x+}D_{y+}u) \\
			\frac{1}{2}(D_{y+}D_{x+}u+D_{x+}D_{y+}u) & D_{y+}D_{y+}u
		\end{array}
		\right).
	\end{align}
	Indeed, a discrete Gauss--Green theorem as follows is valid for $u \in \mathbb{R}^{N_1 \times N_2}$ and $v \in S(\mathbb{R}^4)^{N_1 \times N_2}$:
	\begin{equation}
		\langle u,\text{div}^{2} v\rangle=-\langle \mathcal{D}u, \text{Div}\, v\rangle=
		\langle \mathcal{D}^{2}u, v\rangle.
	\end{equation}
	Consequently, it can be verified that the Fenchel--Rockafellar dual form of the second-order classic discrete TGV (\ref{tgv5}) is as follows:
	\begin{align} \notag
		\text{TGV}_{\alpha}^2(u)=\max \left\{\langle u,\text{div}^2 v\rangle: v \right. &\in S(\mathbb{R}^4)^{N_1\times N_2},  \\ \notag & |(\text{Div}\,  v)(n_1,n_2)|\leq\alpha_l\\ & \label{tgv} \left.\forall (n_1,n_2)\in A, l=0,1\right\},
	\end{align}
	where for $v \in S(\mathbb{R}^4)$, $|v| = \sqrt{v_1^2 + v_2^2 + 2v_3^2}$.
	Similar to classic discrete TV, classic discrete TGV suffers from non-invariance with respect to $90^\circ$  rotations. To compensate  this shortcoming,
	inspired by Condat's idea, our proposed model {considers} the dual formulation (instead of the primal one) of the continuous TGV model for discretization. In addition, constraints based on domain conversion operators are proposed to enhance rotational invariance properties. The new proposed model has the advantages of both Condat's discrete TV and classic discrete TGV
	simultaneously. It can attenuate staircase artifacts which is one of the important properties of classic discrete TGV as well as remove noise, reconstruct edges and admit some isotropy properties.
	\subsection{Other TGV Discretizations} \label{sec:other_tgv}
	{
		Here, we provide a brief explanation of two discrete second-order total generalized variation (TGV) models: the Shannon TGV and a TGV model for piecewise constant functions on general triangular meshes.\\
		The second-order Shannon TGV is based on Shannon interpolation \cite{Hosseini}. In this model, a discrete image \( u \in \mathbb{R}^{N_1 \times N_2} \) is interpolated, and then the definition of the continuous TGV value is applied to the obtained continuous interpolated image. Determining this continuous TGV value is generally impossible. Therefore, after discretizing this model, the second-order Shannon TGV is defined as follows:
		\begin{definition}[\cite{Hosseini}]
			Assume \( 2 \leq n \in \mathbb{N} \) and \( u \in \mathbb{R}^{N_1 \times N_2} \). For a given \( \alpha = (\alpha_0, \alpha_1) \in (\mathbb{R}^+)^2 \), the \( n \)-Shannon second-order TGV of \( u \) with the weight vector \( \alpha \) is defined by:
			\begin{multline}
				\text{TGV}^{2(\alpha)}_{\text{SH} (n)}(u) = \max_v \left\{ \frac{1}{n^2} \langle u, \text{div}_n^2 v \rangle : v \in S(\mathbb{R}^4)^{nN_1 \times nN_2}, \right. \\ \left. \|v\|_{\infty} \leq \alpha_0, \| \text{div}_n^{'} v \|_{\infty} \leq \alpha_1 \right\},
			\end{multline}
			where \( \text{div}_n^2 \) and \( \text{div}_n^{'} \) are Shannon divergence operators.
		\end{definition}
		This model involves interpolation over a grid domain more than four times the size of the given image. The variable dimensions are \( n\geq 2 \) times those of the variable dimensions in the both directions in the classic discrete TGV and our proposed upcoming model. Therefore, we expect very high numerical complexity compared to other discrete TGV models (Appendix \ref{sec:alg_comp}). More precisely, we need   
                \begin{multline*}
                  42n^2N_1N_2+15n^2N_1N_2(\log (nN_1)+\log (nN_2)) \\+2N_1N_2(\log N_1 + \log N_2)
                \end{multline*}
                floating-point operations (flops) for denoising problems through primal-dual algorithms. On the other hand, this interpolation nature can reduce the fine edge artifacts that appear in the image, whereas models defined on grids by the size of standard grid domains cannot handle such effects.

		Another discrete second-order total generalized variation (TGV) model is designed for piecewise constant functions on general triangular meshes \cite{Lukas}. Let $\Omega$ be a two-dimensional polygonal domain covered by a mesh of non-degenerate triangular cells $T$ and interior edges $E$. The discontinuous Lagrange finite element spaces of order $r\in \mathbb{N}\cup \{0\}$ (the non-negative integers) on such a mesh are defined by
		\begin{equation}
			\mathcal{D}\mathcal{G}_r(\Omega) = \{u\in L^2(\Omega): u|_T\in P_r(T)\},
		\end{equation}
		where $P_r(T)$ denotes the space of bivariate polynomials of degree at most $r$.
		This model is a discretization of non-symmetric TGV \cite{bredies}, which is defined for piecewise constant $u\in \mathcal{D}\mathcal{G}_0(\Omega):$
		\begin{multline}
                  \text{FETGV}^2_{(\alpha_0, \alpha_1)}(u) = \min_{w\in\mathcal{R}\mathcal{T}_0(\Omega)}  \alpha_1  \sum_E\|\llbracket  u \rrbracket  + h_Ew \cdot   \mu_+\|_{L^1(E)} \\ + \alpha_0\sum_T\int_T|\nabla w|_F dx 
                  +\alpha_0\sum_E \int_E \mathcal{I}_E\{|\llbracket  u \rrbracket|_2\} dS,
		\end{multline} 
		where $\mathcal{I}_E\{|\llbracket  u \rrbracket|_2\}$ denotes the linear interpolation of the pointwise 2-norm of the linear
		function $\llbracket  w \rrbracket = w_+ - w_-$ onto the space of linear functions along $E$. The interpolation points are the
		end points of $E$. $m_+, m_- \in\mathbb{R}^2$ are  the circumcenters of two adjacent triangles sharing the edge $E.$ In addition,   $\mu_+, \mu_-$ are the outward unit normal
		vectors on $E,$  $h_E=\|m_+-m_-\|_2,$ and $|\cdot|_F$ is the Frobenius norm.  Moreover, $$\mathcal{R}\mathcal{T}_0(\Omega) = \left\{v\in H(\text{div};\Omega) : v|_T\in P_0(T)^2+
		\left(\begin{array}{c}
			x\\
			y
		\end{array} \right) P_0(T)
		\right\},$$
		is the lowest-order Raviart-Thomas finite element space.\\
		This approach extends the applicability of the TGV functional to more general data structures than pixel images, particularly in finite element discretizations. However, applying this approach in regular pixel meshes for denoising yields results similar to the classic discrete TGV.}
	\section{The Proposed Discretization of the Second-Order TGV}
	To establish the new discrete TGV functional, we first discuss the required ``building blocks''.
	\begin{enumerate}
		\item{\it{Staggered grid domains of the discrete images:}}
		Staggered grids are defined. These sets are essential to define the elementary operators required by the new discrete TGV.
		\item {\it{Finite-difference operators:}}
		For a given image defined on a staggered grid domain, we explain how {to} determine the staggered grid domain for the images {resulting} from applying finite-difference and averaging operators to the given image. Consequently, discrete differentiation %
		operators on different staggered grid domains are defined, in particular
		primal first- and second-order discrete derivatives ($\mathcal{D}^{new}$, $\mathcal{E}^{new}$, $\mathcal{D}^{2new}$). Boundary conditions and grid domains of the images resulting from the new primal operators are determined.
		\item {\it{Dual difference operators:}} The adjoint operators of ($-\mathcal{D}^{new}$, $-\mathcal{E}^{new}$, $\mathcal{D}^{2new}$) which are first and second-order divergence operators ($\text{div}^{new}$, $\text{Div}^{new},$ $\text{div}^{2new}$) are derived by
		enforcing a discrete Gauss--Green theorem. %
		In particular, the associated boundary conditions and grid domains for the images resulting from the new dual operators are determined.
		\item {\it{Grid interpolation}:} In order to design a discretization for TGV with some rotational invariance properties, domain  conversion operators are defined. The staggered grid domains and boundary conditions of images obtained from these operators and their duals are studied.
		\item {\it{Proposed model and its Fenchel--Rockafellar dual:}} The new proposed discrete TGV model is formulated, and a dual formulation is provided for this model.
	\end{enumerate}
	The ``building blocks'' will be realized as follows.
	\subsection{Staggered Grid Domains of the Discrete Images} \label{sec:grids}
	We start with introducing the relevant staggered grid sets.
	\begin{definition} \label{def1} For $N_1,N_2\in\mathbb{N},$ we define the following grid sets:
		\begin{enumerate}
			\item $A_{\bullet}=\{1,\ldots,N_1\}\times \{1,\ldots,N_2\}$,
			\item $A_{\leftrightarrow}=\{\frac{1}{2},\frac{3}{2},\ldots, N_1+\frac{1}{2}\}\times\{1,\ldots, N_2\},$
			\item $A_{\updownarrow}=\{1,\ldots, N_1\}\times \{\frac{1}{2},\frac{3}{2},\ldots, N_2+\frac{1}{2}\},$
			\item $\bar{A}^x_{\bullet}=\{0,1,\ldots,N_1,N_1+1\}\times \{1,\ldots,N_2\},$
			\item $\bar{A}^y_{\bullet}=\{1,\ldots,N_1\}\times \{0,1,\ldots,N_2,N_2+1\},$
			\item $A_{\times}=\{\frac{1}{2},\frac{3}{2},\ldots, N_1+\frac{1}{2}\}\times \{\frac{1}{2},\frac{3}{2},\ldots, N_2+\frac{1}{2}\}.$
		\end{enumerate}
		See Figure \ref{grids} for an illustration. Moreover, we define the following spaces of discrete functions:\\
                \begin{align*}
                  \mathcal{U}_\bullet&=\{u:A_\bullet\rightarrow\mathbb{R}\}, & \mathcal{U}_{\leftrightarrow}&=\{u:A_{\leftrightarrow}\rightarrow\mathbb{R}\}, \\ \mathcal{U}_{\updownarrow}&=\{u:A_{\updownarrow}\rightarrow\mathbb{R}\}, &
                  \bar{\mathcal{U}}_\bullet^x&=\{u:\bar{A}_\bullet^x\rightarrow\mathbb{R}\}, \\ \bar{\mathcal{U}}_\bullet^y&=\{u:\bar{A}_\bullet^y\rightarrow\mathbb{R}\},  & \mathcal{U}_{\times}&=\{u:A_{\times}\rightarrow\mathbb{R}\}.
                \end{align*}
	\end{definition}
	\pgfdeclarelayer{background}
	\pgfsetlayers{background,main}
	\tikzstyle{vertex}=[arrow,fill=black!25,minimum size=20pt,inner sep=0pt]
	\tikzstyle{vertexx}=[circle,fill=white!25,minimum size=0pt,inner sep=0pt]
	\tikzstyle{selected vertex} = [vertexx, fill=red!24]
	\tikzstyle{edge} = [draw]
	\tikzstyle{weight} = [font=\small]
	\tikzstyle{selected edge} = [draw,line width=5pt,-,red!25]
	\tikzstyle{ignored edge} = [draw,line width=5pt,-,black!25]
	\begin{figure} \fontsize{8}{9.5}\selectfont \centering
		\begin{tabular}{ccc}
			\scalebox{0.7}{\begin{tikzpicture}[scale=.8, auto,swap]
					\foreach \pos/\name in {{(0,0)/a}, {(0,-1)/zz},{(1.5,.5)/aa},{(1.5,1.5)/aaa},
						{(1.5,.25)/aa1},{(1.5,1.2)/aaa1}, {(1.5,1)/ab}, {(1.5,1.2)/ab2}, {(1.51,1)/ab1}, {(2,1.5)/ac}, {(2,1.8)/ac2}, {(2,1.65)/ac3}, {(2,1.51)/ac1}, {(0,1)/b}, {(0,2)/c},
						{(0,3)/d}, {(1,0)/e}, {(1,-1)/ee}, {(1,1)/f}, {(1,2)/g}, {(1,3)/h}, {(2,0)/i}, {(2,1)/j},{(2.5,.65)/j1}, {(2.5,.8)/j2}, {(2,2)/k}, {(1.5,2.4)/k1}, {(1.5,2.2)/k2}, {(2,3)/l}, {(3,0)/m}, {(3,1)/n}, {(3,2)/o}, {(3,3)/p},
						{(.5,.5)/q1}, {(.5,1.5)/q2}, {(.5,2.5)/q3}, {(1.5,2.5)/q4}, {(2.5,.5)/q5}, {(2.5,1.5)/q6}, {(2.5,2.5)/q7}, {(1,1.5)/r1}, {(1,1.51)/r10}, {(1,.5)/r2}, {(1,.51)/r20}, {(2,.5)/r3}, {(.5,.2)/r4}, {(.5,1.2)/r5}, {(.5,2.2)/r6}, {(2,.51)/r30}, {(1.5,.2)/s1}, {(1.5,2.2)/s3},{(2.5,.2)/t1}, {(2.5,1.2)/t2}, {(2.5,2.2)/t3}}
					\node[vertexx] (\name) at \pos {};
					\foreach \source/ \dest /\weight in {a/b/, b/c/,c/d/,d/h/,
						h/g/, g/f/,f/e/,
						a/e/,b/f/,
						c/g/,d/h/, e/i/, i/j/, g/k/, l/k/,l/h/,
						k/j/, j/f/, i/m/, m/n/, n/o/, o/p/, p/l/, k/o/, j/n/}
					\path[edge] (\source) -- node[weight] {$\weight$} (\dest);
					\foreach \source/ \dest /\weight in {aa/aa/}
					\draw[edge](\source)--
					node[]{}(\dest) node {\color{blue}{\Large{${\bullet}$}}};
					\foreach \source/ \dest /\weight in {aaa/aaa/}
					\draw[edge](\source)--
					node[]{}(\dest) node {\color{blue}{\Large{$\bullet$}}};
					\foreach \source/ \dest /\weight in {aaa1/aaa1/}
					\foreach \source/ \dest /\weight in {r4/r4/}
					\foreach \source/ \dest /\weight in {r5/r5/}
					\foreach \source/ \dest /\weight in {r6/r6/}
					\foreach \source/ \dest /\weight in {t1/t1/}
					\foreach \source/ \dest /\weight in {t3/t3/}
					\foreach \source/ \dest /\weight in {s1/s1/}
					\foreach \source/ \dest /\weight in {t2/t2/}
					\foreach \source/ \dest /\weight in {s3/s3/}
					\foreach \source/ \dest /\weight in {q1,q2,q3,q4,q5,q6,q7}
					\draw[edge](\source)--
					node[]{}(\dest) node {\color{blue}{\Large{$\bullet$}}};
			\end{tikzpicture}}
			&
			\scalebox{0.7}{\begin{tikzpicture}[scale=.8, auto,swap]
					\foreach \pos/\name in {{(0,0)/a}, {(0,.5)/zz},{(0,2.5)/zz1},{(0,1.5)/zz2},{(1.5,.5)/aa},{(1.5,1.5)/aaa},
						{(1.5,.25)/aa1},{(1.5,1.3)/aaa1}, {(1.5,1)/ab}, {(1.5,1.3)/ab2}, {(1.51,1)/ab1}, {(2,1.5)/ac}, {(2,1.8)/ac2}, {(2,1.65)/ac3}, {(2,1.51)/ac1}, {(0,1)/b}, {(0,2)/c},
						{(0,3)/d}, {(1,0)/e}, {(1,-1)/ee}, {(1,1)/f}, {(1,2)/g}, {(1,3)/h}, {(2,0)/i}, {(2,1)/j},{(2.5,.65)/j1}, {(2.5,.8)/j2}, {(2,2)/k}, {(1.5,2.4)/k1}, {(1.5,2.2)/k2}, {(2,3)/l}, {(3,0)/m}, {(3,1)/n}, {(3,2)/o}, {(3,3)/p},{(1,.5)/p0}, {(1,1.5)/p1}, {(1,2.5)/p2}, {(2,0.5)/p00}, {(2,1.5)/p11},
						{(2,2.5)/p22}, {(3,.5)/p000}, {(3,1.5)/p111}, {(3,2.5)/p222},
						{(.5,.5)/q1}, {(.5,1.5)/q2}, {(.5,2.5)/q3}, {(1.5,2.5)/q4}, {(2.5,.5)/q5}, {(2.5,1.5)/q6}, {(2.5,2.5)/q7}, {(1,1.5)/r1}, {(1,1.51)/r10}, {(1,.5)/r2}, {(1,.51)/r20}, {(2,.5)/r3}, {(.5,.3)/r4}, {(.5,1.3)/r5}, {(1,2.7)/r6}, {(2,.51)/r30}, {(1.5,.3)/s1}, {(1.5,2.3)/s3},{(2.5,.3)/t1}, {(2.5,1.3)/t2}, {(2.5,2.3)/t3}}
					\node[vertexx] (\name) at \pos {};
					\foreach \source/ \dest /\weight in {a/b/, b/c/,c/d/,d/h/,
						h/g/, g/f/,f/e/,
						a/e/,b/f/,
						c/g/,d/h/, e/i/, i/j/, g/k/, l/k/,l/h/,
						k/j/, j/f/, i/m/, m/n/, n/o/, o/p/, p/l/, k/o/, j/n/}
					\path[edge] (\source) -- node[weight] {$\weight$} (\dest);
					\foreach \source/ \dest /\weight in {aaa1/aaa1/}
					\foreach \source/ \dest /\weight in {r4/r4/}
					\foreach \source/ \dest /\weight in {r5/r5/}
					\foreach \source/ \dest /\weight in {r6/r6/}
					\foreach \source/ \dest /\weight in {t1/t1/}
					\foreach \source/ \dest /\weight in {t3/t3/}
					\foreach \source/ \dest /\weight in {s1/s1/}
					\foreach \source/ \dest /\weight in {t2/t2/}
					\foreach \source/ \dest /\weight in {s3/s3/}
					\foreach \source/ \dest /\weight in {p0,p1,p2,zz,zz1,zz2,p00,p11,p22,p000,p111,p222}
					\draw[edge](\source)--
					node[]{}(\dest) node {\color{purple}{{\Large{$\bullet$}}}};
					\foreach \source/ \dest /\weight in {p0,p1,p2,p00,p11,p22,p000,p111,p222}
					\draw[edge](\source)--
					node[]{}(\dest) node {\color{purple}{{\tiny{$\bullet$}}}};
			\end{tikzpicture}}
			&
			\scalebox{0.7}{\begin{tikzpicture}[scale=.8, auto,swap]
					\foreach \pos/\name in {{(0,0)/a}, {(.5,3)/zz},{(1.5,3)/zz1}, {(2.5,3)/zz2},{(1.5,.5)/aa},{(1.5,1.5)/aaa},
						{(1.5,.25)/aa1},{(1.5,1.3)/aaa1}, {(1.5,1)/ab}, {(1.5,1.3)/ab2}, {(1.51,1)/ab1}, {(2,1.5)/ac}, {(2,1.8)/ac2}, {(2,1.65)/ac3}, {(2,1.51)/ac1}, {(0,1)/b}, {(0,2)/c},
						{(0,3)/d}, {(1,0)/e}, {(1,-1)/ee}, {(1,1)/f}, {(1,2)/g}, {(1,3)/h}, {(2,0)/i}, {(2,1)/j},{(2.5,.65)/j1}, {(2.5,.8)/j2}, {(2,2)/k}, {(1.5,2.4)/k1}, {(1.5,2.2)/k2}, {(2,3)/l}, {(3,0)/m}, {(3,1)/n}, {(3,2)/o}, {(3,3)/p},{(.5,0)/p0}, {(.5,1)/p1}, {(.5,2)/p2}, {(1.5,0)/p00}, {(1.5,1)/p11},
						{(1.5,2)/p22}, {(2.5,0)/p000}, {(2.5,1)/p111}, {(2.5,2)/p222},
						{(.5,.5)/q1}, {(.5,1.5)/q2}, {(.5,2.5)/q3}, {(1.5,2.5)/q4}, {(2.5,.5)/q5}, {(2.5,1.5)/q6}, {(2.5,2.5)/q7}, {(1,1.5)/r1}, {(1,1.51)/r10}, {(1,.5)/r2}, {(1,.51)/r20}, {(2,.5)/r3}, {(.5,.3)/r4}, {(.5,1.3)/r5}, {(.3,2.5)/r6}, {(2,.51)/r30}, {(1.5,.3)/s1}, {(1.5,2.3)/s3},{(2.5,.3)/t1}, {(2.5,1.3)/t2}, {(2.5,2.3)/t3}}
					\node[vertexx] (\name) at \pos {};
					\foreach \source/ \dest /\weight in {a/b/, b/c/,c/d/,d/h/,
						h/g/, g/f/,f/e/,
						a/e/,b/f/,
						c/g/,d/h/, e/i/, i/j/, g/k/, l/k/,l/h/,
						k/j/, j/f/, i/m/, m/n/, n/o/, o/p/, p/l/, k/o/, j/n/}
					\path[edge] (\source) -- node[weight] {$\weight$} (\dest);
					\foreach \source/ \dest /\weight in {aaa1/aaa1/}
					\foreach \source/ \dest /\weight in {r4/r4/}
					\foreach \source/ \dest /\weight in {r5/r5/}
					\foreach \source/ \dest /\weight in {r6/r6/}
					\foreach \source/ \dest /\weight in {t1/t1/}
					\foreach \source/ \dest /\weight in {t3/t3/}
					\foreach \source/ \dest /\weight in {s1/s1/}
					\foreach \source/ \dest /\weight in {t2/t2/}
					\foreach \source/ \dest /\weight in {s3/s3/}
					\foreach \source/ \dest /\weight in {p0,p1,p2,p00,p11,p22,p000,p111,p222,zz,zz1,zz2}
					\draw[edge](\source)--
					node[]{}(\dest) node {\color{teal}{{\Large{${\bullet}$}}}};
					\foreach \source/ \dest /\weight in {p0,p1,p2,p00,p11,p22,p000,p111,p222}
					\draw[edge](\source)--
					node[]{}(\dest) node {\color{teal}{{\tiny{$\bullet$}}}};
                                      \end{tikzpicture}} \\[-3mm]
                  			$(a)$                                      & $(b)$ & $(c)$ \\[3mm]
			\scalebox{0.7}{\begin{tikzpicture}[scale=.8, auto,swap]
					\foreach \pos/\name in {{(0,0)/a}, {(-.5,.5)/zz},{(-.5,1.5)/zz1}, {(-.5,2.5)/zz2}, {(3.5,.5)/zz3},{(3.5,1.5)/zz4},{(3.5,2.5)/zz5},{(1.5,.5)/aa},{(1.5,1.5)/aaa},
						{(1.5,.25)/aa1},{(1.5,1.2)/aaa1}, {(1.5,1)/ab}, {(1.5,1.2)/ab2}, {(1.51,1)/ab1}, {(2,1.5)/ac}, {(2,1.8)/ac2}, {(2,1.65)/ac3}, {(2,1.51)/ac1}, {(0,1)/b}, {(0,2)/c},
						{(0,3)/d}, {(1,0)/e}, {(1,-1)/ee}, {(1,1)/f}, {(1,2)/g}, {(1,3)/h}, {(2,0)/i}, {(2,1)/j},{(2.5,.65)/j1}, {(2.5,.8)/j2}, {(2,2)/k}, {(1.5,2.4)/k1}, {(1.5,2.2)/k2}, {(2,3)/l}, {(3,0)/m}, {(3,1)/n}, {(3,2)/o}, {(3,3)/p},
						{(.5,.5)/q1}, {(.5,1.5)/q2}, {(.5,2.5)/q3}, {(1.5,2.5)/q4}, {(2.5,.5)/q5}, {(2.5,1.5)/q6}, {(2.5,2.5)/q7}, {(1,1.5)/r1}, {(1,1.51)/r10}, {(1,.5)/r2}, {(1,.51)/r20}, {(2,.5)/r3}, {(.5,.2)/r4}, {(.5,1.2)/r5}, {(.5,2.2)/r6}, {(2,.51)/r30}, {(1.5,.2)/s1}, {(1.5,2.2)/s3},{(2.5,.2)/t1}, {(2.5,1.2)/t2}, {(2.5,2.2)/t3}}
					\node[vertexx] (\name) at \pos {};
					\foreach \source/ \dest /\weight in {a/b/, b/c/,c/d/,d/h/,
						h/g/, g/f/,f/e/,
						a/e/,b/f/,
						c/g/,d/h/, e/i/, i/j/, g/k/, l/k/,l/h/,
						k/j/, j/f/, i/m/, m/n/, n/o/, o/p/, p/l/, k/o/, j/n/}
					\path[edge] (\source) -- node[weight] {$\weight$} (\dest);
					\foreach \source/ \dest /\weight in {aa/aa/}
					\draw[edge](\source)--
					node[]{}(\dest) node {\color{blue}{\Large{${\bullet}$}}};
					\foreach \source/ \dest /\weight in {aaa/aaa/}
					\draw[edge](\source)--
					node[]{}(\dest) node {\color{blue}{\Large{$\bullet$}}};
					\foreach \source/ \dest /\weight in {aaa1/aaa1/}
					\foreach \source/ \dest /\weight in {r4/r4/}
					\foreach \source/ \dest /\weight in {r5/r5/}
					\foreach \source/ \dest /\weight in {r6/r6/}
					\foreach \source/ \dest /\weight in {t1/t1/}
					\foreach \source/ \dest /\weight in {t3/t3/}
					\foreach \source/ \dest /\weight in {s1/s1/}
					\foreach \source/ \dest /\weight in {t2/t2/}
					\foreach \source/ \dest /\weight in {s3/s3/}
					\foreach \source/ \dest /\weight in {q1,q2,q3,q4,q5,q6,q7,zz,zz1,zz2,zz3,zz4,zz5}
					\draw[edge](\source)--
					node[]{}(\dest) node {\color{blue}{\Large{$\bullet$}}};
			\end{tikzpicture}}
			&
			\scalebox{0.7}{\begin{tikzpicture}[scale=.8, auto,swap]
					\foreach \pos/\name in {{(0,0)/a}, {(.5,3.5)/zz},{(1.5,3.5)/zz1},{(2.5,3.5)/zz2},{(.5,-.5)/zz3}, {(1.5,-.5)/zz4},{(2.5,-.5)/zz5},{(1.5,.5)/aa},{(1.5,1.5)/aaa},
						{(1.5,.25)/aa1},{(1.5,1.2)/aaa1}, {(1.5,1)/ab}, {(1.5,1.2)/ab2}, {(1.51,1)/ab1}, {(2,1.5)/ac}, {(2,1.8)/ac2}, {(2,1.65)/ac3}, {(2,1.51)/ac1}, {(0,1)/b}, {(0,2)/c},
						{(0,3)/d}, {(1,0)/e}, {(1,-1)/ee}, {(1,1)/f}, {(1,2)/g}, {(1,3)/h}, {(2,0)/i}, {(2,1)/j},{(2.5,.65)/j1}, {(2.5,.8)/j2}, {(2,2)/k}, {(1.5,2.4)/k1}, {(1.5,2.2)/k2}, {(2,3)/l}, {(3,0)/m}, {(3,1)/n}, {(3,2)/o}, {(3,3)/p},
						{(.5,.5)/q1}, {(.5,1.5)/q2}, {(.5,2.5)/q3}, {(1.5,2.5)/q4}, {(2.5,.5)/q5}, {(2.5,1.5)/q6}, {(2.5,2.5)/q7}, {(1,1.5)/r1}, {(1,1.51)/r10}, {(1,.5)/r2}, {(1,.51)/r20}, {(2,.5)/r3}, {(.5,.2)/r4}, {(.5,1.2)/r5}, {(.5,2.2)/r6}, {(2,.51)/r30}, {(1.5,.2)/s1}, {(1.5,2.2)/s3},{(2.5,.2)/t1}, {(2.5,1.2)/t2}, {(2.5,2.2)/t3}}
					\node[vertexx] (\name) at \pos {};
					\foreach \source/ \dest /\weight in {a/b/, b/c/,c/d/,d/h/,
						h/g/, g/f/,f/e/,
						a/e/,b/f/,
						c/g/,d/h/, e/i/, i/j/, g/k/, l/k/,l/h/,
						k/j/, j/f/, i/m/, m/n/, n/o/, o/p/, p/l/, k/o/, j/n/}
					\path[edge] (\source) -- node[weight] {$\weight$} (\dest);
					\foreach \source/ \dest /\weight in {aa/aa/}
					\draw[edge](\source)--
					node[]{}(\dest) node {\color{black}{\Large{${\bullet}$}}};
					\foreach \source/ \dest /\weight in {aaa/aaa/}
					\draw[edge](\source)--
					node[]{}(\dest) node {\color{black}{\Large{$\bullet$}}};
					\foreach \source/ \dest /\weight in {aaa1/aaa1/}
					\foreach \source/ \dest /\weight in {r4/r4/}
					\foreach \source/ \dest /\weight in {r5/r5/}
					\foreach \source/ \dest /\weight in {r6/r6/}
					\foreach \source/ \dest /\weight in {t1/t1/}
					\foreach \source/ \dest /\weight in {t3/t3/}
					\foreach \source/ \dest /\weight in {s1/s1/}
					\foreach \source/ \dest /\weight in {t2/t2/}
					\foreach \source/ \dest /\weight in {s3/s3/}
					\foreach \source/ \dest /\weight in {q1,q2,q3,q4,q5,q6,q7,zz1,zz,zz2,zz3,zz4,zz5}
					\draw[edge](\source)--
					node[]{}(\dest) node {\color{black}{\Large{$\bullet$}}};
			\end{tikzpicture}}
			&
			\scalebox{0.7}{\begin{tikzpicture}[scale=.8, auto,swap]
					\foreach \pos/\name in {{(0,0)/a}, {(0,3)/zz},{(1,3)/zz1},{(2,3)/zz2}, {(3,3)/zz3},{(3,0)/zz4},{(3,1)/zz5}, {(3,2)/zz6},{(1.5,.5)/aa},{(1.5,1.5)/aaa},
						{(1.5,.25)/aa1},{(1.5,1.3)/aaa1}, {(1.5,1)/ab}, {(1.5,1.3)/ab2}, {(1.51,1)/ab1}, {(2,1.5)/ac}, {(2,1.8)/ac2}, {(2,1.65)/ac3}, {(2,1.51)/ac1}, {(0,1)/b}, {(0,2)/c},
						{(0,3)/d}, {(1,0)/e}, {(1,-1)/ee}, {(1,1)/f}, {(1,2)/g}, {(1,3)/h}, {(2,0)/i}, {(2,1)/j},{(2.5,.65)/j1}, {(2.5,.8)/j2}, {(2,2)/k}, {(1.5,2.4)/k1}, {(1.5,2.2)/k2}, {(2,3)/l}, {(3,0)/m}, {(3,1)/n}, {(3,2)/o}, {(3,3)/p},{(.5,0)/p0}, {(.5,1)/p1}, {(.5,2)/p2}, {(1.5,0)/p00}, {(1.5,1)/p11},
						{(1.5,2)/p22}, {(2.5,0)/p000}, {(2.5,1)/p111}, {(2.5,2)/p222},
						{(.5,.5)/q1}, {(.5,1.5)/q2}, {(.5,2.5)/q3}, {(1.5,2.5)/q4}, {(2.5,.5)/q5}, {(2.5,1.5)/q6}, {(2.5,2.5)/q7}, {(1,1.5)/r1}, {(1,1.51)/r10}, {(1,.5)/r2}, {(1,.51)/r20}, {(2,.5)/r3}, {(.5,.3)/r4}, {(.5,1.3)/r5}, {(.5,2.3)/r6}, {(2,.51)/r30}, {(1.5,.3)/s1}, {(1.5,2.3)/s3},{(2.5,.3)/t1}, {(2.5,1.3)/t2}, {(2.5,2.3)/t3}}
					\node[vertexx] (\name) at \pos {};
					\foreach \source/ \dest /\weight in {a/b/, b/c/,c/d/,d/h/,
						h/g/, g/f/,f/e/,
						a/e/,b/f/,
						c/g/,d/h/, e/i/, i/j/, g/k/, l/k/,l/h/,
						k/j/, j/f/, i/m/, m/n/, n/o/, o/p/, p/l/, k/o/, j/n/}
					\path[edge] (\source) -- node[weight] {$\weight$} (\dest);
					\foreach \source/ \dest /\weight in {aaa1/aaa1/}
					\foreach \source/ \dest /\weight in {r4/r4/}
					\foreach \source/ \dest /\weight in {r5/r5/}
					\foreach \source/ \dest /\weight in {r6/r6/}
					\foreach \source/ \dest /\weight in {t1/t1/}
					\foreach \source/ \dest /\weight in {t3/t3/}
					\foreach \source/ \dest /\weight in {s1/s1/}
					\foreach \source/ \dest /\weight in {t2/t2/}
					\foreach \source/ \dest /\weight in {s3/s3/}
					\foreach \source/ \dest /\weight in {a,b,c,e,f,g,i,j,k,zz,zz1,zz2,zz3,zz3,zz4,zz5,zz6}
					\draw[edge](\source)--
					node[]{}(\dest) node {\color{magenta}{{\Large{$\bullet$}}}};
					\foreach \source/ \dest /\weight in {a,b,c,e,f,g,i,j,k}
					\draw[edge](\source)--
					node[]{}(\dest) node {\color{magenta}{{\tiny{$\bullet$}}}};
			\end{tikzpicture}}                                         \\
  $(d)$ & $(e)$ & $(f)$
		\end{tabular}
		\caption{Illustration of the grid sets introduced in Definition \ref{def1}: $(a)$ ${A_{\bullet}},$ $(b)$ ${A}_{\leftrightarrow}$, $(c)$ ${A}_{\updownarrow}$, $(d)$ $\bar{A}^x_{\bullet}$, $(e)$
			$\bar{A}^y_{\bullet},$ $(f)$ $A_{\times}$.}\label{grids}
	\end{figure}
	\subsection{Finite-Difference operators} \label{sec:diff_op}
	In the following, differentiation and averaging techniques are introduced to determine the grid domain of an image. These techniques are then applied to obtain grid domains of finite-difference operators, which are essential for designing the new discrete TGV in the sequel.
	\subsubsection{Principles to Assign Suitable Grids as Domains for Discrete Images}
	Hereafter, we assume that the domain of a given discrete image $u\in\mathbb{R}^{N_1\times N_2}$ is $A_{\bullet}$, i.e., $u:A_{\bullet}\rightarrow \mathbb{R}.$\\
	The domain of discrete images, obtained from some linear operators, can be determined based on two principles: numerical approximation of derivatives and averaging via convex combinations of some objects. We present these two principles along with examples, which are essential for the sequel of the paper. The first principle allows us to find natural discrete domains for the images obtained by derivative operators such as $\mathcal{D}, \mathcal{E}, \text{div}, \text{Div}, \mathcal{D}^2,$ and $\text{div}^2$ (see the definitions of these operators in Subsection \ref{sec25}). The second principle allows us to define grid domains associated with averaging operators, such as $L_\bullet, L_\leftrightarrow,$ and $L_\updownarrow$ (these operators will be defined in Subsection \ref{sec35}). We need both principles to determine the correct grid domains, and they are explained as follows:
	\begin{tec}\label{tec1} (Numerical differentiation) The location associated with the difference of two elements in a grid is the center of the locations of these two elements. In other words, the associated grid point for $u(n_1,n_2)-u(m_1,m_2)$ is $(\frac{n_1+m_1}{2}, \frac{n_2+m_2}{2}).$
	\end{tec}
	\begin{tec}\label{tec2} (Numerical integration and averaging) The convex combination of elements in some grid domains is located at the respective convex combination of the elements' locations. In other words, the value $\sum_{j=1}^k \alpha_j u(n_1^j,n_2^j)$, where $\alpha_j\geq 0$ and $\sum_j\alpha_j=1$, is associated with the grid point $\sum_{j=1}^k \alpha_j (n_1^j,n_2^j)$.
	\end{tec}
	\begin{ex}\label{ex2} Suppose $u:A_{\bullet} \rightarrow \mathbb{R}$, then using Principle \ref{tec1}, the location of the element
		$u(2,2)-u(1,2)$ is at the point $\left(\frac{1}{2}(2+1), \frac{1}{2}(2+2)\right)=(\frac{3}{2},2).$ \\
		Moreover, assume $v:A_{\leftrightarrow}\rightarrow\mathbb{R}$, then using Principle \ref{tec2}, the location of %
		$\frac{1}{4}\left(v(1,\frac{3}{2})+v(1,\frac{5}{2})+v(2,\frac{3}{2})+v(2,\frac{5}{2})\right)$ is at the point $\left(\frac{1}{4}(1+1+2+2), \frac{1}{4}(\frac{3}{2}+\frac{5}{2}+\frac{3}{2}+\frac{5}{2})\right)=(\frac{3}{2},2)$ (see Figure \ref{fig2}).
	\end{ex}
	\begin{figure} \fontsize{8}{9.5}\selectfont \centering
		\begin{tabular}{cc}
			\scalebox{0.7}{\begin{tikzpicture}[scale=1.2, auto,swap]
					\foreach \pos/\name in {{(0,0)/a}, {(0,-1)/zz},{(1.5,.5)/aa},{(1.5,1.5)/aaa},
						{(1.5,.25)/aa1},{(1.5,1.3)/kk}, {(1.5,1.98)/kk2}, {(1.5,1.8)/kk3},{(1.5,1)/ab}, {(1.5,1.2)/ab2}, {(1.51,1)/ab1}, {(2,1.5)/ac}, {(2,1.8)/ac2}, {(2,1.65)/ac3}, {(2,1.51)/ac1}, {(0,1)/b}, {(0,2)/c},
						{(0,3)/d}, {(1,0)/e}, {(1,-1)/ee}, {(1,1)/f}, {(1,2)/g}, {(1,3)/h}, {(2,0)/i}, {(2,1)/j},{(2.5,.65)/j1}, {(2.5,.8)/j2}, {(2,2)/k}, {(1.5,2.3)/k1}, {(1.5,2.2)/k2}, {(2,3)/l}, {(3,0)/m}, {(3,1)/n}, {(3,2)/o}, {(3,3)/p},
						{(.5,.5)/q1}, {(.5,1.5)/q2}, {(.5,2.5)/q3}, {(1.5,2.5)/q4}, {(2.5,.5)/q5}, {(2.5,1.5)/q6}, {(2.5,2.5)/q7}, {(1,1.5)/r1}, {(1,1.51)/r10}, {(1,.5)/r2}, {(1,.51)/r20}, {(2,.5)/r3}, {(.5,.2)/r4}, {(.5,1.2)/r5}, {(.5,2.2)/r6}, {(2,.51)/r30}, {(1.5,.2)/s1}, {(1.5,2.2)/s3},{(2.5,.2)/t1}, {(2.5,1.2)/t2}, {(2.5,2.2)/t3}}
					\node[vertexx] (\name) at \pos {};
					\foreach \source/ \dest /\weight in {a/b/, b/c/,c/d/,d/h/,
						h/g/, g/f/,f/e/,
						a/e/,b/f/,
						c/g/,d/h/, e/i/, i/j/, g/k/, l/k/,l/h/,
						k/j/, j/f/, i/m/, m/n/, n/o/, o/p/, p/l/, k/o/, j/n/}
					\path[edge] (\source) -- node[weight] {$\weight$} (\dest);
					\foreach \source/ \dest /\weight in {aa/aa/}
					\draw[edge](\source)--
					node[]{}(\dest) node {\color{blue}{\Large{${\bullet}$}}};
					\foreach \source/ \dest /\weight in {aaa/aaa/}
					\draw[edge](\source)--
					node[]{}(\dest) node {\color{blue}{\Large{$\bullet$}}};
					\foreach \source/ \dest /\weight in {k1}
					\draw[edge](\source)--
					node[]{}(\dest) node {\color{blue}{\fontsize{6.5}{7.5}\selectfont{\scalebox{1.1}{$u(2,1)$}}}};
					\foreach \source/ \dest /\weight in {kk}
					\draw[edge](\source)--
					node[]{}(\dest) node {\color{blue}{\fontsize{6.5}{7.5}\selectfont{\scalebox{1.1}{$u(2,2)$}}}};
					\foreach \source/ \dest /\weight in {kk2}
					\draw[edge](\source)--
					node[]{}(\dest) node {\color{red}{\Large{$\circ$}}};
					\foreach \source/ \dest /\weight in {r4/r4/}
					\foreach \source/ \dest /\weight in {r5/r5/}
					\foreach \source/ \dest /\weight in {r6/r6/}
					\foreach \source/ \dest /\weight in {t1/t1/}
					\foreach \source/ \dest /\weight in {t3/t3/}
					\foreach \source/ \dest /\weight in {s1/s1/}
					\foreach \source/ \dest /\weight in {t2/t2/}
					\foreach \source/ \dest /\weight in {s3/s3/}
					\foreach \source/ \dest /\weight in {q1,q2,q3,q4,q5,q6,q7}
					\draw[edge](\source)--
					node[]{}(\dest) node {\color{blue}{\Large{$\bullet$}}};
			\end{tikzpicture}}
			&
			\scalebox{0.7}{\begin{tikzpicture}[scale=1.2, auto,swap]
					\foreach \pos/\name in {{(0,0)/a}, {(0,.5)/zz},{(0,2.5)/zz1},{(0,1.5)/zz2},{(1.5,.5)/aa},{(1.5,1.5)/aaa},
						{(1.5,.25)/aa1},{(1.5,1.3)/aaa1}, {(1.5,1)/ab}, {(1.5,1.3)/ab2}, {(1.51,1)/ab1}, {(2,1.5)/ac}, {(2,1.8)/ac2}, {(2,1.65)/ac3}, {(2,1.51)/ac1}, {(0,1)/b}, {(0,2)/c},
						{(0,3)/d}, {(1,0)/e}, {(1,-1)/ee}, {(1,1)/f}, {(1,2)/g}, {(1,3)/h}, {(2,0)/i}, {(2,1)/j},{(2.5,.65)/j1}, {(2.5,.8)/j2}, {(2,2)/k}, {(1.5,2.4)/k1}, {(1.5,2.2)/k2}, {(2,3)/l}, {(3,0)/m}, {(3,1)/n}, {(3,2)/o}, {(3,3)/p},{(1,.5)/p0}, {(1,1.5)/p1}, {(1,2.5)/p2}, {(2,0.5)/p00}, {(2,1.5)/p11},
						{(2,2.5)/p22}, {(3,.5)/p000}, {(3,1.5)/p111}, {(3,2.5)/p222},
						{(.5,.5)/q1}, {(.5,1.5)/q2}, {(.5,2.5)/q3}, {(1.5,2.5)/q4}, {(2.5,.5)/q5}, {(2.5,1.5)/q6}, {(2.5,2.5)/q7}, {(1,1.5)/r1}, {(1,1.51)/r10}, {(1,.5)/r2}, {(1,.51)/r20}, {(2,.5)/r3}, {(.5,.3)/r4}, {(.5,1.3)/r5}, {(1,2.7)/r6}, {(2,.51)/r30}, {(1.5,.3)/s1}, {(1.5,2.3)/s3},{(2.5,.3)/t1}, {(2.5,1.3)/t2}, {(2.5,2.3)/t3}, {(1.5,1.98)/kk2},{(.6,2.5)/kk4},{(1.5,1.8)/kk3},{(1.6,2.5)/kk5},{(.6,1.5)/kk6},{(1.6,1.5)/kk7}}
					\node[vertexx] (\name) at \pos {};
					\foreach \source/ \dest /\weight in {a/b/, b/c/,c/d/,d/h/,
						h/g/, g/f/,f/e/,
						a/e/,b/f/,
						c/g/,d/h/, e/i/, i/j/, g/k/, l/k/,l/h/,
						k/j/, j/f/, i/m/, m/n/, n/o/, o/p/, p/l/, k/o/, j/n/}
					\path[edge] (\source) -- node[weight] {$\weight$} (\dest);
					\foreach \source/ \dest /\weight in {aaa1/aaa1/}
					\foreach \source/ \dest /\weight in {r4/r4/}
					\foreach \source/ \dest /\weight in {r5/r5/}
					\foreach \source/ \dest /\weight in {r6/r6/}
					\foreach \source/ \dest /\weight in {t1/t1/}
					\foreach \source/ \dest /\weight in {t3/t3/}
					\foreach \source/ \dest /\weight in {s1/s1/}
					\foreach \source/ \dest /\weight in {t2/t2/}
					\foreach \source/ \dest /\weight in {s3/s3/}
					\foreach \source/ \dest /\weight in {kk4}
					\draw[edge](\source)--
					node[]{}(\dest) node {\color{black}{\fontsize{6.5}{7.5}\selectfont{\scalebox{1.1}{$v(\frac{3}{2},1)\,\,\,$}}}};
					\foreach \source/ \dest /\weight in {kk5}
					\draw[edge](\source)--
					node[]{}(\dest) node {\color{black}{\fontsize{6.5}{7.5}\selectfont{\scalebox{1.1}{$v(\frac{5}{2},1)\,\,\,$}}}};
					\foreach \source/ \dest /\weight in {kk6}
					\draw[edge](\source)--
					node[]{}(\dest) node {\color{black}{\fontsize{6.5}{7.5}\selectfont{\scalebox{1.1}{$v(\frac{3}{2},2)\,\,\,$}}}};
					\foreach \source/ \dest /\weight in {kk7}
					\draw[edge](\source)--
					node[]{}(\dest) node {\color{black}{\fontsize{6.5}{7.5}\selectfont{\scalebox{1.1}{$v(\frac{5}{2},2)\,\,\,$}}}};
					\foreach \source/ \dest /\weight in {kk2}
					\draw[edge](\source)--
					node[]{}(\dest) node {\color{red}{\Large{$\circ$}}};
					\foreach \source/ \dest /\weight in {p0,p1,p2,zz,zz1,zz2,p00,p11,p22,p000,p111,p222}
					\draw[edge](\source)--
					node[]{}(\dest) node {\color{black}{{\Large{${\bullet}$}}}};
					\foreach \source/ \dest /\weight in {p0,p1,p2,p00,p11,p22,p000,p111,p222}
					\draw[edge](\source)--
					node[]{}(\dest) node {\color{black}{{\tiny{$\bullet$}}}};
			\end{tikzpicture}} \\[-2em]
			$(a)$ & $(b)$
		\end{tabular}
		\caption{Illustration of Example \ref{ex2}: $(a)$ Grid domain of $u:A_{\bullet}\rightarrow\mathbb{R}.$ The location of $u(2,2)-u(1,2)$ is marked by a red circle. $(b)$ Grid domain of  $v:A_{\leftrightarrow}\rightarrow\mathbb{R}$. The location of $\frac{1}{4}\left(v(1,\frac{3}{2})+v(1,\frac{5}{2})+v(2,\frac{3}{2})+v(2,\frac{5}{2})\right)$ is marked by a red circle.}
		\label{fig2}
	\end{figure}
	\subsubsection{Grid Domains and Boundary Conditions of the Finite-Difference Operators}
	In the following, elementary  difference operators are defined over some images with special grid domains and special boundary conditions. The properties of the images obtained from such difference operators containing their domains and boundary conditions are expressed. These images and their domains are employed to define the new discrete TGV in the upcoming subsections.\\
	\begin{definition}\label{def}
		The first- and second-order gradient operators used for the new discretization
		are defined as follows:
		\begin{enumerate}
			\item
			$
			\mathcal{D}^{new}= (\mathcal{D}^{new }_{x \bullet},\  \mathcal{D}^{new}_{y \bullet}): \mathcal{U}_\bullet\rightarrow \mathcal{U}_\leftrightarrow\times\mathcal{U}_\updownarrow,
			$
			\begin{gather}\label{for}
                          \begin{aligned}
                            & (\mathcal{D}^{new}u)_1(n_1,n_2)=\mathcal{D}^{new}_{x \bullet}u(n_1,n_2)\\ & \qquad =\left\{
                              \begin{array}{l}
                                u(n_1+\frac{1}{2},n_2)-u(n_1-\frac{1}{2},n_2), \\ \qquad\qquad \frac{3}{2}\leq n_1\leq N_1-\frac{1}{2}, 1\leq n_2\leq N_2, \\
                                0,                                           \qquad\quad \text{else},
                              \end{array}
                            \right.
                          \end{aligned}
\\ \label{for2}
\begin{aligned}
  & (\mathcal{D}^{new}u)_2(n_1,n_2)=\mathcal{D}^{new}_{y \bullet}u(n_1,n_2) \\ & \qquad =\left\{
    \begin{array}{l}
      u(n_1,n_2+\frac{1}{2})-u(n_1,n_2-\frac{1}{2}), \\ \qquad\qquad 1\leq n_1\leq N_1, \frac{3}{2}\leq n_2\leq N_2-\frac{1}{2}, \\
      0,                                             \qquad\quad \text{else},
    \end{array}
  \right.
\end{aligned}
			\end{gather}
			\item $\mathcal{E}^{new}= \left(
			\begin{array}{cc}
				\mathcal{D}^{new}_{x \leftrightarrow}            & 0                                             \\
				0                                                & \mathcal{D}^{new}_{y \updownarrow}            \\
				\frac{1}{2}\mathcal{D}^{new}_{y \leftrightarrow} & \frac{1}{2}\mathcal{D}^{new}_{x \updownarrow}
			\end{array}
			\right): \mathcal{U}_\leftrightarrow \times\mathcal{U}_\updownarrow\rightarrow \bar{\mathcal{U}}_{\bullet}^x\times \bar{\mathcal{U}}_{\bullet}^y\times \mathcal{U}_\times,$
			\begin{gather}\label{d+-}
                          \begin{aligned}
                            &(\mathcal{E}^{new}v)_1(n_1,n_2)=\mathcal{D}^{new}_{x \leftrightarrow}v_1(n_1,n_2) \\ & \qquad =\left\{
                              \begin{array}{l}
                                v_1(n_1+\frac{1}{2},n_2)-v_1(n_1-\frac{1}{2},n_2), \\ \qquad\qquad 1\leq n_1\leq N_1, 1\leq n_2\leq N_2, \\
                                0,                                                 \qquad\quad \text{else},
                              \end{array}
                            \right.
                          \end{aligned}
			\\ \label{d+-1}
                        \begin{aligned}
                          & (\mathcal{E}^{new}v)_2(n_1,n_2)=\mathcal{D}^{new}_{y \updownarrow}v_2(n_1,n_2) \\ & \qquad =\left\{
                                                                                                                \begin{array}{l}
                                                                                                                  v_2(n_1,n_2+\frac{1}{2})-v_2(n_1,n_2-\frac{1}{2}), \\ \qquad\qquad 1\leq n_1\leq N_1, 1\leq n_2\leq N_2, \\
                                                                                                                  0,                                                 \qquad\quad \text{else},
                                                                                                                \end{array}
                                                                                                                \right.
                        \end{aligned}
			\end{gather}
			$(\mathcal{E}^{new}v)_3= \frac{1}{2}(\mathcal{D}^{new}_{y \leftrightarrow}v_1+ \mathcal{D}^{new}_{x \updownarrow}v_2),$ where
                        \begin{equation}
                            \begin{aligned}
                              & \mathcal{D}^{new}_{x \updownarrow} v_2 (n_1,n_2) \\ & \qquad =\left\{
                                                                                      \begin{array}{l}
                                                                                        v_2(n_1+\frac{1}{2},n_2)-v_2(n_1-\frac{1}{2},n_2),  \\ \qquad\qquad \frac{3}{2}\leq n_1\leq N_1-\frac{1}{2}, \frac{1}{2}\leq n_2\leq N_2+\frac{1}{2}, \\
                                                                                        0,                                                 \qquad\quad \text{else},
                                                                                      \end{array}
                                                                                      \right.
                            \end{aligned}
                        \end{equation}
			and
			\begin{equation}\label{for3}
                          \begin{aligned}
                            &\mathcal{D}^{new}_{y \leftrightarrow} v_1(n_1,n_2) \\  & \qquad=\left\{
                            \begin{array}{l}
                              v_1(n_1,n_2+\frac{1}{2})-v_1(n_1,n_2-\frac{1}{2}), \\ \qquad\qquad \frac{1}{2}\leq n_1\leq N_1+\frac{1}{2}, \frac{3}{2}\leq n_2\leq N_2-\frac{1}{2}, \\
                              0,                                                 \qquad\quad \text{else}.
                            \end{array}
                            \right.
                          \end{aligned}
			\end{equation}
			\item
			$\mathcal{D}^{2new}:\mathcal{U}_\bullet\rightarrow \bar{\mathcal{U}}_{\bullet}^x\times \bar{\mathcal{U}}_{\bullet}^y\times \mathcal{U}_\times, \mathcal{D}^{2new}u= \mathcal{E}^{new}\mathcal{D}^{new} u.$
		\end{enumerate}
	\end{definition}
	\subsection{Divergences, Grid Domains, and Boundary Conditions}
        \label{sec:div}
	In the sequel, we need the dual of the operators in Definition \ref{def}.
	By requiring a discrete Gauss--Green theorem, the dual operators are obtained as follows:
	\begin{enumerate}
		\item $\text{div}^{new}= -(\mathcal{D}^{new})^*: \mathcal{U}_{\leftrightarrow}\times \mathcal{U}_{\updownarrow}\rightarrow \mathcal{U}_\bullet,$
		$\text{div}^{new} (v_1,v_2)= -\mathcal{D}^{new}_{x \leftrightarrow}v_1-\mathcal{D}^{new}_{y \updownarrow} v_2$, where
		\begin{gather}\label{d+-2}
                  \begin{aligned}
                    & \mathcal{D}^{new}_{x \leftrightarrow}v_1(n_1,n_2) \\ & \quad =\left\{
                    \begin{array}{l}
                      v_1(\frac{3}{2},n_2),                              \quad\qquad\ \ \ \;  n_1=1, 1\leq n_2\leq N_2,               \\
                      v_1(n_1+\frac{1}{2},n_2)-v_1(n_1-\frac{1}{2},n_2), \\ \qquad\qquad\quad 2\leq n_1\leq N_1-1, 1\leq n_2\leq N_2, \\
                      -v_1(N_1-\frac{1}{2},n_2),                         \quad  n_1=N_1, 1\leq n_2\leq N_2,
                    \end{array}
                    \right.
                  \end{aligned}
		\\ \label{d+-3}
                \begin{aligned}
                  &\mathcal{D}^{new}_{y \updownarrow}v_2(n_1,n_2) \\ & \quad =\left\{
                  \begin{array}{l}
                    v_2(n_1,\frac{3}{2}),                              \qquad\quad\ \ \ \, 1\leq n_1\leq N_1, n_2=1,               \\
                    v_2(n_1,n_2+\frac{1}{2})-v_2(n_1,n_2-\frac{1}{2}), \\ \qquad\qquad\quad 1\leq n_1\leq N_1, 2\leq n_2\leq N_2-1, \\
                    -v_2(n_1, N_2-\frac{1}{2}),                        \quad 1\leq n_1\leq N_1, n_2=N_2.
                  \end{array}
                  \right.
                \end{aligned}
		\end{gather}
		Note that although $\mathcal{D}^{new}_{x \leftrightarrow}$ and $\mathcal{D}^{new}_{y \updownarrow}$ share the same notation as the operators defined in~\eqref{d+-} and \eqref{d+-1}, their domains of definition differ. In the following we ensure that it is always clear from the context such that there is no chance of confusion.
		\item $\text{Div}^{new}= -(\mathcal{E}^{new})^*: \bar{\mathcal{U}}_{\bullet}^x\times \bar{\mathcal{U}}_{\bullet}^y\times \mathcal{U}_\times\rightarrow\mathcal{U}_\leftrightarrow\times\mathcal{U}_\updownarrow,$ $\text{Div}^{new}v=\left(\begin{array}{c} \mathcal{D}^{new}_{x \bar{\bullet}}v_1+ \mathcal{D}^{new}_{y \times}v_3 \\
			\mathcal{D}^{new}_{y \bar{\bullet}}v_2+ \mathcal{D}^{new}_{x \times}v_3
		\end{array}
		\right),$
		where $\mathcal{D}^{new}_{x \bar{\bullet}}: \bar{\mathcal{U}}_\bullet^x \to \mathcal{U}_\leftrightarrow$ and $\mathcal{D}^{new}_{y \bar{\bullet}}: \bar{\mathcal{U}}_\bullet^x \to \mathcal{U}_{\updownarrow}$ are defined by
		\begin{gather}\label{d+-22}
                  \begin{aligned}
                    & \mathcal{D}^{new}_{x \bar{\bullet}}v_1(n_1,n_2) \\ &\quad =\left\{
                                                                           \begin{array}{l}
                                                                             -v_1(1,n_2),                                       \quad\ \ \ \, n_1=\frac{1}{2}, 1\leq n_2\leq N_2,                         \\
                                                                             v_1(n_1+\frac{1}{2},n_2)-v_1(n_1-\frac{1}{2},n_2), \\ \qquad\qquad\quad \frac{3}{2}\leq n_1\leq N_1-\frac{1}{2}, 1\leq n_2\leq N_2, \\
                                                                             v_1(N_1,n_2),                                      \qquad n_1=N_1+\frac{1}{2}, 1\leq n_2\leq N_2,         
                                                                           \end{array}
                                                                           \right.
                  \end{aligned}
\\ \label{d+-23}
\begin{aligned}
  & \mathcal{D}^{new}_{y\bar{ \bullet}}v_2(n_1,n_2) \\ & \quad =\left\{
  \begin{array}{l}
    -v_2(n_1,1),                                       \quad\ \ \  1\leq n_1\leq N_1, n_2=\frac{1}{2},                         \\
    v_2(n_1,n_2+\frac{1}{2})-v_2(n_1,n_2-\frac{1}{2}), \\ \qquad\qquad\quad 1\leq n_1\leq N_1, \frac{3}{2}\leq n_2\leq N_2-\frac{1}{2}, \\
    v_2(n_1, N_2),                                     \qquad 1\leq n_1\leq N_1, n_2=N_2+\frac{1}{2}.
  \end{array}
  \right.
\end{aligned}
		\end{gather}
		and
		\begin{gather}\label{d+-5}
                  \begin{aligned}
                    & \mathcal{D}^{new}_{x \times}v_3 (n_1,n_2) \\ &\quad  =\left\{
                    \begin{array}{l}
                      v_3(\frac{3}{2},n_2),                              \quad\qquad\ \ \  n_1= 1, \frac{1}{2}\leq n_2\leq N_2+\frac{1}{2},              \\
                      v_3(n_1+\frac{1}{2},n_2)-v_3(n_1-\frac{1}{2},n_2), \\ \qquad\qquad\qquad 2\leq n_1\leq N_1-1, \frac{1}{2}\leq n_2\leq N_2+\frac{1}{2}, \\
                      -v_3(N_1-\frac{1}{2},n_2),                         \quad n_1=N_1, \frac{1}{2}\leq n_2\leq N_2+\frac{1}{2},
                    \end{array}
                    \right.
                  \end{aligned}
\\ \label{d+-6}
\begin{aligned}
  & \mathcal{D}^{new}_{y \times}v_3(n_1,n_2) \\ & \quad =\left\{
  \begin{array}{l}
    v_3(n_1,\frac{3}{2}),                              \qquad\quad\ \ \ \, \frac{1}{2}\leq n_1\leq N_1+\frac{1}{2}, n_2=1,               \\
    v_3(n_1,n_2+\frac{1}{2})-v_3(n_1,n_2-\frac{1}{2}), \\ \qquad\qquad\qquad  \frac{1}{2}\leq n_1\leq N_1+\frac{1}{2}, 2\leq n_2\leq N_2-1, \\
    -v_3(n_1, N_2-\frac{1}{2}),                        \quad \frac{1}{2}\leq n_1\leq N_1+\frac{1}{2}, n_2=N_2,
  \end{array}
  \right.
\end{aligned}
		\end{gather}
		\item $\text{div}^{2new}=(\mathcal{D}^{2new})^*:  \bar{\mathcal{U}}_{\bullet}^x\times \bar{\mathcal{U}}_{\bullet}^y\times \mathcal{U}_\times\rightarrow \mathcal{U}_\bullet$, $\text{div}^{2new}v =  \text{div}^{new}\text{Div}^{new}v.$
	\end{enumerate}
	Indeed, one can verify that with the above definitions, a discrete Gauss--Green theorem as follows holds for $u \in \mathcal{U}_\bullet$ and $v \in \bar{\mathcal{U}}^x_\bullet \times \bar{\mathcal{U}}^y_\bullet \times \mathcal{U}_\times$:
	\begin{equation}
		\langle u,\text{div}^{2new} v\rangle=-\langle \mathcal{D}^{new}u, \text{Div}^{new} v\rangle=
		\langle \mathcal{D}^{2new}u, v\rangle.
	\end{equation}
	\subsection{Grid Interpolation and Conversion Operators}\label{sec35}
	Assume $w=\left(\begin{array}{c} w_1 \\ w_2
	\end{array}\right)\in{\mathcal{U}}_\leftrightarrow\times{\mathcal{U}}_\updownarrow$ and $v=(v_1,v_2,v_3)\in \bar{\mathcal{U}}_\bullet^x\times \bar{\mathcal{U}}_\bullet^y \times{\mathcal{U}}_{\times}$.
	We define linear grid domain  conversion operators
	$L_\bullet:{\mathcal{U}}_\leftrightarrow\times{\mathcal{U}}_\updownarrow\rightarrow {\mathcal{U}}_\bullet\times {\mathcal{U}}_\bullet$, $L_\leftrightarrow:{\mathcal{U}}_\leftrightarrow\times{\mathcal{U}}_\updownarrow\rightarrow {\mathcal{U}}_\leftrightarrow\times{\mathcal{U}}_\leftrightarrow$ and
	$L_\updownarrow:{\mathcal{U}}_\leftrightarrow\times{\mathcal{U}}_\updownarrow\rightarrow {\mathcal{U}}_\updownarrow\times {\mathcal{U}}_\updownarrow$ as follows:
	\begin{align} \notag
				(L_\bullet w)_1(n_1,n_2) &=
				\tfrac{1}{2}(w_1(n_1+\tfrac{1}{2},n_2)+w_1(n_1-\tfrac{1}{2},n_2)),  \\ \notag & \qquad\qquad\qquad\quad 1\leq n_1\leq N_1, 1\leq n_2\leq N_2,
				\\ \notag
				(L_\bullet w)_2(n_1,n_2) &=
				\tfrac{1}{2}(w_2(n_1, n_2+\tfrac{1}{2})+w_2(n_1, n_2-\tfrac{1}{2})), \\ \notag & \qquad\qquad\qquad\quad  1\leq n_1\leq N_1, 1\leq n_2\leq N_2, \\ \notag
				(L_\leftrightarrow w)_1(n_1,n_2)
				&=w_1(n_1,n_2), \ \  \tfrac{1}{2}\leq n_1\leq N_1+\tfrac{1}{2}, 1\leq n_2\leq N_2,                                \\ \notag
				\displaystyle	(L_\leftrightarrow w)_2(n_1,n_2)
				\\ \notag & \hspace{-4.5em} =\left\{\begin{array}{l}
					\frac{1}{4}( w_2(1,n_2-\frac{1}{2})+w_2(1,n_2+\frac{1}{2})), \\ \qquad\qquad\qquad\qquad\ \  n_1=\frac{1}{2}, 1\leq n_2\leq N_2, \\
					\begin{array}{l}
						\frac{1}{4}(w_2(n_1-\frac{1}{2},n_2-\frac{1}{2})+w_2(n_1-\frac{1}{2},n_2+\frac{1}{2}) \\ \qquad+w_2(n_1+\frac{1}{2},n_2-\frac{1}{2})+w_2(n_1+\frac{1}{2}, n_2+\frac{1}{2})),
					\end{array}
					\\ \qquad\qquad\qquad\qquad\ \ 
					\frac{3}{2}\leq n_1\leq N_2-\frac{1}{2}, 1\leq n_2\leq N_2,                                            \\
					\frac{1}{4}(w_2(N_1, n_2-\frac{1}{2})+w_2(N_1,n_2+\frac{1}{2})), \\ \qquad\qquad\qquad\qquad\ \  n_1=N_1+\frac{1}{2},
					1\leq n_2\leq N_2,
				\end{array}
				\right.
				\\ \notag 
				\displaystyle	(L_\updownarrow w)_1(n_1,n_2) \\ \notag &\hspace{-4.5em} = \left\{\begin{array}{l}
					\frac{1}{4}(w_1(n_1-\frac{1}{2},1)+w_1(n_1+\frac{1}{2},1)),   \\ \qquad\qquad\qquad\qquad\ \  1\leq n_1\leq N_1, n_2=\frac{1}{2},                         \\
					\begin{array}{l}
						\frac{1}{4}(w_1(n_1-\frac{1}{2},n_2-\frac{1}{2})+w_1(n_1-\frac{1}{2},n_2+\frac{1}{2}) \\ \qquad +w_1(n_1+\frac{1}{2},n_2-\frac{1}{2})+w_1(n_1+\frac{1}{2}, n_2+\frac{1}{2})),
					\end{array}
					\\ \qquad\qquad\qquad\qquad\ \  1\leq n_1\leq N_1, \frac{3}{2}\leq n_2\leq N_2-\frac{1}{2}, \\
					\frac{1}{4}(w_1(n_1-\frac{1}{2}, N_2)+w_1(n_1+\frac{1}{2},N_2)), \\ \qquad\qquad\qquad\qquad\ \  1\leq n_1\leq N_1, n_2=N_2+\frac{1}{2},
				\end{array}
				\right.                                                                                                         \\ \label{operatorL2}
				(L_\updownarrow w)_2(n_1,n_2)&=w_2(n_1,n_2), \ \  1\leq n_1\leq N_1, \tfrac{1}{2}\leq n_2\leq N_2+\tfrac{1}{2}.
	\end{align}
	Note that at some points in the above definitions, we extended the respective grid in a natural manner and assumed zero values in order to adhere to Principle~\ref{tec2}.
	Moreover, the linear operator $L_\bullet: \bar{\mathcal{U}}_\bullet^x\times \bar{\mathcal{U}}_\bullet^y \times{\mathcal{U}}_{\times}\rightarrow {\mathcal{U}}_\bullet\times {\mathcal{U}}_\bullet\times {\mathcal{U}}_\bullet$ is defined by
	\begin{align} \notag
			(L_\bullet v)_1(n_1,n_2)& =v_1(n_1,n_2), \quad 1\leq n_1\leq N_1, 1\leq n_2\leq N_2, \\ \notag
			(L_\bullet v)_2(n_1,n_2)&=v_2(n_1,n_2), \quad 1\leq n_1\leq N_1, 1\leq n_2\leq N_2, \\ \notag
			(L_\bullet v)_3(n_1,n_2)&=\tfrac{1}{4}(v_3(n_1-\tfrac{1}{2},n_2-\tfrac{1}{2})
          +v_3(n_1-\tfrac{1}{2},n_2+\tfrac{1}{2}) \\ \notag
          & \quad +v_3(n_1+\tfrac{1}{2},n_2-\tfrac{1}{2})
			+v_3(n_1+\tfrac{1}{2},n_2+\tfrac{1}{2})),
			\\ 
          & \qquad\qquad\qquad\qquad \ \  1\leq n_1\leq N_1,  1\leq n_2\leq N_2.
	\end{align}
	Again, in the sequel, the domain of $L_\bullet$ is made clear such that this operator cannot be confused with the previously-defined operator with the same notation.
	In summary, $L_\bullet$ are operators that convert the grid domain of each component of a given image to an image on $A_\bullet$, $L_\leftrightarrow$ is a similar grid domain conversion operator to $A_\leftrightarrow$ and $L_\updownarrow$ is a similar grid domain conversion operator to $A_\updownarrow$.
	\subsection{Proposed Model and its Fenchel--Rockafellar Dual}
        \label{sec:model}
	\subsubsection{Formulation of the Discrete TGV Functional}
	Now, we propose the following discretization of TGV of order 2 according to~\eqref{tgv1}:
	\begin{align} \notag
          &\text{TGV}_{\alpha}^{2 (new)}(u) \\ \notag
          & \quad =\max_{v,w}\left\{\langle u,s\rangle:\ v \in \bar{\mathcal{U}}_\bullet^x \times \bar{\mathcal{U}}_\bullet^y \times \mathcal{U}_\bullet,  w \in \mathcal{U}_{\leftrightarrow} \times \mathcal{U}_{\updownarrow},\right. \\ \notag & \qquad\qquad \left.\|L_\bullet v\|_{\infty}\leq \alpha_0,\ \|L_{\bullet}w\|_{\infty}\leq \alpha_1,\ \|L_{\leftrightarrow}w\|_{\infty}\leq \alpha_1, \right. \\ \label{nrt}  & \qquad\qquad \left.\|L_{\updownarrow}w\|_{\infty}\leq \alpha_1, \  w=\text
            {Div}^{new} v,\ s=\text{div}^{new} w \right\},
	\end{align}
	where
	\begin{align} \notag
          &\|L_\star w\|_\infty \\ \notag & = \max\{\sqrt{(L_{\star}w)_1(n_1,n_2)^2+(L_{\star}w)_2(n_1,n_2)^2}:  (n_1,n_2)\in A_\star\}, \\ \notag & \qquad\qquad\qquad\qquad\qquad\qquad\qquad\qquad\qquad \star=\bullet,  \leftrightarrow, \updownarrow, \\ \notag
          & \|L_\bullet v\|_\infty \\  & = \max\left\{\sqrt{
                                               \begin{array}{l}
                                     (L_{\bullet}v)_1(n_1,n_2)^2 \\ \quad +(L_{\bullet}v)_2(n_1,n_2)^2 \\ \qquad\qquad + 2 (L_{\bullet}v)_3(n_1,n_2)^2
                                   \end{array}
                                   }:  (n_1,n_2)\in A_\bullet\right\}. \label{eq:newtgv_norms}
	\end{align}
	In the formulation of classic discrete TGV (\ref{tgv}), two constraints are used: $\|v\|_\infty\leq\alpha_0$ and $\|\text{Div}\,v\|_\infty\leq\alpha_1,$ whereas in the new proposed discrete TGV (\ref{nrt}), we use four constraints: $\|L_\bullet v\|_\infty\leq\alpha_0, \|L_\bullet (\text{Div}^{new}v)\|_\infty\leq\alpha_1,
	\|L_\leftrightarrow (\text{Div}^{new} v)\|_\infty\leq\alpha_1$ and $\|L_\updownarrow (\text{Div}^{new}v)\|_\infty\leq\alpha_1.$ In other words, instead of the boundedness of the vector field $\text{Div}\,v$ and the tensor field $v$, we impose boundedness for their converted versions.
	
	Let us revisit the operators $\text{div}$ and $\text{div}^2$ of Subsection~\ref{sec25} in view of Principles~\ref{tec1} and~\ref{tec2}. Then, $\text{div}: (\mathbb{R}^2)^{N_1 \times N_2} \to \mathbb{R}^{N_1 \times N_2}$ can be interpreted as $\text{div}: \mathcal{U}_\leftrightarrow \times \mathcal{U}_\updownarrow \to \mathcal{U}_\bullet$ if we identify $\mathcal{U}_\bullet \,\hat{=}\, \mathbb{R}^{N_1 \times N_2}$, $\mathcal{U}_\leftrightarrow \,\hat{=}\, \bigl(\mathbb{R}^{1 \times N_2} + (-\tfrac12, 0) \bigr) \times \bigl(\mathbb{R}^{N_1 \times N_2} + (\tfrac12,0) \bigr)$, $\mathcal{U}_\updownarrow \,\hat{=}\, \bigl( \mathbb{R}^{N_1 \times 1}+ (0,-\tfrac12) \bigr) \times \bigl( \mathbb{R}^{N_1 \times N_2} + (0,\tfrac12) \bigr)$, where $+$ denotes an index shift and the entries that do not correspond to $\mathbb{R}^{N_1 \times N_2}$ are filled with zero. Likewise, $\text{Div}: S(\mathbb{R}^4)^{N_1 \times N_2} \to (\mathbb{R}^2)^{N_1 \times N_2}$ can be interpreted as $\text{Div}: \bar{\mathcal{U}}_\bullet^x \times \bar{\mathcal{U}}_\bullet^y \times \mathcal{U}_\bullet \to \mathcal{U}_\leftrightarrow \times \mathcal{U}_\updownarrow$ by the identification $(v_1,v_2,v_3) \,\hat{=}\, \left(\begin{array}{cc}
		v_1 & v_3 \\
		v_3 & v_2
	\end{array}
	\right)$ and $\bar{\mathcal{U}}_\bullet^x \,\hat{=}\, \bigl(\mathbb{R}^{1 \times N_2} + (-1,0)\bigr) \times \mathbb{R}^{N_1 \times N_2} \times \bigl(\mathbb{R}^{1 \times N_2} + (N_1,0)\bigr)$, $\bar{\mathcal{U}}_\bullet^y \,\hat{=}\, \bigl(\mathbb{R}^{N_1 \times 1} + (0,-1) \bigr) \times \mathbb{R}^{N_1 \times N_2} \times \bigl(\mathbb{R}^{N_1 \times 1} + (0,N_2)\bigr)$, $\mathcal{U}_\times \,\hat{=}\, (\mathbb{R}^{N_1 \times N_2} \times \bigl(\mathbb{R}^{1 \times (N_2 +1)} + (-\tfrac12,-\tfrac12) \bigr) \times \bigl(\mathbb{R}^{N_1 \times 1} + (\tfrac12, -\tfrac12)\bigr) \times \bigl(\mathbb{R}^{N_1 \times N_2} + (\tfrac12,\tfrac12) \bigr)$, where again $+$ denotes an index shift and the entries that do not correspond to $\mathbb{R}^{N_1 \times N_2}$ are filled with zero.
	With the index shifts introduced in the above identifications, the constraints in the classic discrete second-order TGV according to~\eqref{tgv} correspond to:
	\begin{align}\notag
		\sqrt{v_1(n_1,n_2)^2+v_2(n_1,n_2)^2+2v_3(n_1+\tfrac12,n_2 + \tfrac12)^2} &\leq \alpha_0, \\ \label{loc3} \sqrt{(\text{Div}\,v)_1(n_1+\tfrac12,n_2)^2+(\text{Div}\,v)_2(n_1,n_2+\tfrac12)^2}&\leq\alpha_1,
	\end{align}
	for $v \in \bar{\mathcal{U}}_\bullet^x \times \bar{\mathcal{U}}_\bullet^y \times \mathcal{U}_\bullet$ and %
	$(n_1,n_2) \in A_\bullet$.
	The constraint on the first line of \eqref{loc3} refers to the square root of the sum of two elements on the common grid $A_\bullet$ (subset of both $\bar{A}^x_\bullet$ and $\bar{A}^y_\bullet$), whereas the third element corresponds to the shifted grid $A_\times$. Likewise, in the constraint on the second line, two elements of the different grids $A_\leftrightarrow$ and $A_\updownarrow$ %
	are added. In other words, %
	for both constraints, there exists an inconsistency in terms of the grid point evaluation.
	
	As it is explained before, if $u\in\mathcal{U}_\bullet,$ then $\mathcal{D}^{new}u, \text{Div}^{new}v\in \mathcal{U}_\leftrightarrow \times\mathcal{U}_\updownarrow, \mathcal{D}^2u,v\in \bar{\mathcal{U}}_\bullet^x\times\bar{\mathcal{U}}_\bullet^y\times\mathcal{U}_\times$.
	Assume $w_1=(\text{Div}^{new} v)_1, w_2=(\text{Div}^{new} v)_2$. Then, the constraints in optimization problem (\ref{nrt}) can be expressed by
	\begin{align} \notag
			\displaystyle	\sqrt{(L_\star w)_1(n_1,n_2)^2+(L_\star w)_2(n_1,n_2)^2}&\leq\alpha_1, \\ \notag &  \hspace*{-1.5em} (n_1,n_2)\in A_\star, \star=\bullet, \leftrightarrow, \updownarrow, \\ \label{cons} 
			\displaystyle	\sqrt{
          \begin{aligned}
            &(L_\bullet v)_1(n_1,n_2)^2+(L_\bullet v)_2(n_1,n_2)^2 \\ &\qquad\qquad\qquad\ \  + 2(L_\bullet v)_3(n_1,n_2)^2
          \end{aligned}
          }&\leq\alpha_0, \quad   (n_1,n_2)\in A_\bullet,	\end{align}
	where $(L_\star w)_1, (L_\star w)_2 \in A_\star, (L_\bullet v)_1, (L_\bullet v)_2, (L_\bullet v)_3\in A_\bullet.$ Therefore, the norm definitions in (\ref{cons}) admit grid domain consistency.
	Moreover, another difference of the classic discrete TGV in comparison with the new proposed one is the rotational invariance with respect to $90^{\circ}$ rotation. This property is discussed in the next section.
	\begin{remark}
		Note that other choices of interpolation operators in \eqref{nrt} are possible. %
		Generally, we can define
		operators converting elements of $v \in \bar{\mathcal{U}}_\bullet^x \times \bar{\mathcal{U}}_\bullet^y \times \mathcal{U}_\bullet$ and $w \in \mathcal{U}_{\leftrightarrow} \times \mathcal{U}_{\updownarrow}$ to respective versions on the grids $A_\bullet, A_\leftrightarrow, A_\updownarrow, A_\times$, resulting in 8 operators, denoted by $L_\bullet$, $L_\leftrightarrow$, $L_\updownarrow, L_\times$ with a slight abuse of notation. In principle, any non-empty subset of these operators applied to $v$ and $w$ would also be possible in~\eqref{nrt}.
		As it can be observed, (\ref{nrt})
		only contains the operator
		$L_{\bullet}$  for $v$ and the three conversion operators $L_{\leftrightarrow}, L_{\updownarrow}$ and $L_{\bullet}$ for $w$. %
		As $v$ contains two components in the extended center grids $\bar{A}_\bullet^x$, $\bar{A}_\bullet^y$ which are supersets of $A_\bullet$, and one component in the corner grid $A_\times$, we preferred to use only the conversion operator $L_{\bullet}$.  %
		For the variable $w$, as the components belong to $\mathcal{U}_{\leftrightarrow}$ and $\mathcal{U}_{\updownarrow}$, we use %
		the conversion operators $L_{\leftrightarrow}$ and $ L_{\updownarrow} $ as well as the natural conversion operator $L_{\bullet}$. %
		This selection %
		realizes a good trade-off between
		accuracy and efficiency. Also, as we will see in Section~\ref{invariance}, the choice of conversion operators allows us to prove a $90^\circ$ rotational invariance property. In contrast, the classic discrete TGV is not invariant with respect to $90^\circ$ rotations.
	\end{remark}
	\subsubsection{Fenchel--Rockafellar Dual of the Proposed Model}
	In this subsection we find a dual form for the proposed new discrete TGV (\ref{nrt}). We need such formulation to employ a primal-dual algorithm to solve corresponding denoising and inverse problems. Define
	\begin{align} \notag
		K=\{&(v,w,s) \in (\bar{\mathcal{U}}_\bullet^x \times \bar{\mathcal{U}}_\bullet^y \times \mathcal{U}_\bullet) \times (\mathcal{U}_\leftrightarrow\times \mathcal{U}_\updownarrow) \times \mathcal{U}_\bullet : \\ \notag &|L_\bullet v(n_1,n_2)|\leq\alpha_0 \ \forall (n_1,n_2)\in A_\bullet, \\ \notag &|L_{\star}w(n_1,n_2)|\leq \alpha_1, \star=\bullet, \leftrightarrow, \updownarrow \ \forall (n_1,n_2)\in A_\star, &
		\\ \label{K} & \qquad\qquad\qquad\quad\  w=-\text{Div}^{new} v, s=-\text{div}^{new}w  \}.
	\end{align}
	Then, obviously
	\begin{equation}\label{opp}
		\text{TGV}_{\alpha}^{2 (new)}(u)= \max_{(v,w,s)} \ \langle u,s\rangle -I_{K}(v,w,s),
	\end{equation}
	where
	$I_{K}(t)=\left\{
	\begin{array}{ll}
		0,      & t\in K,    \\
		\infty, & t\notin K.
	\end{array}
	\right.$
	We aim at finding a dual definition of $\text{TGV}_{\alpha}^{2 (new)}$. For this purpose, the adjoint operators of $L_{\bullet}, L_{\leftrightarrow}$ and $L_{\updownarrow}$ are calculated in the following.

	Let $w_{\bullet}=\left(\begin{array}{c} w_{\bullet}^1 \\ w_{\bullet}^2\\
	\end{array}\right)\in {\mathcal{U}}_\bullet\times {\mathcal{U}}_\bullet$,  $w_{\leftrightarrow}=\left(\begin{array}{c} w_{\leftrightarrow}^1 \\ w_{\leftrightarrow}^2
	\end{array}\right)\in {\mathcal{U}}_\leftrightarrow\times {\mathcal{U}}_\leftrightarrow$, $w_{\updownarrow}=\left(\begin{array}{c} w_{\updownarrow}^1 \\ w_{\updownarrow}^2
	\end{array}\right)\in {\mathcal{U}}_\updownarrow\times {\mathcal{U}}_\updownarrow$.
	Then, we have the following adjoint operators $L_\bullet^*: {\mathcal{U}}_\bullet\times {\mathcal{U}}_\bullet\rightarrow {\mathcal{U}}_\leftrightarrow\times {\mathcal{U}}_\updownarrow$, $L_\leftrightarrow^*: {\mathcal{U}}_\leftrightarrow\times {\mathcal{U}}_\leftrightarrow\rightarrow {\mathcal{U}}_\leftrightarrow\times {\mathcal{U}}_\updownarrow$ and
	$L_\updownarrow^*: {\mathcal{U}}_\updownarrow\times {\mathcal{U}}_\updownarrow\rightarrow {\mathcal{U}}_\leftrightarrow\times {\mathcal{U}}_\updownarrow$:
	\begin{align} \notag
          (L_{\bullet}^*w_{\bullet})_1(n_1,n_2) &= \left\{
					\begin{array}{l}
						\frac{1}{2}w_{\bullet}^1(1,n_2),                                                    \ \  n_1=\frac{1}{2}, 1\leq n_2\leq N_2,                         \\
						\frac{1}{2}(w_{\bullet}^1(n_1+\frac{1}{2},n_2)+w_{\bullet}^1(n_1-\frac{1}{2},n_2)), \\ \qquad\quad \frac{3}{2}\leq n_1\leq N_1-\frac{1}{2}, 1\leq n_2\leq N_2, \\
						\frac{1}{2}w_{\bullet}^1(N_1,n_2),                                                  \\ \qquad\qquad\ \  n_1=N_1+\frac{1}{2}, 1\leq n_2\leq N_2,                     \\
					\end{array}
					\right.
					\\ \notag
          (L_{\bullet}^*w_{\bullet})_2(n_1,n_2) &= \left\{
					\begin{array}{l}
						\frac{1}{2}w_{\bullet}^2(n_1,1),                                                    \ \  1\leq n_1\leq N_1, n_2=\frac{1}{2},                         \\
						\frac{1}{2}(w_{\bullet}^2(n_1,n_2+\frac{1}{2})+w_{\bullet}^2(n_1,n_2-\frac{1}{2})), \\ \qquad\quad 1\leq n_1\leq N_1, \frac{3}{2}\leq n_2\leq N_2-\frac{1}{2}, \\
						\frac{1}{2}w_{\bullet}^2(n_1, N_2),                                                 \\ \qquad\qquad\ \ 1\leq n_1\leq N_1, n_2=N_2+\frac{1}{2}.
					\end{array}
					\right.
          \\ \notag
					(L_{\leftrightarrow}^*w_{\leftrightarrow})_1(n_1,n_2) &= w_{\leftrightarrow}^1(n_1,n_2), \\ \notag & \qquad\qquad \tfrac{1}{2}\leq n_1\leq N_1+\tfrac{1}{2}, 1\leq n_2\leq N_2,
          \\ \notag
					\displaystyle	(L_{\leftrightarrow}^*w_{\leftrightarrow})_2(n_1,n_2) \\ \notag & \hspace{-5em}= \left\{\begin{array}{l}
						\frac{1}{4}(w_{\leftrightarrow}^2(n_1+\frac{1}{2},1)+w_{\leftrightarrow}^2(n_1-\frac{1}{2},1)), \\ \qquad\qquad\qquad\qquad 1\leq n_1\leq N_1, n_2=\frac{1}{2},                         \\
						\frac{1}{4}(w_{\leftrightarrow}^2(n_1+\frac{1}{2},n_2-\frac{1}{2})+w_{\leftrightarrow}^2(n_1-\frac{1}{2},n_2-\frac{1}{2})
						\\ \qquad
						+w_{\leftrightarrow}^2(n_1+\frac{1}{2},n_2+\frac{1}{2})+w_{\leftrightarrow}^2(n_1-\frac{1}{2}, n_2+\frac{1}{2})), \\  \qquad\qquad\qquad\qquad 1\leq n_1\leq N_1, \frac{3}{2}\leq n_2\leq N_2-\frac{1}{2}, \\
						\frac{1}{4}(w_{\leftrightarrow}^2(n_1+\frac{1}{2},N_2)+w_{\leftrightarrow}^2(n_1-\frac{1}{2},N_2)),             \\ \qquad\qquad\qquad\qquad 1\leq n_1\leq N_1, n_2=N_2+\frac{1}{2},
					\end{array}
					\right. \\ \notag
					\displaystyle	(L_{\updownarrow}^*w_{\updownarrow})_1(n_1,n_2)
					\\ \notag &\hspace{-5em}=\left\{\begin{array}{l}
						\frac{1}{4}(w_{\updownarrow}^1(1,n_2-\frac{1}{2})+w_{\updownarrow}^1(1,n_2+\frac{1}{2})),                   \\ \qquad\qquad\qquad n_1=\frac{1}{2}, 1\leq n_2\leq N_2,                         \\
						\frac{1}{4}(w_{\updownarrow}^1(n_1+\frac{1}{2},n_2-\frac{1}{2})+w_{\updownarrow}^1(n_1-\frac{1}{2},n_2-\frac{1}{2})
						\\ \qquad
						+w_{\updownarrow}^1(n_1+\frac{1}{2},n_2+\frac{1}{2})+w_{\updownarrow}^1(n_1-\frac{1}{2}, n_2+\frac{1}{2})), \\ \qquad\qquad\qquad\qquad \frac{3}{2}\leq n_1\leq N_1-\frac{1}{2}, 1\leq n_2\leq N_2, \\
						\frac{1}{4}(w_{\updownarrow}^1(N_1,n_2-\frac{1}{2})+w_{\updownarrow}^1(N_1,n_2+\frac{1}{2})),               \\ \label{adjop} \qquad\qquad\qquad\qquad n_1=N_1+\frac{1}{2}, 1\leq n_2\leq N_2.
					\end{array}
					\right. \\ \notag
					(L_{\updownarrow}^*w_{\updownarrow})_2(n_1,n_2) &= w_{\updownarrow}^2(n_1,n_2), \\ & \qquad\qquad 1\leq n_1\leq N_1, \tfrac{1}{2}\leq n_2\leq N_2+\tfrac{1}{2}.
	\end{align}
	Moreover, for $v_{\bullet}=\left(\begin{array}{c} v_{\bullet}^1 \\ v_{\bullet}^2\\ v_{\bullet}^3
	\end{array}\right)\in {\mathcal{U}}_\bullet\times {\mathcal{U}}_\bullet\times {\mathcal{U}}_\bullet$, %
	the adjoint operator $L_\bullet^*: {\mathcal{U}}_\bullet\times {\mathcal{U}}_\bullet\times {\mathcal{U}}_\bullet\rightarrow \bar{\mathcal{U}}_\bullet^x\times \bar{\mathcal{U}}_\bullet^y\times {\mathcal{U}}_{\times}$ reads:
	\begin{align} \notag
					(L_\bullet^*v_\bullet)_1(n_1,n_2)&=\left\{\begin{aligned}
						&v_\bullet^1(n_1,n_2), && 1\leq n_1\leq N_1, 1\leq n_2\leq N_2, \\
						&0,                   && n_1=0, N_1+1, 1\leq n_2\leq N_2,
					\end{aligned}
					\right.
					\\ \notag
					(L_\bullet^*v_\bullet)_2(n_1,n_2)&=\left\{\begin{aligned}
						& v_\bullet^2(n_1,n_2), && 1\leq n_1\leq N_1, 1\leq n_2\leq N_2, \\
						& 0,                   && 1\leq n_2\leq N_2, n_2=0, N_2+1,
					\end{aligned}
					\right.
					\\ \label{adjj}
					(L_\bullet^*v_\bullet)_3(n_1,n_2) \\ \notag & \hspace{-5em}=\left\{
					\begin{array}{l}
						\frac{1}{4} v_\bullet^3(1,1),                                                                   \qquad n_1=n_2=\frac{1}{2},                                                                  \\
						\frac{1}{4} (v_\bullet^3(1,n_2-\frac{1}{2})+v_\bullet^3(1,n_2+\frac{1}{2})),                   \\ \qquad\qquad\qquad\quad n_1=\frac{1}{2}, \frac{3}{2} \leq  n_2 \leq N_2-\frac{1}{2},                          \\
						\frac{1}{4} v_\bullet^3(1,N_2),                                                                 \qquad n_1=\frac{1}{2}, n_2=N_2+\frac{1}{2},                                                 \\
						\frac{1}{4}((v_\bullet^3(n_1-\frac{1}{2},1)
						+(v_\bullet^3(n_1+\frac{1}{2},1)),                                                               \\ \qquad\qquad\qquad\quad \frac{3}{2} \leq n_1 \leq N_1-\frac{1}{2}, n_2 = \frac{1}{2},                         \\
						\frac{1}{4}((v_\bullet^3(n_1-\frac{1}{2},n_2-\frac{1}{2})
						+(v_\bullet^3(n_1+\frac{1}{2},n_2-\frac{1}{2})                                                                                                                                          \\
						\qquad
						+(v_\bullet^3(n_1-\frac{1}{2},n_2+\frac{1}{2}) 	+(v_\bullet^3(n_1+\frac{1}{2},n_2+\frac{1}{2})), \\ \qquad\qquad\qquad\quad \frac{3}{2} \leq n_1 \leq N_1-\frac{1}{2}, \frac{3}{2} \leq n_2 \leq N_2-\frac{1}{2}, \\
						\frac{1}{4} (v_\bullet^3(n_1-\frac{1}{2},N_2)+v_\bullet^3(n_1+\frac{1}{2},N_2)),                \\ \qquad\qquad\qquad\quad \frac{3}{2} \leq  n_1 \leq N_1-\frac{1}{2}, n_2= N_2 + \frac{1}{2},                   \\
						\frac{1}{4} v_\bullet^3(N_1,1),                                                                 \qquad n_1=N_1+\frac{1}{2}, n_2=\frac{1}{2},                                                 \\
						\frac{1}{4} (v_\bullet^3(N_1,n_2-\frac{1}{2})+v_\bullet^3(N_1,n_2+\frac{1}{2})),                \\ \qquad\qquad\qquad\quad n_1=N_1 + \frac{1}{2}, \frac{3}{2} \leq  n_2 \leq N_2-\frac{1}{2},                    \\
						\frac{1}{4} v_\bullet^3(N_1,N_2),                                                               \qquad n_1=N_1+\frac{1}{2}, n_2=N_2+\frac{1}{2}.
					\end{array} 
					\right.  
	\end{align}
	Now, we define the operator $L$  and the corresponding dual $L^*$
	via the following operator matrices:
	\begin{align} \notag
          L&=\left(\begin{array}{ccc}
					L_\bullet        & 0                   & 0 \\
					0                & L_{\bullet}         & 0 \\
					0                & L_{\leftrightarrow} & 0 \\
					0                & L_{\updownarrow}    & 0 \\
					0                & \text{div}^{new}    & I \\
					\text{Div}^{new} & I                   & 0\end{array}\right), \\ \label{Lop}
				L^*&=\left(\begin{array}{cccccc}
					L_\bullet^* & 0             & 0                     & 0                  & 0                  & -\mathcal{E}^{new} \\
					0           & L_{\bullet}^* & L_{\leftrightarrow}^* & L_{\updownarrow}^* & -\mathcal{D}^{new} & I                  \\
					0           & 0             & 0                     & 0                  & I                  & 0                  \\
				\end{array}\right). %
	\end{align}
	In the first column of $L$, we have $L_\bullet: \bar{\mathcal{U}}_\bullet^x \times \bar{\mathcal{U}}_\bullet^y \times \mathcal{U}_\times \to \mathcal{U}_\bullet \times \mathcal{U}_\bullet \times \mathcal{U}_\bullet$ while in the second column, $L_\bullet: \mathcal{U}_{\leftrightarrow} \times \mathcal{U}_{\updownarrow} \to \mathcal{U}_\bullet \times \mathcal{U}_\bullet$. Consequently, in the first row of $L^*$, we have $L_\bullet^*: \mathcal{U}_\bullet \times \mathcal{U}_\bullet \times \mathcal{U}_\bullet \to \bar{\mathcal{U}}_\bullet^x \times \bar{\mathcal{U}}_\bullet^y \times \mathcal{U}_\times$ while in the second row, $L_\bullet^*: \mathcal{U}_\bullet \times \mathcal{U}_\bullet \to \mathcal{U}_{\leftrightarrow} \times \mathcal{U}_{\updownarrow}$. Analogous considerations apply to the operators $\text{div}^{new},$ and $\text{Div}^{new}$ in $L$.
	\begin{remark}
		The boundary conditions associated with the adjoint of the above conversion operators are dictated by the adjointness requirement:
		$$\langle L^*_\star w_\star, w\rangle=\langle w_\star,L_\star w\rangle, \star=\bullet,\leftrightarrow,\updownarrow, \quad \langle L^*_\bullet v_\bullet, v\rangle=\langle v_\bullet,L_\bullet v\rangle,$$
		for each $w \in \mathcal{U}_{\leftrightarrow} \times \mathcal{U}_{\updownarrow}$, $w_\star \in \mathcal{U}_\star \times \mathcal{U}_\star$, $\star = \bullet, \leftrightarrow, \updownarrow$, $v \in \bar{\mathcal{U}}_\bullet^x \times \bar{\mathcal{U}}_\bullet^y \times \mathcal{U}_\times$ and $v_\bullet \in \mathcal{U}_\bullet \times \mathcal{U}_\bullet \times \mathcal{U}_\bullet$.
	\end{remark}
	We employ the following theorem in order to find a dual form of %
	the proposed regularization term \cite{Convex}.
	\begin{theorem}[Fenchel Duality Theorem]\label{t1}
		Assume $X, Y$ are real Banach spaces, $f:X\rightarrow {]{-\infty, +\infty}]}$ and $g:Y\rightarrow {]{-\infty, +\infty}]}$ are proper, convex and lower-semicontinuous functions and $A:X \rightarrow Y$ is a linear continuous operator. If there exists $x_0\in X$ such that $f(x_0)<\infty$ and $g$ is continuous at $Ax_0$, then
		\begin{equation}{\label{FDT}}
			\sup_{x \in X} \ -f(x)-g(Ax) =\min_{y^* \in Y^*} \ g^*(y^*)+f^*(-A^*y^*),
		\end{equation}
		where $f^*$ and $g^*$ are the Fenchel conjugates of $f$ and $g$, respectively.
	\end{theorem}
	\begin{theorem}\label{theorem3}
		The functional $TGV_{\alpha}^{2(new)}$ according to (\ref{nrt})
		satisfies:
		\begin{align} \notag
			& \displaystyle \text{TGV}_{\alpha}^{2 (new)}(u)=\min_{v_\bullet, w_{\bullet}, w_{\leftrightarrow}, w_{\updownarrow}, \omega} \ \alpha_0\| v_\bullet\|_1+ \alpha_1\|{w}_{\bullet}\|_1  \\[-0.5em]  & \notag \qquad\qquad\qquad\qquad\qquad\qquad\qquad + \alpha_1\| {w}_{\leftrightarrow}\|_1+
				\alpha_1\|{w}_{\updownarrow}\|_1 \\[0.5em] \label{nrt2}
				& \text{subject to} \quad %
				\left\{\begin{aligned}
					\mathcal{D}^{new}u - \omega & = L_\bullet^* v_\bullet + L_\leftrightarrow^* v_\leftrightarrow + L_\updownarrow^* v_\updownarrow, \\
					\mathcal{E}^{new}           & = L_\bullet^* w_\bullet,
				\end{aligned}\right.
		\end{align}
		where $v_\bullet \in \mathcal{U}_\bullet \times \mathcal{U}_\bullet \times \mathcal{U}_\bullet$, $w_\star \in \mathcal{U}_\star \times \mathcal{U}_\star$, $\star = \bullet, \leftrightarrow, \updownarrow$ %
		and $\omega \in \mathcal{U}_\leftrightarrow \times \mathcal{U}_\updownarrow$.
	\end{theorem}
	\begin{proof}
		Consider the optimization problem (\ref{opp}). To find the Fenchel dual problem via the Fenchel duality theorem, we define, for a given $u\in\mathcal{U}_\bullet$,
		$f(v, w, s) = - \langle u,s \rangle$ for $(v,w,s) \in (\bar{\mathcal{U}}_\bullet^x \times \bar{\mathcal{U}}_\bullet^y \times \mathcal{U}_\times) \times (\mathcal{U}_{\leftrightarrow} \times \mathcal{U}_{\updownarrow}) \times \mathcal{U}_\bullet$, $g = I_{\bar K}$
		where
		\begin{align} \notag
			\bar{K}=\bigl\{
			&(v_\bullet, {w}_{\bullet}, {w}_{\leftrightarrow},	{w}_{\updownarrow}, \bar u,	\omega) \\ \notag &\in \mathcal{U}_\bullet^3 \times \mathcal{U}_\bullet^2 \times \mathcal{U}_\leftrightarrow^2 \times \mathcal{U}_\updownarrow^2 \times \mathcal{U}_\bullet \times (\mathcal{U}_\leftrightarrow \times \mathcal{U}_\updownarrow)
			: \\ \notag & \quad |v_\bullet(n_1,n_2)|\leq\alpha_0 \ \forall (n_1,n_2)\in A_\bullet, \\ \notag
			& \quad |{w}_{\star} (n_1,n_2)|\leq \alpha_1, \star=\bullet, \leftrightarrow, \updownarrow \ \forall (n_1,n_2)\in A_\star, \\ \label{kbar} &\qquad\qquad\qquad\qquad\qquad\qquad \bar u =0, \omega=0 \bigr\},
		\end{align}
		and $A = L$ is defined in (\ref{Lop}). Obviously, $\bar{K}$ is non-empty, convex and closed, and therefore, $g$ is proper, convex and lower-semicontinuous. Furthermore, $f$ is convex and continuous. Thus, the assumptions of the Fenchel duality theorem hold.
		
		Now, the optimization problem corresponding to the left hand-side of (\ref{FDT}) corresponds to $\text{TGV}^{2(new)}_\alpha(u)$. To find the right hand-side, i.e., the dual minimization problem, the Fenchel conjugates $f^*, g^*$ are needed, whereas the adjoint operator $A^*=L^*$ is already given in (\ref{Lop}). Thus, consider
                \begin{multline*}
                  f(v,w,s) =-\langle u,s\rangle \\ =\sup_{v^\star, w^\star, s^\star} \left\langle (v,w,s) , (v^\star, w^\star, s^\star) \right\rangle %
                  - I_{\{(0,0,-u)\}}(v^\star, w^\star, s^\star),
                \end{multline*}
		therefore, $f^*(v,w,s)=%
		I_{\{(0,0,-u)\}}(v,w,s)$.
		Since the $1$-norm is the dual of the $\infty$-norm, we get:
                \begin{align*}
                  g^*(
                  v_\bullet,
                  {w}_{\bullet},
                  {w}_{\leftrightarrow},
                  {w}_{\updownarrow},
                  \bar u,
                  \omega
                  ) &= \alpha_0\| v_\bullet\|_1+ \alpha_1\|{w}_{\bullet}\|_1 \\ & \quad +\alpha_1\| {w}_{\leftrightarrow}\|_1+
                  \alpha_1\|{w}_{\updownarrow}\|_1,
                \end{align*}
		where
		\begin{align} \notag
			\|{w}_{\star}\|_1&=\sum_{(n_1,n_2)\in A_\star}|{w}_{\star}(n_1,n_2)|,\\ \notag |{w}_{\star}(n_1,n_2)|&=\sqrt{\sum_{i=1}^2 {w}_{\star}^i(n_1,n_2)^2}, \qquad \star= \bullet, \leftrightarrow, \updownarrow, \\ \notag
				\|{v_\bullet}\|_1&=\sum_{(n_1,n_2)\in A_\bullet}|v_\bullet(n_1,n_2)|,\\ |{v_\bullet}(n_1,n_2)|&=\sqrt{ v_{\bullet}^1(n_1,n_2)^2+v_{\bullet}^2(n_1,n_2)^2+2v_{\bullet}^3(n_1,n_2)^2}.\label{eq:fenchel_dual_1norms}
		\end{align}
		From the Fenchel duality theorem, we get:		
		\begin{align*}
			\text{TGV}_{\alpha}^{2 (new)}(u)= & \min_{v_\bullet, w_{\bullet}, w_{\leftrightarrow}, w_{\updownarrow}, \bar u, \omega} \ \alpha_0\| v_\bullet\|_1+ \alpha_1\|{w}_{\bullet}\|_1 \\ & \qquad\qquad\qquad + \alpha_1\| {w}_{\leftrightarrow}\|_1+
			\alpha_1\|{w}_{\updownarrow}\|_1                                                                                                                                                                                         \\[0.25em]
			\text{subject to} &
			\quad L^*(v_\bullet, w_\bullet, v_\leftrightarrow, v_\updownarrow, \bar u, \omega) = (0, 0, u),
		\end{align*}
		which is equivalent to
		\begin{align*}
			\text{TGV}_{\alpha}^{2 (new)}(u)= & \min_{v_\bullet, w_{\bullet}, w_{\leftrightarrow}, w_{\updownarrow}, \bar u, \omega} \ \alpha_0\| v_\bullet\|_1+ \alpha_1\|{w}_{\bullet}\|_1 \\ & \qquad\qquad\qquad + \alpha_1\| {w}_{\leftrightarrow}\|_1+
			\alpha_1\|{w}_{\updownarrow}\|_1                                                                                                                                                                                         \\[0.25em]
			\text{subject to} &
			\quad
			\left\{\begin{aligned}
				\mathcal{E}^{new} \omega         & = L_\bullet^*v_\bullet,                                                                                 \\
				\mathcal{D}^{new}\bar u - \omega & = L_\bullet^* v_\bullet + L_\leftrightarrow^*v_\leftrightarrow + L_\updownarrow^* v_\updownarrow, \\ u &= \bar u,
			\end{aligned}\right.
		\end{align*}
		leading to the desired statement.
	\end{proof}
	\begin{remark}
		Consider the classic discrete version of TGV in (\ref{tgv}):
		\begin{equation}\label{sim}
			\displaystyle \text{TGV}^2_{\alpha}(u)=\min_{\omega\in (\mathbb{R}^2)^{N_1\times N_2}} \alpha_1\|\mathcal{D}u-\omega\|_1+\alpha_0\|\mathcal{E}\omega\|_1
		\end{equation}
		which can be rewritten to
		\begin{equation}\label{sim2}
			\begin{array}{l}
				\displaystyle \text{TGV}^2_{\alpha}(u)=\min_{w,\omega\in (\mathbb{R}^2)^{N_1\times N_2}} \alpha_1\|w\|_1+\alpha_0\|v\|_1 \\[0.25em]
				\text{subject to}\quad \left\{
				\begin{aligned}
					\mathcal{D}u-\omega & =w, \\
					\mathcal{E}\omega   & =v.
				\end{aligned}                      \right.
			\end{array}
		\end{equation}
		Compare this to the proposed discrete TGV in~\eqref{nrt2}:
		\begin{align} \notag
				&\text{TGV}_{\alpha}^{2(new)}(u)=\min_{v_\bullet,w_\bullet,w_\leftrightarrow,w_\updownarrow,\omega}\alpha_1(\|w_\updownarrow\|_1+\|w_\leftrightarrow\|_1 +\|w_\bullet\|_1) \\ \notag & \qquad\qquad\qquad\qquad\qquad\qquad  +\alpha_0 \|v_\bullet\|_1 \\[0.25em] \label{sim3}
				&\text{subject to}\quad
				\left\{
				\begin{aligned}
					\mathcal{D}^{new}u-\omega & =L_{\bullet}^*w_{\bullet} + L_{\leftrightarrow}^*w_{\leftrightarrow}+L_{\updownarrow}^*w_{\updownarrow}, \\ 
					\mathcal{E}^{new}\omega   & =L^*_\bullet v_\bullet.
				\end{aligned}\right.
                \end{align}
		It can be seen that in the classic discrete TGV, the aim is the minimization of an energy function containing $\alpha_1\|w\|_1$ and $\alpha_0 \|v\|_1,$ where $w$ and $v$ are discrete gradient fields and symmetric matrix fields, respectively.
		For the newly defined discrete TGV (\ref{sim3}), instead of $w$, three gradient fields, $w_\bullet, w_\leftrightarrow, w_\updownarrow$, are used and penalized with the sum of their respective $1$-norms.
		Likewise, $v$ in the classic discrete TGV is replaced by $v_\bullet$ in the proposed TGV. Moreover, instead of the constraints
		$\mathcal{D}u - \omega=w$ and $\mathcal{E}\omega=v$,
		we have the different constraints
		\begin{equation} \label{OO}
			\mathcal{D}^{new}u - \omega=L_{\bullet}^*w_{\bullet}+L_{\leftrightarrow}^*w_{\leftrightarrow}+L_{\updownarrow}^*w_{\updownarrow} \quad \text{and}\quad \mathcal{E}^{new}\omega=L^*_\bullet v_\bullet.
		\end{equation}
		To interpret (\ref{OO}),
		observe that $\mathcal{D}^{new} u - \omega$ is decomposed into $w_\bullet, w_\leftrightarrow, w_\updownarrow$ which live on the grids $A_\bullet, A_\leftrightarrow, A_\updownarrow$, respectively, and are interpolated, as a consequence of Principle~\ref{tec2}, to be compatible with $\mathcal{D}^{new} - \omega$ whose components live on the grid $A_\leftrightarrow$ and $A_\updownarrow$, respectively. Minimizing over the sum of the $1$-norms of $w_\bullet$, $w_\leftrightarrow$ and $w_\updownarrow$ thus asks for an optimal decomposition of $\mathcal{D}^{new} u - \omega$ into vector fields on different grids in terms of the $1$-norm, similar (but not identical) to an infimal convolution.
		Similarly, $\mathcal{E}^{new} \omega$ can be interpreted to be converted to the grid $A_\bullet$ by choosing a $v_\bullet$ which is interpolated to be compatible to $\mathcal{E}^{new} \omega$ and whose $1$-norm is also penalized.
		
		Note that it is not possible to rewrite this problem to a simplified version without the variables $w_\star, \star=\bullet,\leftrightarrow, \updownarrow$, whereas for the classic discrete TGV, we can simplify \eqref{sim2} to~\eqref{sim}.
	\end{remark}
	
	\subsection{Alternative Choices and Extensions}\label{subsec:extensions}
		In the following, we aim at commenting and discussing the choices made for the design of the proposed discrete TGV functional in~\eqref{nrt} as well as possible alternatives and extensions. Recall that our construction depends, on the one hand, on the staggered grid domains in Definition~\ref{def1}, but also on the implementation of Principles~\ref{tec1} and~\ref{tec2}.
	
	{%
		Note that Principle~\ref{tec1} implies an interplay between the used grids and the finite-difference approximation scheme. In this regard, one could employ alternative discrete differentiation schemes that span more than two grid points and have higher accuracy than the employed two-point schemes which are of first order. A central difference scheme would, for instance, be a second-order scheme for which the need of staggered grids does not arise. Using this scheme for a discrete TV and, consequently, for a discrete TGV would consequently be possible without further effort. However, such a choice usually leads to checkerboard-type artifacts in associated variational problems, see Appendix \ref{sec:central_diff} for an example involving central-differences TV. We expect the same effects when designing a discrete TGV with central differences. Also, according to our experience, considering even more grid points in a finite-difference approximation does not mitigate this effect. For this reason, the employed two-point schemes already appear to be a reasonable choice that cannot easily be improved without introducing undesired effects.}
	
	{%
		Nevertheless, an alternative approach to formulate new discrete gradients is considering directional derivatives in more than two directions. Indeed, applying finite differences on staggered grids has been used earlier to improve isotropy for Mumford--Shah-type regularizers and related higher-order models in earlier works (see, for example, \cite{M1,M2,M3}). In \cite{M1} and \cite{M2}, appropriate weights for the
		finite differences in several directions were derived by comparing penalties with ideal (digital) lines. The idea was
		used in \cite{Stor1} to obtain a more isotropic finite-difference discretization of (first-order) TV, both in two and three dimensions. This discretization of TV implicitly uses staggered grids, horizontal, vertical, and diagonal differences and is $90^{\circ}$ rotationally invariant. %
		The discrete total variation introduced in \cite{Stor1} reads as
		\begin{multline}
			\text{TV}_{1}(u) = \sum_{n_1=1}^{N_1}\sum_{n_2=1}^{N_2} \sum_{s=1}^S {\omega_s |u(n_1,n_2)-u((n_1,n_2)+a_s)|},\\ a_1,\ldots, a_S\in\mathbb{Z}^2 \setminus \{0\},
		\end{multline}
		where $\omega_s > 0$, $s=1,\ldots, S$ are suitable weights. However, such a functional is anisotropic in the sense that a continuous counterpart would not be rotationally invariant.
		In \cite{Hosseini2}, an isotropic version is considered whose continuous counterpart is rotationally invariant. It reads as
		\begin{multline}
			\text{TV}_{2}(u) = \sum_{n_1=1}^{N_1}\sum_{n_2=1}^{N_2} \sqrt{\sum_{s=1}^4\bar{\omega}_s \bigl(u(n_1,n_2)-u((n_1,n_2)+a_s)\bigr)^2},\\ a_1,\ldots, a_4\in\mathbb{Z}^2 \setminus \{0\},
		\end{multline}
		where $a_s$ and $\tilde \omega_s > 0$, $s=1,\ldots, 4$ are horizontal, vertical, diagonal vectors and suitable weights, respectively. This idea could potentially be combined with Condat's discrete TV model as well as our proposed second-order TGV model. Such a combination would, however, be a topic for future research.} %

	{  {Additionally, inspired by our model, Bogensperger et al. \cite{learnedTGV} suggest that different types of operators can be defined for $L.$
			They proposed a learning model to find the optimal interpolation operators that admit the least squares error, using our model as the initial foundation due to its simplicity and straightforward nature. However, it is important to note that their model does not satisfy the rotational invariance property. Developing a learning algorithm to achieve the best rotationally invariant model could be an interesting new research direction.}\\
		Finally, let us note that it does not pose great challenges to extend the framework to color or multichannel images. In principle, one can proceed as outlined in \cite{brediestgv12} to obtain a classic TV and second-order TGV discretization using discrete vector and tensor fields as well as respective Euclidean and Frobenius norms. An extension of Condat's discrete total variation according to~\eqref{condat} to $C$ channels would arise from considering $u \in (\mathbb{R}^{C})^{N_1 \times N_2}$, $v \in (\mathbb{R}^{2 \times C})^{N_1 \times N_2}$, constructing $\mathcal{D}: (\mathbb{R}^{C})^{N_1 \times N_2} \to (\mathbb{R}^{2 \times C})^{N_1 \times N_2}$ as channelwise application of the discrete gradient operator and taking the Euclidean scalar product for $v_1(n_1,n_2), v_2(n_1,n_2) \in \mathbb{R}^C$. Further, $L_\bullet$, $L_\leftrightarrow$ and $L_\updownarrow$ would also have to be considered channelwise and the norm in the constraints $|L_\star v(n_1,n_2)| \leq 1$, $\star = \bullet, \leftrightarrow, \updownarrow$ would have to be the Frobenius norm for $\mathbb{R}^{2 \times C}$ matrices. An extension of the proposed TGV model according to~\eqref{nrt} to $C$ channels is then analogous. This means that the discrete function spaces have to be replaced by versions that map into $\mathbb{R}^C$, such as $\mathcal{U}_\bullet = \{u: A_\bullet \to \mathbb{R}^C\}$ and so on. The operators $L_\star$, $\star = \bullet, \leftrightarrow, \updownarrow,$ $\text{div}^{new},$ and $\text{Div}^{new},$ then have to operate channelwise, while in the norms according to~\eqref{eq:newtgv_norms}, the square terms have to be replaced by the squared Euclidean norm in $\mathbb{R}^C$. In this case, Theorem~\ref{theorem3} holds analogously with a representation~\eqref{nrt2} where $\mathcal{D}^{new}$, $\mathcal{E}^{new}$ and $L_\star^*$, $\star = \bullet, \leftrightarrow, \updownarrow$ operate channelwise and norms according to~\eqref{eq:fenchel_dual_1norms}, where the squared terms have to be replaced by the squared Euclidean norm in $\mathbb{R}^C$. As the subsequent results and algorithms also extend according to these straightforward principles, we will limit the discussion to single-channel images.}
	
	\section{A Basic Invariance Property}
	\label{invariance}
	In the following, we prove that the new proposed discrete TGV is $90^\circ$ rotationally invariant, which can be expected as a consequence of the proposed building blocks. However, as mentioned before, this property is not fulfilled for the classic discrete second-order TGV. For this purpose, denote by $A_\bullet^\perp$, $A_\leftrightarrow^\perp$, $A_\updownarrow^\perp$, $\bar{A}_\bullet^{x\perp}$, $\bar{A}_\bullet^{y\perp}$ and $A_\times^\perp$ the grids according to Definition~\ref{def1} with $N_1$ and $N_2$ interchanged. The resulting function spaces will also be marked with a ${}^\perp$, i.e., $\mathcal{U}_\bullet^\perp$ for the functions on $A_\bullet^\perp$ and so on. Since there will be no chance of confusion, we will use the same notation for the operators on the function spaces involving original and rotated grids such as $\mathcal{D}^{new}$, $\mathcal{E}^{new}$, $L_\bullet^*$, $L_\leftrightarrow^*$, $L_\updownarrow^*$ etc.
	\begin{theorem}[$90^\circ$ isotropy]\label{isot}
		 Let $u\in \mathcal{U}_\bullet$ and let $\mathcal{R}u\in \mathcal{U}_\bullet^\perp$ be the $90^\circ$ rotated image, that is,
		$u$ applied to $\mathcal{R}: \mathcal{U}_\bullet \to \mathcal{U}_\bullet^\perp$,
		the $90^\circ$ rotation operator mapping $u \in U_\bullet$ to
                \begin{multline*}
                  \mathcal{R}u(n_1,n_2)=u(n_2,N_2-n_1+1), \\
                  n_1=1,2,\ldots, N_2, \  n_2=1,2,\ldots, N_1.
                \end{multline*}
		Then, $\text{TGV}_{\alpha}^{2 (new)}(\mathcal{R}u)=\text{TGV}_{\alpha}^{2 (new)}(u)$ where the functional has to be understood in the respective
		domain.
	\end{theorem}
	\begin{proof}
		First note that the reparametrization $(n_1, n_2) \mapsto (n_2, N_2+1-n_1)$ is a bijection when mapping as follows: $A_\bullet \to A_\bullet^\perp$, $A_\leftrightarrow \to A_\updownarrow^\perp$, $A_\updownarrow \to A_\leftrightarrow^\perp$, $\bar A_\bullet^x \to \bar A_\bullet^{y\perp}$, $\bar A_\bullet^y \to \bar A_\bullet^{x\perp}$ and $A_\times \to A_\times^\perp$. Consequently, $\mathcal{R}$ considered as a map between $\mathcal{U}_\bullet \to \mathcal{U}_\bullet^\perp$ is a linear isomorphism. The same applies to the analogous versions, i.e., $\mathcal{R}: \mathcal{U}_\leftrightarrow \to \mathcal{U}_\updownarrow^\perp$ etc. With these preparations, we see, for instance, for $u \in \mathcal{U}_\bullet$ that
		\begin{align} \notag
				(\mathcal{D}_{x\bullet}^{new}\mathcal{R}u)&(n_1,n_2) =
				(\mathcal{R}u)(n_1+\tfrac12, n_2) - (\mathcal{R}u)(n_1 - \tfrac12, n_2) \\ \notag &= u(n_2, N_2 + 1 - n_1 - \tfrac12) - u(n_2, N_2 + 1 - n_1 + \tfrac12) \\ \notag &=
				-(\mathcal{D}_{y\bullet}^{new}u)(n_2, N_2 + 1 - n_1) \\ &=
				- (\mathcal{R}\mathcal{D}_{y\bullet}^{new}u)(n_1, n_2),
			\end{align}
		for $\tfrac32 \leq n_1 \leq N_2 - \tfrac12$, $1 \leq n_2 \leq N_1$.
		Also considering the boundary cases, it is easy to conclude that
		$\mathcal{D}_{x\bullet}^{new}\mathcal{R} = -\mathcal{R}\mathcal{D}_{y\bullet}^{new}$. Likewise,
			\begin{align} \notag
				(\mathcal{D}_{y\bullet}^{new}\mathcal{R}u)&(n_1,n_2) =
				(\mathcal{R}u)(n_1, n_2+\tfrac12) - (\mathcal{R}u)(n_1, n_2 - \tfrac12) \\ \notag &= u(n_2 + \tfrac12, N_2 + 1 - n_1) - u(n_2 - \tfrac12, N_2 + 1 - n_1) \\ \notag &=
				(\mathcal{D}_{x\bullet}^{new}u)(n_2, N_2 + 1 - n_1) \\ &=
				(\mathcal{R}\mathcal{D}_{x\bullet}^{new}u)(n_1, n_2),
			\end{align}
		for $1 \leq n_1 \leq N_2$, $\tfrac32 \leq n_2 \leq N_1 - \tfrac12$,
		allowing us to conclude analogously that $\mathcal{D}_{y\bullet}^{new}\mathcal{R} = \mathcal{R}\mathcal{D}_{x\bullet}^{new}$. Thus, with the linear isomorphism $\bar{\mathcal{R}}: \mathcal{U}_{\leftrightarrow} \times \mathcal{U}_{\updownarrow} \to \mathcal{U}_{\leftrightarrow}^\perp \times \mathcal{U}_{\updownarrow}^\perp$ according to $\bar{\mathcal{R}}(w_1,w_2) = (-\mathcal{R}w_2, \mathcal{R}w_1)$, we have $\mathcal{D}^{new}\mathcal{R} = \bar{\mathcal{R}}\mathcal{D}^{new}$.
		
		Considerations that are completely analogous also lead to the identities $\mathcal{D}_{x\star}^{new} \mathcal{R} = - \mathcal{R}\mathcal{D}^{new}_{y\star}$,
		$\mathcal{D}^{new}_{y\star} \mathcal{R} = \mathcal{R}\mathcal{D}^{new}_{x\star}$ for $u \in \mathcal{U}_{\star}$, $\star = \leftrightarrow, \updownarrow$.
		Thus, for $w \in \mathcal{U}_\leftrightarrow \times \mathcal{U}_\updownarrow$ we see that
		\begin{equation}
			\mathcal{E}^{new}\bar{\mathcal{R}}w = \left(
			\begin{array}{c}
				\mathcal{R} \mathcal{D}^{new}_{y\updownarrow} w_2    \\
				\mathcal{R} \mathcal{D}^{new}_{x\leftrightarrow} w_1 \\
				- \tfrac12 (\mathcal{R} \mathcal{D}^{new}_{x\updownarrow} w_2 + \mathcal{R} \mathcal{D}^{new}_{y\leftrightarrow} w_1)
			\end{array}
			\right) = \tilde{\mathcal{R}} \mathcal{E}^{new}w,
		\end{equation}
		where the linear isomorphism $\tilde{\mathcal{R}}: \bar{\mathcal{U}}_\bullet^x \times \bar{\mathcal{U}}_\bullet^y \times {\mathcal{U}}_\times \to \bar{\mathcal{U}}_\bullet^{x\perp} \times \bar{\mathcal{U}}_\bullet^{y\perp} \times {\mathcal{U}}_\times^\perp$ is given by $\tilde{\mathcal{R}} = (\mathcal{R}v_2, \mathcal{R}v_1, -\mathcal{R}v_3)$.
		
		Let us now discuss how the operators $L_\star^*$, $\star = \bullet, \leftrightarrow, \updownarrow$ behave under rotation. For instance, for $w_\bullet \in \mathcal{U}_\bullet \times \mathcal{U}_\bullet$ we have
		\begin{align} \notag
			(\bar{\mathcal{R}} L_\bullet^* w_\bullet)(n_1,n_2) &=
			\left(
			\begin{array}{l}
				-\tfrac12(w_\bullet^2(n_2 + \tfrac12, N_2 + 1 - n_1) \\ \qquad\qquad + w_\bullet^2(n_2 - \tfrac12, N_2 + 1 - n_1)) \\
				\tfrac12(w_\bullet^1(n_2, N_2 + 1 - n_1 - \tfrac12) \\ \qquad\qquad + w_\bullet^1(n_2, N_2 + 1 - n_1 + \tfrac12))
			\end{array}
			\right) \\ &= (L_\bullet^* \bar{\mathcal{R}} w_\bullet)(n_1,n_2)
		\end{align}
		for $\tfrac32 \leq n_1 \leq N_2 - \tfrac12$, $\tfrac32 \leq n_2 \leq N_1 - \tfrac12$ where $\bar{\mathcal{R}}$ on the right-hand side has to be understood, analogous to the above, as a mapping $\mathcal{U}_\bullet \times \mathcal{U}_\bullet \to \mathcal{U}_\bullet^\perp \times \mathcal{U}_\bullet^\perp$. Taking also the boundary cases into account, we are able to conclude that $\bar{\mathcal{R}}L_\bullet^* = L_\bullet^*\bar{\mathcal{R}}$. With the same reasoning, we also get that $\bar{\mathcal{R}}L_\leftrightarrow^* = L_\updownarrow^*\bar{\mathcal{R}}$ as well as $\bar{\mathcal{R}}L_\updownarrow^* = L_\leftrightarrow^*\bar{\mathcal{R}}$. This also applies to $L_\bullet^*$ given for $v \in \mathcal{U}_\bullet \times \mathcal{U}_\bullet \times \mathcal{U}_\bullet$ for which the identity $\tilde{\mathcal{R}}L_\bullet^* = L_\bullet^* \tilde{\mathcal{R}}$ holds for $\tilde{\mathcal{R}}$ on the right-hand side mapping $\mathcal{U}_\bullet \times \mathcal{U}_\bullet \times \mathcal{U}_\bullet \to \mathcal{U}_\bullet^\perp \times \mathcal{U}_\bullet^\perp \times \mathcal{U}_\bullet^\perp$.
		
		For the $u \in \mathcal{U}_\bullet$ given in the statement of the
		theorem, consider $(v_\bullet, w_\bullet, w_\leftrightarrow, w_\updownarrow, \omega)$ as well as
		\begin{equation}
			(v_\bullet^\perp, w_\bullet^\perp, w_\leftrightarrow^\perp, w_\updownarrow^\perp, \omega^\perp) = (\tilde{\mathcal{R}}v_\bullet, \bar{\mathcal{R}} w_\bullet,  \bar{\mathcal{R}} w_\updownarrow, \bar{\mathcal{R}} w_\leftrightarrow, \bar{\mathcal{R}}\omega).
		\end{equation}
		Now, using the above identities, we can see that
			\begin{align} \notag
				&  & \mathcal{D}^{new} u - \omega                             & = L_\bullet^* w_\bullet + L_\leftrightarrow^* w_\leftrightarrow + L_\updownarrow^* w_\updownarrow \\ \notag
  & \Leftrightarrow
				& \bar{\mathcal{R}} \mathcal{D}^{new} u - \bar{\mathcal{R}}\omega & = \bar{\mathcal{R}}L_\bullet^* w_\bullet + \bar{\mathcal{R}}L_\leftrightarrow^* w_\leftrightarrow + \bar{\mathcal{R}} L_\updownarrow^* w_\updownarrow            \\ \notag
				& \Leftrightarrow                                                       & \mathcal{D}^{new}\mathcal{R} u - \bar{\mathcal{R}}\omega & = L_\bullet^* \bar{\mathcal{R}}w_\bullet +  L_\updownarrow^* \bar{\mathcal{R}}w_\leftrightarrow +  L_\leftrightarrow^* \bar{\mathcal{R}}w_\updownarrow \\ 			\label{deg90_1}
                                  & \Leftrightarrow                                          & \mathcal{D}^{new} \mathcal{R} u - \omega^\perp                                                                                                         & = L_\bullet^* w_\bullet^\perp + L_\leftrightarrow^* w_\leftrightarrow^\perp + L_\updownarrow^* w_\updownarrow^\perp,
			\end{align}
		as well as
                        \begin{align} \notag
                          \mathcal{E}^{new} \omega &= L_\bullet^* v_\bullet &
                          \Leftrightarrow & & \tilde{\mathcal{R}}\mathcal{E}^{new}\omega
                          &= \tilde{\mathcal{R}} L_\bullet^* v_\bullet  \\ 			\label{deg90_2}
                          \Leftrightarrow \quad \mathcal{E}^{new}\bar{\mathcal{R}}
                          \omega &= L_\bullet^* \tilde{\mathcal{R}} v_\bullet
                          & \Leftrightarrow & &
                          \mathcal{E}^{new} \omega^\perp &= L_\bullet^* v_\bullet^\perp.
                        \end{align}
		Consequently, $(v_\bullet, w_\bullet, w_\leftrightarrow, w_\updownarrow, \omega)$ is feasible for~\eqref{nrt2} if and only if $(v_\bullet^\perp, w_\bullet^\perp, w_\leftrightarrow^\perp, w_\updownarrow^\perp, \omega^\perp)$ is feasible for~\eqref{nrt2} with $u$ replaced by the rotated image $\mathcal{R}u$.
		Finally, it is easy to see that $\bar{\mathcal{R}}$ and $\tilde{\mathcal{R}}$ preserve the $1$-norm such that
		\begin{multline}
			\label{deg90_3}
			\alpha_0 \|v_\bullet\|_1 + \alpha_1 \|w_\bullet\|_1 + \alpha_1 \|w_\leftrightarrow\|_1 + \alpha_1\|w_\updownarrow\|_1 \\ =
			\alpha_0 \|v_\bullet^\perp\|_1 + \alpha_1 \|w_\bullet^\perp\|_1 + \alpha_1 \|w_\leftrightarrow^\perp\|_1 + \alpha_1\|w_\updownarrow^\perp\|_1.
		\end{multline}
		With the latter three statements, i.e., \eqref{deg90_1}, \eqref{deg90_2} and \eqref{deg90_3}, the identity $\text{TGV}^{2(new)}_{\alpha}(u) = \text{TGV}^{2(new)}_{\alpha}(\mathcal{R}u)$ then follows directly from~\eqref{nrt2}.
	\end{proof}
	\section{Numerical Algorithms and Application to Image Restoration}
        \label{sec:algorithms}
	{In the following, we present numerical algorithms related to the proposed discrete TGV for solving inverse problems, including denoising and upscaling. We compare the restored image results with several discrete variational models. Additionally, to demonstrate the invariance property of the proposed model, we show that the denoising result remains unchanged after rotating the image.} Moreover, for some test images and their $90^{\circ}$ rotated versions, we compute and compare the value of the TGV for the proposed model and the classic discrete TGV. {The experimental MATLAB code that reproduces all materials is provided on Mendeley Data \cite{Hosseini_data}.}
	{\subsection{Formulation of Discrete Inverse Problems}
		Here, we formulate the discrete inverse problems utilizing variational models, and for denoising, and upscaling, we evaluate our proposed discrete TGV and compare it to classic discrete TV, Condat's TV, Shannon TGV, and the classic discretization of TGV.
		Consider the general form of the inverse problem}
	\begin{equation}\label{DNP}
		\min_{u\in \mathbb{R}^{N_1\times N_2}} \mathcal{F}(u)+ \mathcal{R}(u),
	\end{equation}
	which is the discrete form of the variational problem (\ref{OP}). In this formulation, {$\mathcal{F}(u)=\frac{1}{2}\|\mathcal{B}u-f\|^2$ for some degraded image
		$f\in Range(\mathcal{B})$ and linear operator $\mathcal{B}$.}
	Moreover, we can set any discrete total variation model or discrete second-order TGV model for $\mathcal{R}$. We consider the following five regularized problems:
	\begin{equation}\label{DPS}
			\begin{array}{ll}
				(a)\ \text{Classic TV }   & \min_u \mathcal{F}(u) + \lambda \text{TV}(u),            \\
				(b)\ \text{Condat's TV }  & \min_u \mathcal{F}(u) + \lambda \text{TV}_c(u),          \\
				(c)\ \text{$2^{nd}$ order Shannon TGV}\   & \min_u \mathcal{F}(u) + \text{TGV}^{2(\alpha)}_{\text{SH} (2)}(u),      \\
				(d)\ \text{Classic TGV }  & \min_u \mathcal{F}(u) + \text{TGV}_{\alpha}^2(u),        \\
				(e)\ \text{Proposed TGV} & \min_u \mathcal{F}(u) + \text{TGV}_{\alpha}^{2(new)}(u),
			\end{array}
	\end{equation}
	where $\lambda, \alpha_0, \alpha_1 >0$. As the numerical algorithms for solving problems (\ref{DPS}) $(a)$--$(d)$ have already been studied in the literature, we only focus here on describing a suitable algorithm for solving problem (\ref{DPS}) $(e)$. From Theorem \ref{t1}, it is easy to see that problem (\ref{DPS}) $(e)$ is equivalent to the following problem:
	\begin{equation}\label{pdm}
		\begin{array}{l}
                  \displaystyle \min_{v_\bullet, w_{\bullet}, w_{\leftrightarrow}, w_{\updownarrow}, u, \omega} \ \mathcal{F}(u)+\alpha_0\| v_\bullet\|_1+ \alpha_1\|{w}_{\bullet}\|_1 \\ \displaystyle \qquad\qquad\qquad\quad
                  +\alpha_1\| {w}_{\leftrightarrow}\|_1+
			\alpha_1\|{w}_{\updownarrow}\|_1 \\[0.5em]
			\quad\text{subject to} \qquad \bar L^* (
			v_\bullet,
			w_{\bullet},
			w_{\leftrightarrow},
			w_{\updownarrow},
			u,
			\omega
			)^T=0,
		\end{array}
	\end{equation}
	where $\bar L^* = \left(\begin{array}{cccccc} L_\bullet^* & 0 & 0 & 0 & 0 & -\mathcal{E}^{new} \\ 0 & L_\bullet^* & L_{\leftrightarrow}^* & L_{\updownarrow}^* & -\mathcal{D}^{new} & I
	\end{array}\right)$.
	In the numerical experiments below, we employ the Chambolle--Pock algorithm \cite{pridua} (see Algorithm \ref{algorithm1}). The algorithm generally can be used to solve the following optimization problem:
	\begin{equation}\label{conalg}
		\min_z \ {F}(Az) + {G}(z),
	\end{equation} and its dual form
	\begin{equation}
		\min_y \ {F}^*(y)+{G}^*(-A^*y),
	\end{equation}
	where $A: X \to Y$ is a linear and continuous operator, ${F}:Y\rightarrow {]{-\infty,\infty}]}$ and
	${G}:X\rightarrow {]{-\infty,\infty}]}$
	are proper, convex and lower semicontinuous functions whose corresponding proximal operators have simple forms or can easily be calculated. The algorithm is guaranteed to converge to a primal-dual solution pair, provided that a primal-dual solution exists, there is no duality gap and that $\sigma >0$, $\tau > 0$ satisfy  $\sigma \tau < \frac{1}{||A||^2}$. In practical situations where computing the exact value of $||A||$ is difficult, finding an upper bound $B > ||A||^2$ and setting $\sigma = \tau = \frac{1}{\sqrt{B}}$ is sufficient for convergence.
	\begin{algorithm}[t]
          \caption{The Chambolle--Pock algorithm for solving
            problem (\ref{conalg}).}\label{algorithm1}          
          \begin{algorithmic}[0]
            \Require $A, A^*, {F}^*, {G}$
            \Ensure For iteration number $N$, ${z}^{N}$ primal solution approximation and $y^N$ dual solution approximation
            \Initialization Choose parameters
            $\sigma >0$, $\tau>0$ with $\sigma\tau\|A\|^2 < 1$
            and initial estimates
            $(z^0, y^0) \in X\times Y, \tilde{z}^0=z^0$
            \While{convergence criterion not met, for $k=0,1,\ldots$}
              \State $ y^{k+1} = \text{prox}_{\sigma {F}^*} (y^k + \sigma A\tilde{z}^k)$
               \State $z^{k+1} = \text{prox}_{\tau {G}} (z^k - \tau A^* y^{k+1})$
               \State $\tilde{z}^{k+1} = 2z^{k+1} - z^k$
              \EndWhile
            \end{algorithmic}
          \end{algorithm}
	Assume $\mathcal{B}$ is a suitable linear operator. In this paper, $\mathcal{B}=I,$ for denoising, and for upscaling, $\mathcal{B}$ is set as downscaling operator.   In our case,  $f \in Range(\mathcal{B}),$ is a degraded image (noisy or downscaled image), and
		\begin{align*}
                  z&= (v_\bullet, w_{\bullet}, w_{\leftrightarrow}, w_{\updownarrow},u,\omega) \in X, \\
                  X &= \mathcal{U}_\bullet^3 \times \mathcal{U}_\bullet^2 \times \mathcal{U}_\leftrightarrow^2 \times \mathcal{U}_\updownarrow^2 \times \mathcal{U}_\bullet \times (\mathcal{U}_\leftrightarrow \times \mathcal{U}_\updownarrow), \\
			{G}(z)                                                                    & =\alpha_0\| v_\bullet\|_1+ \alpha_1\|{w}_{\bullet}\|_1+\alpha_1\| {w}_{\leftrightarrow}\|_1+
			\alpha_1\|{w}_{\updownarrow}\|_1.
		\end{align*}
		Furthermore, set $A = \left(\begin{array}{cccccc} L_\bullet^* & 0 & 0 & 0 & 0 & -\mathcal{E}^{new} \\ 0 & L_\bullet^* & L_{\leftrightarrow}^* & L_{\updownarrow}^* & -\mathcal{D}^{new} & I \\
			0 & 0 & 0 & 0 &\mathcal{B} & 0
		\end{array}\right)$, and %
              \begin{align*}
		y &=\left(\begin{array}{c}
                  v \\
                  w\\
                  s
		\end{array}\right) \in Y, \\ Y &= (\bar{\mathcal{U}}_\bullet^x \times \bar{\mathcal{U}}_\bullet^y \times \mathcal{U}_\times) \times (\mathcal{U}_\leftrightarrow \times \mathcal{U}_\updownarrow) \times Range(\mathcal{B}), \\ {F}(y)&=\frac{1}{2} \|s-f\|^2+I_{\{(0,0)\}}(v,w).
              \end{align*}
		It is not difficult to see that
		$$\text{prox}_{\sigma {F}^*}(v, w, s)=\left(v,w,\frac{s-f}{1+\sigma}\right), \quad (v,w,s)\in Y.$$
	Moreover, it is well known that the proximal operator of the $1$-norm is the so-called \emph{shrinkage operator} according to
        \begin{align*}
          \text{prox}_{\tau\|\cdot\|_1}(w)(n_1,n_2) &= \text{shrink}_{\tau}(w)(n_1,n_2) \\ &= \left(1-\frac{\tau}{\max\{|w(n_1,n_2)|,\tau\}}\right)w(n_1,n_2),
        \end{align*}
	for $w \in \mathcal{U}_\star \times \mathcal{U}_\star$, $\star = \bullet, \leftrightarrow, \updownarrow$ where $|w(n_1,n_2)| = \sqrt{w^1(n_1,n_2)^2 + w^2(n_1,n_2)^2}$. The proximal mapping of $\tau \|\cdot\|_1$ for $v \in \mathcal{U}_\bullet \times \mathcal{U}_\bullet \times \mathcal{U}_\bullet$ is given analogously with $|v(n_1,n_2)| = \sqrt{v^1(n_1,n_2)^2 + v^2(n_1,n_2)^2 + 2v^3(n_1,n_2)^2}$.
	It can easily be verified that $\text{prox}_{\tau \mathcal{G}}(v_\bullet,w_\bullet,w_\leftrightarrow,w_\updownarrow,u,\omega)=(\bar{v}_\bullet,\bar{w}_\bullet, \bar{w}_\leftrightarrow, \bar{w}_\updownarrow,\bar{u},\bar{\omega})$ where
	\begin{multline}\label{prox}
				\bar{v}_\bullet= \text{shrink}_{\alpha_0 \tau}(v_\bullet), \quad
				\bar{w}_\star = \text{shrink}_{\alpha_1 \tau}(w_\star), \quad \star = \bullet, \leftrightarrow, \updownarrow, \\ 
				\displaystyle \bar{u} = u, \quad \bar\omega = \omega.
		\end{multline}
	Based on above functionals and parameters, Algorithm \ref{algorithm2} is proposed to solve the inverse problem (\ref{pdm}) as well as its dual form
	which can equivalently be written as the following optimization problem:
	\begin{equation}\label{pdm2}
			\begin{array}{l}
				\displaystyle \min_{v,s} \tfrac{1}{2}\|s\|_2^2+\langle s,f\rangle \\
				\text{subject to} \quad
				\left\{\begin{aligned}
					\|L_\bullet v\|_\infty & \leq \alpha_0,       \\ \|L_\star w\|_\infty&\leq \alpha_1, \star=\bullet,\leftrightarrow, \updownarrow, \\
					\mathcal{B}^* s                      & = \text{div}^{2new}v.
				\end{aligned}\right.
			\end{array}
		\end{equation}
	\begin{algorithm}[t]
            \caption{An algorithm for solving the proposed TGV inverse problem (\ref{pdm}) and its dual form (\ref{pdm2}).} \label{algorithm2}
            \begin{algorithmic}
              \Require Degraded image {$f \in Range(\mathcal{B})$}, %
                $\alpha=(\alpha_0,\alpha_1), \alpha_0 > 0, \alpha_1 > 0$ %
              \Ensure For the iteration number $N$, $z^N=(v_\bullet^N, w_\bullet^N, w_\leftrightarrow^N, w_\updownarrow^N, u^N, \omega^N)$ is an approximation of the solution of the primal problem (\ref{pdm}) and $u^N$ is the obtained reconstructed			image. Moreover, $y^N=(v^N, w^N, s^N)$ is an approximation of a solution of the dual problem (\ref{pdm2})
            \Initialization  Choose $\sigma > 0$ and $\tau > 0$ such that $\sigma\tau \|A\|^2 < 1$, $A = \left(\begin{array}{cccccc} L_\bullet^* & 0 & 0 & 0 & 0 & -\mathcal{E}^{new} \\ 0 & L_\bullet^* & L_{\leftrightarrow}^* & L_{\updownarrow}^* & -\mathcal{D}^{new} & I \\
			0 & 0 & 0 & 0 &\mathcal{B} & 0
		\end{array}\right)$
            \State
            Choose the initial approximation $u^0=\tilde{u}^0 \in\mathcal{U}_\bullet$
            \State
            Choose an arbitrary initial approximation of a primal solution $v_\bullet^0 = \tilde{v}_\bullet^0 \in \mathcal{U}_\bullet \times \mathcal{U}_\bullet \times \mathcal{U}_\bullet$, $w_\star^0= \tilde{w}_{\star}^0 \in \mathcal{U}_{\star}\times \mathcal{U}_{\star}$, $\star = \bullet, {\leftrightarrow,} \updownarrow$, $\omega^0 = \tilde{\omega}^0 \in\mathcal{U}_{\leftrightarrow}\times \mathcal{U}_{\updownarrow}$
            \State
                Choose an arbitrary initial approximation of a dual solution $v^0\in\bar{\mathcal{U}}_{\bullet}^x\times \bar{\mathcal{U}}_{\bullet}^y\times \mathcal{U}_\times$, $w^0\in \mathcal{U}_{\leftrightarrow}\times \mathcal{U}_{\updownarrow}, s^0\in Range (\mathcal{B})$ 
                \While{convergence criterion not met, for\ $k=0,1,\ldots$}
                  \State $v^{k+1}=v^k+\sigma(L^*_\bullet\tilde v_\bullet^k-\mathcal{E}^{new}\tilde\omega^k)$
                    \State $w^{k+1}=w^k+\sigma(L^*_\bullet \tilde w_\bullet^k + L^*_\leftrightarrow \tilde w_\leftrightarrow^k + L^*_\updownarrow \tilde w_\updownarrow^k - \mathcal{D}^{new}\tilde u^k +\tilde\omega^k)$
                    \State $s^{k+1}=\frac{s^k+\sigma(\mathcal{B}u-f)}{1+\sigma}$
                    \ForAll{$\star = \bullet, \leftrightarrow, \updownarrow$}
                    \State $v_\bullet^{k+1} = \text{shrink}_{\alpha_0\tau}(v_\bullet^k - \tau L_\bullet v^{k+1})$ \State $w_\star^{k+1} = \text{shrink}_{\alpha_1\tau}(w_\star^k - \tau L_\star w^{k+1})$ \EndFor
                    \State $\displaystyle u^{k+1} ={u^k - \tau (\text{div}^{new} w^{k+1} + \mathcal{B}^*s)}$ \State  $\omega^{k+1} = \omega^k - \tau(\text{Div}^{new} v^{k+1} + w^{k+1})$
                    \State
                    $\tilde{v}_\bullet^{k+1}=2v_\bullet^{k+1}-v_\bullet^k$ %
                    \ForAll{$\star=\bullet, \leftrightarrow, \updownarrow$} \State $\tilde{w}_\star^{k+1}=2w^{k+1}_\star- w_\star^k$ \EndFor
                    \State $\tilde{u}^{k+1} = 2u^{k+1} - u^k$, $\tilde\omega^{k+1} = 2\omega^{k+1} - \omega^k$
               \EndWhile            
              \end{algorithmic}
            \end{algorithm}
            
	We also applied primal-dual algorithms for the four different variational models (\ref{DPS}) $(a)$--$(d)$ to solve {some} image {reconstruction} problems and compared the results with the newly proposed discrete TGV (problem (\ref{DPS}) $(e)$). {%
		An algorithm description for problems (\ref{DPS}) $(a)$--$(c)$ as well as a rough estimate of the computational complexity {for all these algorithms, and our proposed one, to solve the denoising problem (i.e., $\mathcal{B}=I$)} can be found in Appendix \ref{sec:alg_comp}.
		There, one can see that the number of basic floating-point operations (flops) for the new proposed model is about $1.8$ times the number of flops of the second-order TGV and Condat-TV. This fact is also in accordance with the observed CPU times and confirms that the computational complexity of the proposed model is acceptable.}

	In {all} algorithms, we need {to fix} the two parameters $\sigma > 0$ and $\tau > 0$ to satisfy $\sigma\tau \|A\|^2 < 1$. For  classic discrete TV and Condat's TV, we set $\tau=\frac{0.99}{8}, \sigma=\frac{0.99}{3}$ as recommended in \cite{condat1} and for both discrete TGV models, we set $\tau=\sigma=\frac{5}{37}.$ Moreover, we have incorporated the  minimum iteration number in our simulations. We stop the algorithm when the iteration number reaches 5000. The parameters
		of each model are optimized to achieve the best reconstruction with respect to
		the peak signal-to-noise ratio (PSNR) or structural similarity index measure (SSIM). In other words, with $u^\dagger$ denoting the reference image, for instance, in the proposed model (\ref{DPS}) ($e$), parameters are obtained by solving
		\begin{align} \notag
			\alpha_1^* \in \text{argmax}_{\alpha_1>0}\left\{\right.&\text{PSNR}/\text{SSIM}(u^*_{\alpha}, u^\dagger),  \\ \notag &\alpha=(\alpha_0,\alpha_1), \alpha_0=2\alpha_1, \\ \label{lambda} & \left.u^*_{\alpha}\ \text{ is the solution of}\ (\ref{DPS})\, (d)\right\},
	\end{align}
	In our computational experiments, we set $\alpha_0^*=2\alpha_1^*$.   We determine $\alpha_1^*$ by exhaustive search over a suitable regular grid within a finite interval. Note that here, we fix the ratio $\alpha_0^*/\alpha_1^* = 2$ for the TGV-based models, which could, of course, also be optimized.
	\subsection{Denoising}
	For the {test} images, %
	whose intensity values are stored as double-precision floating-point numbers in the range $[0,1]$, we artificially produce a noisy image by adding Gaussian noise with 0 mean and a fixed standard deviation to the respective clean image.
	We consider two criteria to compare the results{:} accuracy and the ability to remove artifacts (especially, the staircase effect). %
	In \cite{condat1}, Condat shows that the discrete total variation model developed in this work has better quality in terms of accuracy and isotropy in comparison with state-of-the-art discrete total variation models. Moreover, classic discrete TGV \cite{bredies} is a variational model whose experimental results show that it is very efficient in the sense of reducing artifacts such as the staircase effect.\\	
	{To illustrate visually distinctive features of the methods, we first apply the denoising methods {(\eqref{DPS} $(a)$, $(b)$, $(d)$ and $(e)$)} for {a} synthetic test image. We selected a deliberate test image comprising piecewise linear regions and sections with structured edges. The reference and noisy images are depicted in Figure~\ref{test1}.  Details and quality metrics (PSNR and structural similarity index measure (SSIM) \cite{ssim}) of the restored images are displayed in Figure \ref{check}.}
	\begin{figure} \centering \fontsize{8}{9.5}\selectfont
		\begin{tabular}{p{0.325\columnwidth}@{\ }p{0.325\columnwidth}@{\ }p{0.325\columnwidth}}
			\includegraphics[height=2.65cm]{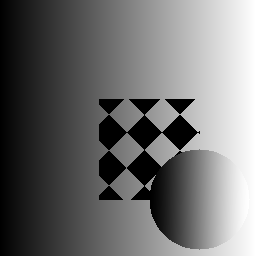}
			&
			\includegraphics[height=2.65cm]{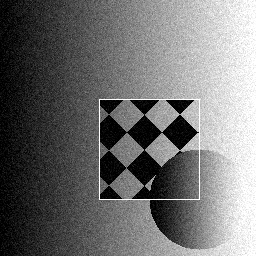}
			&
			\includegraphics[height=2.65cm]{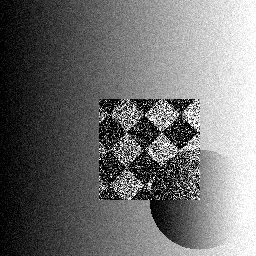} \\
			{($a$) Synthetic test image} & ($b$) a noisy image with a 
			& {($c$) the noisy image }\\
			& selected part highlighted & applied for denoising \\
                  & in the white box
		\end{tabular}
		\caption{Synthetic test image ($a$)  and its artificially generated noisy versions using Gaussian noise ($c$). The noisy image is generated in two stages: first, noise with a standard deviation of 0.05 is added to the reference image, followed by the addition of noise with a standard deviation of 0.4 to a specific region of the generated noisy image, highlighted by a white box $(b)$.}\label{test1}
	\end{figure}
	The restored images from Condat-TV and the proposed TGV method outperform other models in restoring edges (see the third column of Figure \ref{check}), whereas the proposed model demonstrates superiority in attenuating artifacts and removing noise. In the results of the first-order models (classic TV and Condat-TV), a significant amount of staircase artifacts can be observed in the partially linear areas. This experiment confirms that the proposed model preserves the edge restoration and detail retention properties of Condat-TV while also incorporating the inherent staircase effect mitigation of TGV. Additionally, it exhibits enhanced reliability in noise removal (see Figure \ref{check}).
	\begin{figure} \fontsize{8}{9.5}\selectfont
		\begin{tabular}{p{0.325\columnwidth}@{\ }p{0.325\columnwidth}@{\ }p{0.325\columnwidth}}
			\includegraphics[height=2.65cm]{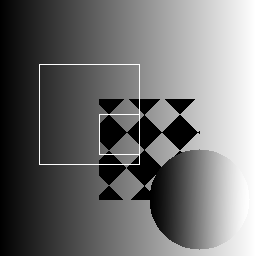}
			& \includegraphics[height=2.65cm]{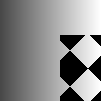}
			& \includegraphics[height=2.65cm]{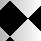}  \\
			\multicolumn{3}{c}{$(a)$ reference image and details}\\[0.25em]
			\includegraphics[height=2.65cm]{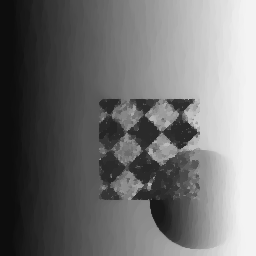}
			& \includegraphics[height=2.65cm]{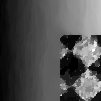}
			& \includegraphics[height=2.65cm]{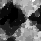} \\
			\multicolumn{3}{c}{$(b)$ TV restored and details}\\[0.25em]
			\includegraphics[height=2.65cm]{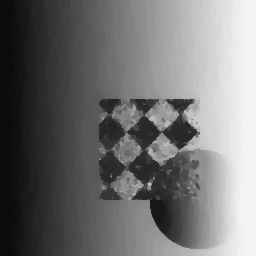}
			& \includegraphics[height=2.65cm]{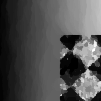}
			& \includegraphics[height=2.65cm]{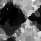} \\
			\multicolumn{3}{c}{$(c)$ Condat-TV restored and details}\\[0.25em]
			\includegraphics[height=2.65cm]{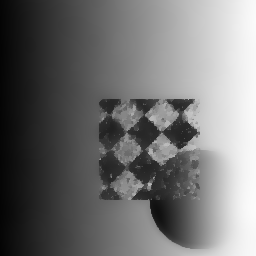}
			& \includegraphics[height=2.65cm]{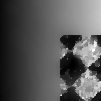}
			& \includegraphics[height=2.65cm]{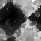} \\
			\multicolumn{3}{c}{$(d)$ TGV restored and details}\\[0.25em]
			\includegraphics[height=2.65cm]{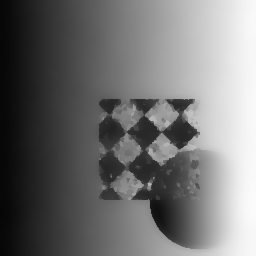}
			& \includegraphics[height=2.65cm]{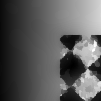}
			& \includegraphics[height=2.65cm]{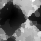} \\
			\multicolumn{3}{c}{$(e)$ New TGV restored and details}\\
		\end{tabular}
		\caption{Comparison of variational denoising results for the synthetic test image (Figure~\ref{test1}($a$)) with the best PSNR criterion. The parameters and quality metrics are as follows: ($b$) $\lambda^*=0.2050$, PSNR=24.22, SSIM=0.8626, ($c$) $\lambda^*=0.190$, PSNR=24.38, SSIM=0.8641, ($d$) $\alpha_1^*=0.205$, PSNR=24.34, SSIM=0.8838, ($e$) $\alpha_1^*=0.215$, PSNR=24.48, SSIM=0.8911. Results in the second column show that due to the piecewise linear nature of a part of the test image, staircase artifacts are evident for the first-order models; TV and Condat TV. The proposed TGV and the Condat-TV  are slightly more successful in restoring edges (for instance, look at the lower left and the upper left parts of the restored images in the third column). Overall, the proposed model is superior in attenuating artifacts, noise removal, and preserving edges.}\label{check}
	\end{figure}
	Moreover, we compare these variational denoising models for natural images (see the reference image  ``Girl'', in Figure \ref{fig000}), with zero-mean additive Gaussian noise and standard deviation of $0.1$. These images contain partially smooth areas as well as textures, edges, and fine details. Based on the above discussion, we expect that the proposed model outperforms the competing methods in terms of accuracy and artifact reduction.
	{As classic TV does not compete with other state-of-the-art variational models (at least in reducing artifacts), we illustrate the visual results for models in (\ref{DPS}), excluding classic TV. It is worth mentioning that the nature of Shannon TGV is different from the others because it involves interpolation in a domain four times the size of the grid domain of the given image, leading to significant computational complexity. In contrast, all other models are based on finite-difference operators with ranges of the same size as the given image. We include the results of the Shannon TGV model in our visual examples to demonstrate the effect of interpolation in a higher-dimensional grid domain. In our opinion, applying Shannon interpolation to the current approach could further reduce the fine edges of staircase artifacts and lead to improvements.
          
          In comparison with Condat's TV, and classic TGV}, the proposed model can restore images with better accuracy (PSNR and SSIM values) whereas staircase artifacts are more attenuated (see Figures  {\ref{fig1}, and \ref{fig2.5}}).  That is, the new proposed TGV is more accurate in comparison with Condat's TV and preserves the artifact-reducing property of the discrete classic second-order TGV (see again Figures {\ref{fig1}, and \ref{fig2.5}}). The resulting images are cleaner from noise and the PSNR and SSIM values are the highest in comparison with the other variational models.  To observe more details of the reconstructed images, parts of the obtained images are also shown in a zoomed version in the figures. To better judge the real superiority of the approach, we evaluated the performance of four variational models (TV, TGV, Condat-TV, and our proposed model) on the first 25 test images from the Berkeley Segmentation Dataset (BSDS). Table \ref{BLL} shows the average values of PSNR and SSIM for each model, where the parameters are tuned for both the best PSNR and SSIM criteria separately. The proposed method achieves the best PSNR and SSIM metrics (highlighted in bold).
        
	\begin{figure} \centering \fontsize{8}{9.5}\selectfont
 		\begin{tabular}{c@{\ \ }c}
			\includegraphics[height=4cm]{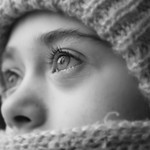} &
			\includegraphics[height=4cm]{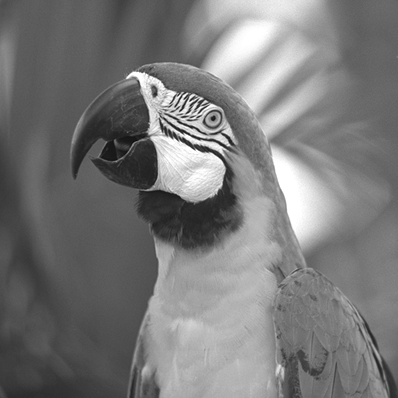}
			\\
			$(a)$ Girl                  &            $(b)$ Parrot  
		\end{tabular}
		\caption{{Reference images: images in ``png'' format with the following sizes: $(a)$ Girl ($150\times 150$ pixels), $(b)$ Parrot ($398\times 398$ pixels).}}\label{fig000}
	\end{figure}
	\begin{figure} \centering \fontsize{8}{9.5}\selectfont
          \begin{tabular}{c@{\ \ }c@{\ \ }c}
            \includegraphics[width=0.305\linewidth]{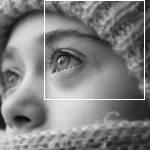} &
            \includegraphics[width=0.305\linewidth]{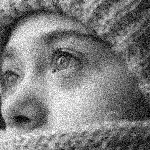} &
            \includegraphics[width=0.305\linewidth]{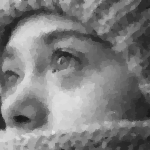} \\
            $(a)$ reference image & $(b)$ noisy image & $(c)$ Condat-TV-restored \\[0.25em]
            \includegraphics[width=0.305\linewidth]{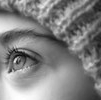} &
            \includegraphics[width=0.305\linewidth]{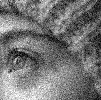} &
            \includegraphics[width=0.305\linewidth]{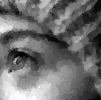} \\
            \multicolumn{1}{p{0.28\linewidth}}{\centering $(d)$ reference image: 
            details} &
            \multicolumn{1}{p{0.28\linewidth}}{\centering $(e)$ noisy image: details} &
             \multicolumn{1}{p{0.28\linewidth}}{\centering $(f)$ Condat-TV-restored: details}
          \end{tabular}          
		\caption{{Denoising experiment for the test image ``Girl'' with the best PSNR criterion: reference image, noisy image and restored image using Condat's discrete TV model are shown in $(a)$, $(b)$ and $(c)$, respectively. Some details of the reference and noisy image are shown in $(d)$ and $(e)$. The details of the restored image using Condat's TV model are shown in $(f)$, where strong staircase artifacts around the interfaces are evident. The parameters and quality metrics for $(c)$ are as follows: $\lambda^*=0.06$, PSNR=29.04, SSIM=0.8441.}}\label{fig1}
	\end{figure}
	\begin{figure} \fontsize{8}{9.5}\selectfont
          \begin{tabular}{c@{\ \ }c@{\ \ }c}
            \includegraphics[width=0.305\linewidth]{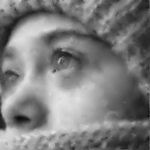} &
            \includegraphics[width=0.305\linewidth]{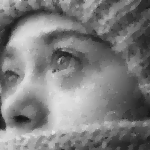} &
            \includegraphics[width=0.305\linewidth]{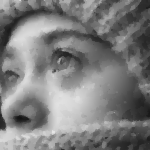} \\
            \multicolumn{1}{p{0.28\linewidth}}{\centering$(a)$ Shannon TGV-restored} & $(b)$ TGV-restored & $(c)$ new-TGV-restored \\[0.25em]
            \includegraphics[width=0.305\linewidth]{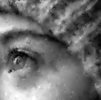} &
            \includegraphics[width=0.305\linewidth]{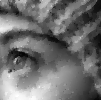} &
            \includegraphics[width=0.305\linewidth]{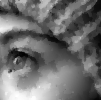} \\
            \multicolumn{1}{p{0.28\linewidth}}{\centering $(d)$ Shannon TGV-restored: details} &
            \multicolumn{1}{p{0.28\linewidth}}{\centering $(e)$ TGV-restored: details} &
             \multicolumn{1}{p{0.28\linewidth}}{\centering $(f)$ new-TGV-restored: details}
          \end{tabular}          
		\caption{{Denoising experiment for the test image ``Girl'' with the best PSNR criterion: the restored images using the Shannon TGV model, the second-order classic discrete TGV model, and the proposed model are shown in $(a)$, $(b)$, and $(c)$, respectively. Some details are shown in $(d)$, $(e)$, and $(f)$. The images obtained by means of the TGV models are far better than the restored image using Condat's TV model in terms of reducing staircase artifacts. In the result for classic discrete TGV, some unwanted narrow edges and speckles are visible in the smooth parts, whereas, in the proposed model, these kinds of effects have decreased significantly. On the other hand, the amount of noise in the image has been reduced using the proposed model. As we expect, because of using interpolation in a higher-dimensional grid domain in the Shannon TGV model, the very fine artifacts are attenuated. The parameters and quality metrics are as follows: $(a)$ $\alpha_1^*=0.05$, PSNR=29.67, SSIM=0.8697; $(b)$ $\alpha_1^*=0.06$, PSNR=29.16, SSIM=0.8529; $(c)$ $\alpha_1^*=0.06$, PSNR=29.45, SSIM=0.8657.}}\label{fig2.5}
	\end{figure}
	\begin{table}
		\centering
		\caption{{Denoising with both the PSNR and SSIM criteria: average PSNR and SSIM for each variational model, for the first 25 test images from BSDS.}}
		\label{tab:results}
		\begin{tabular}{lcc}
			\toprule
			Model   & Average PSNR & Average SSIM \\ \midrule
			TV                & 27.5052                 & 0.7726               \\ 
			Condat-TV              & 27.6486                & 0.7777                \\ 
			TGV         & 27.5332                & 0.7749               \\ 
			Proposed          & \textbf{27.6738}                 & \textbf{0.7798}                 \\ \botrule
		\end{tabular}\label{BLL}
              \end{table}
              
	\subsection{Upscaling}
	{In this section, we discuss the upscaling performance of our proposed model, focusing on its ability to enhance image resolution while preserving details and textures. We compare our method with the variational models in \eqref{DPS}, highlighting its superior performance in terms of texture reproduction and edge preservation. In the visual illustrations, results of classic TV are excluded. \\
		Upscaling involves increasing the resolution of an image $f$ by a factor of $k \in \mathbb{N}$ in both directions. It is considered the inverse problem of downscaling. The downscaling operator ``${Dow}$'' maps an image to the image of its averages over $k \times k$ blocks, and we suppose that $f = {Dow}(u)$ for some reference image $u$ that we want to estimate. In the general inverse problem \eqref{DNP}, assume $\mathcal{B} = {Dow}$. In our experiments, we set $k = 2,$
		and $f$ is a downscaled noisy version of the reference image $u$:  $f={Dow}(\hat{u}), \hat{u}=u+\epsilon,$ where $\epsilon$ is additive Gaussian noise with 0 mean and a standard deviation of 0.001.
		\\
		We present the results of upscaling the ``parrot'' test image (Figure \ref{fig000}) using our model and the compared variational models  (see Figures \ref{fig:ups1}, and \ref{fig:ups2}). Our method consistently outperforms the others, particularly in reproducing intricate textures and preserving edges, such as the textures in the parrot's cheek patches.
		One of the key strengths of our upscaling model is its ability to capture fine details with high fidelity.\\
		To further validate the performance of our upscaling model, we conducted a comprehensive analysis using the first 25 test images from the BSDS dataset for TV, TGV, Condat-TV,
		and our proposed model. Table \ref{BLL2} shows the average PSNR and SSIM values for these images, where the parameters are tuned for both the best PSNR and SSIM criteria separately. The proposed method achieves the best SSIM and PSNR values (highlighted in bold). This indicates the superior visual quality and detail preservation of the upscaled images produced by our model.
	}  
	\begin{figure} \centering \fontsize{8}{9.5}\selectfont
          \begin{tabular}{c@{\ \ }c@{\ \ }c}
            \includegraphics[width=0.305\linewidth]{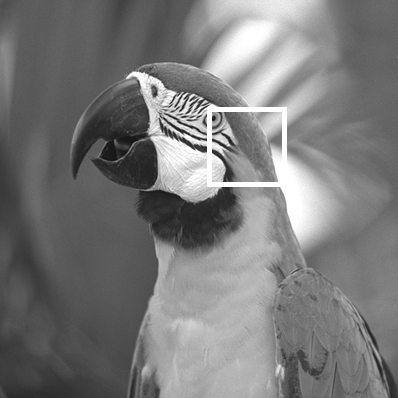} &
            \includegraphics[width=0.305\linewidth]{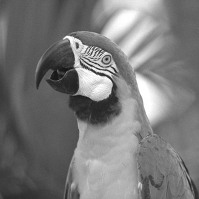} &
            \includegraphics[width=0.305\linewidth]{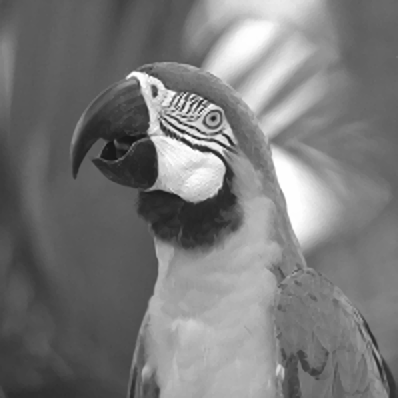} \\
            $(a)$ reference image & $(b)$ downscaled image & $(c)$ Condat-TV-restored \\[0.25em]
            \includegraphics[width=0.305\linewidth]{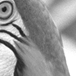} &
            \includegraphics[width=0.305\linewidth]{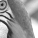} &
            \includegraphics[width=0.305\linewidth]{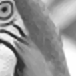} \\
            \multicolumn{1}{p{0.28\linewidth}}{\centering $(d)$ reference image: 
            details} &
            \multicolumn{1}{p{0.28\linewidth}}{\centering $(e)$ downscaled image: details} &
             \multicolumn{1}{p{0.28\linewidth}}{\centering $(f)$ Condat-TV-restored: details}
          \end{tabular}          
		\caption{{Upscaling experiment for the test image ``Parrot'' with the best PSNR criterion: the reference image, downscaled image, and restored image using Condat's discrete TV model are shown in $(a)$, $(b)$, and $(c)$, respectively. Some details of the reference and downscaled images are shown in $(d)$ and $(e)$. The details of the restored image using Condat's TV model are shown in $(f)$, where staircase artifacts in the smooth upper right region of the image are evident. The parameters and quality metrics for $(c)$ are as follows: $\lambda^*=0.0011$, PSNR=34.39, SSIM=0.9498. }}\label{fig:ups1}
	\end{figure}
	\begin{table}
		\centering
		\caption{{Upscaling with both the PSNR and SSIM criteria: average PSNR and SSIM for each variational model, for the first 25 test images from BSDS.}}
		\label{tab:results_upscaling}
		\begin{tabular}{lcc}
			\toprule
			Model   & Average PSNR & Average SSIM \\ \midrule
			TV                & 28.5268                 & 0.8658               \\ 
			Condat-TV              & 29.0356               & 0.8788               \\ 
			TGV         & 28.8466              & {0.8764}                \\ 
			Proposed          & \textbf{29.0443}                 & \textbf{0.8789}                 \\ \botrule
		\end{tabular}\label{BLL2}
	\end{table}
	\begin{figure} \centering \fontsize{8}{9.5}\selectfont
          \begin{tabular}{c@{\ \ }c@{\ \ }c}
            \includegraphics[width=0.305\linewidth]{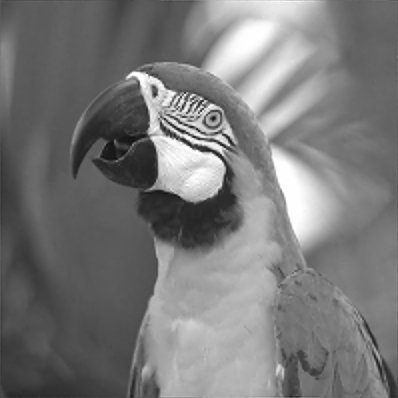} &
            \includegraphics[width=0.305\linewidth]{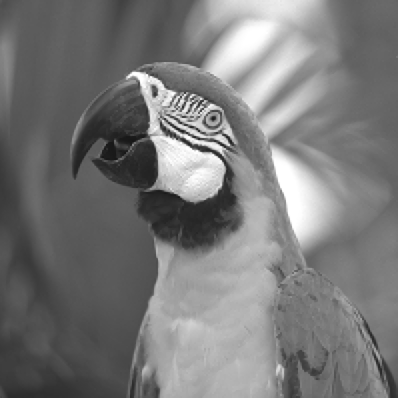} &
            \includegraphics[width=0.305\linewidth]{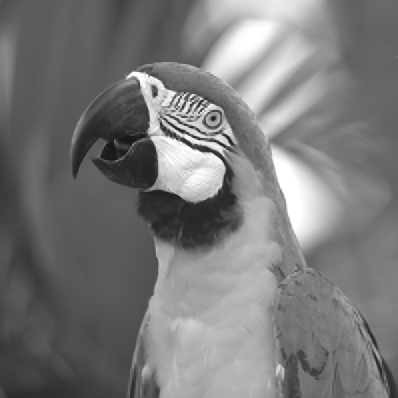} \\
            \multicolumn{1}{p{0.28\linewidth}}{\centering $(a)$ Shannon TGV-restored} & $(b)$ TGV-restored & $(c)$ new-TGV-restored \\[0.25em]
            \includegraphics[width=0.305\linewidth]{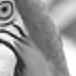} &
            \includegraphics[width=0.305\linewidth]{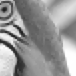} &
            \includegraphics[width=0.305\linewidth]{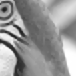} \\
 \multicolumn{1}{p{0.28\linewidth}}{\centering $(d)$ Shannon TGV-restored: details} &
            \multicolumn{1}{p{0.28\linewidth}}{\centering $(e)$ TGV-restored: details} &
             \multicolumn{1}{p{0.28\linewidth}}{\centering $(f)$ new-TGV-restored: details}
          \end{tabular}                    
		\caption{{Upscaling experiment for the test image ``Parrot'' with the best PSNR criterion: the restored images using the Shannon TGV 
				model, the second-order classic discrete TGV model, and the proposed model are shown in $(a)$, $(b)$, and $(c)$, respectively. Some details are shown in $(d)$, $(e)$, and $(f)$. The images obtained by means of the classic discrete TGV models are superior    
				to the restored image using Condat's TV model in terms of reducing staircase artifacts. The proposed model outperforms in    
				reproducing intricate  textures and preserving edges, such as the textures in the cheek patches. Despite the  
				expectation of obtaining better results for the Shannon TGV model due to its interpolation in a higher-dimensional grid domain, we cannot observe better visual results in the Shannon TGV restored image. The parameters and quality metrics are as follows: $(a)$ $\alpha_1^* = 0.0002$, PSNR = 34.37, SSIM = 0.9520; $(b)$ $\alpha_1^* = 0.0002$, PSNR = 33.84, SSIM = 0.9532; $(c)$ $\alpha_1^* = 0.0008$, PSNR = 34.42, SSIM = 0.9506.}}\label{fig:ups2}
	\end{figure}		
	\subsection{Effect of the Invariance Property in Denoising}
	The invariance property of the proposed model can clearly affect the results of inverse problems in imaging. In this section, we compare the denoising results of both the classic second order discrete TGV and the proposed model with the corresponding results obtained from the $90^\circ$ rotated images. Specifically, let $u$ be a given image and $f$ be the related noisy image. Additionally, let $\mathcal{R}u$ and $\mathcal{R}f$ denote the $90^\circ$ rotated versions of $u$ and $f$, respectively. We aim to solve denoising problems to restore $u$ and $\mathcal{R}u$ from $f$ and $\mathcal{R}f$, respectively. If $u^*$ is a solution for the original problem and $v^*$ is the solution for the rotated version, according to Theorem \ref{isot}, we expect that in the proposed model $v^*=\mathcal{R}u^*$, whereas this is not the case for the classic TGV. To demonstrate this fact empirically, consider $u$ and $f$ as the test image ``Girl'' and its noisy version, respectively, shown in Figure \ref{fig1}.\\
		We conducted a denoising experiment on the ``Girl'' test image using both the classic discrete TGV and our proposed model. In Figure \ref{fig:ROT1}, we compared the restored images for the original problem and the $-90^\circ$ rotated version of the restored image for $90^\circ$ rotated data using TGV. In the first row these two images are shown. We highlighted two small parts of both restored images in two boxes. In the second row, we show the intensities of the bigger box, and in the third row, we show the intensity values of the smaller box. The differences between the results in the second and third row are evident, indicating that the results for the original data and the rotated one are different.\\
		In Figure \ref{fig:ROT2}, we performed the same analysis for our proposed model. Interestingly, everything for both restored results (original data and rotated ones) was completely consistent, demonstrating the rotational invariance property of our model.\\
		Furthermore, we reported the results in Table \ref{table:ROT1} for the TGV model, including PSNR and SSIM values for the original problem and the rotated one, as well as primal and dual values of the primal-dual algorithm for the corresponding optimization problem for both of them. Additionally, we reported the 2-norm of the difference between the obtained results for the original problem and the $-90^\circ$ rotated version of the problem with $90^\circ$ rotated data. We included a comparative table for our model (Table \ref{table:ROT2}). The differences of all these values for the original and rotated problems were compared. For our model, all differences were of the order of at most $10^{-14}$, indicating perfect rotational invariance. In contrast, for the TGV model, the differences were at least of the order of $10^{-4}$, highlighting the superior rotational invariance property of our model.
              \begin{table*} \centering \tablebodyfont
                \caption{{Comparison of denoising results for the test image ``Girl'' and its $90^\circ$ rotated version using the classic discrete TGV: the differences in the quantities for both problems are at least of the order of $10^{-4}$, indicating that this model is not $90^\circ$ rotationally invariant}} %
		\begin{tabular}{lrrrrr}
			\toprule
			\multicolumn{1}{l}{Model/TGV} & \multicolumn{1}{c}{PSNR} & \multicolumn{1}{c}{SSIM}  & \multicolumn{1}{c}{Primal value} & \multicolumn{1}{c}{Dual value}                          \\%
			\midrule
			Original data                                & 29.1634                    & 0.8529                & 105.8004                            & 105.9118                                             \\
			Rotated data                                & 29.1498                    & 0.8525                 & 106.2450                      & 106.3516                                       \\
			\midrule
			Difference                                & 0.0136                    & $3.63\times 10^{-4}$                 & 0.4446                      & 0.4399 \\
			\midrule
			\rlap{$\|u^*-\mathcal{R}^{-1}v^*\|=0.2089$}                                &                    &                  &                       & \\				
			\botrule
		\end{tabular}
		\label{table:ROT1}
	\end{table*}
	\begin{table*}
          \centering \tablebodyfont
          \caption{{Comparison of denoising results for the test image ``Girl'' and its $90^\circ$ rotated version using the proposed model: the differences in the quantities for both problems are at most of the order of $10^{-14}$, indicating that this model is $90^\circ$ rotationally invariant.}} %
		\begin{tabular}{lrrrrr}
			\toprule
			\multicolumn{1}{l}{Model/Proposed} & \multicolumn{1}{c}{PSNR} & \multicolumn{1}{c}{SSIM}  & \multicolumn{1}{c}{Primal value} & \multicolumn{1}{c}{Dual value}                          \\%
			\midrule
			Original data                                & 29.4524                   & 0.8657                & 107.7793                            & 107.9145                                            \\
			Rotated data                                & 29.4524                   & 0.8657                 & 107.7793                      & 107.9145                                      \\
			\midrule
			Difference                                & $3.20\times 10^{-14}$                   & $1.11\times 10^{-15}$                 &  $2.84\times 10^{-14}$                     & $2.84\times 10^{-14}$\\
			\midrule
			\rlap{$\|u^*-\mathcal{R}^{-1}v^*\|=1.25\times 10^{-15}$}                             &                    &                  &                       & \\
			\bottomrule
		\end{tabular}
		\label{table:ROT2}
	\end{table*}
	\begin{figure} \centering \fontsize{8}{9.5}\selectfont
		\begin{tabular}{c@{\ }c}
			\includegraphics[height=4cm]{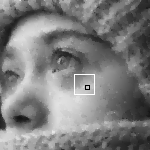} &
			\includegraphics[height=4cm]{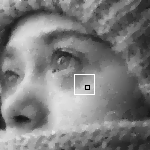} 
			\\
			$(a)$  $u^*$         & $(b)$ $\mathcal{R}^{-1}v^*$       \\[0.5em]
			\includegraphics[height=4cm]{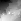}   &
			\includegraphics[height=4cm]{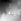}     
			\\
			$(c)$ details of the white box for $u^*$      & $(d)$ details of the white box for $\mathcal{R}^{-1}v^*$      \\[0.75em]
			\includegraphics[height=2cm]{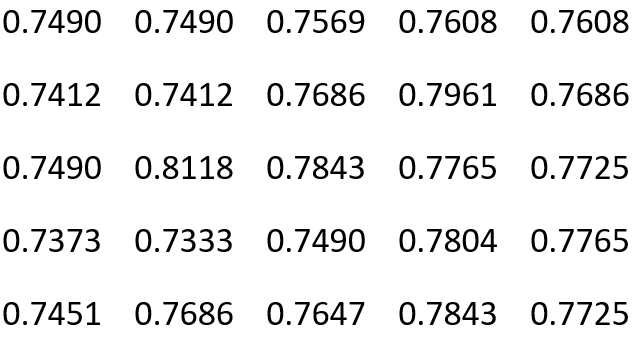}   &
			\includegraphics[height=2cm]{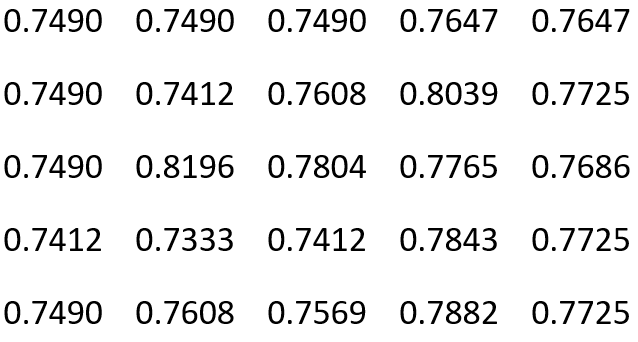}     
			\\[0.5em]
			$(e)$ details of the black box for $u^*$                & $(f)$ details of the black box for $\mathcal{R}^{-1}v^*$
		\end{tabular}
		\caption{Comparison of denoising results for the test image ``Girl'' and its $90^\circ$ rotated version using the classic discrete TGV: $(a)$ $u^*$ is the restored image for the original problem, $(b)$ $v^*$ is the restored image for the $90^\circ$ rotated data, and $\mathcal{R}^{-1}v^*$ is its $-90^\circ$ rotated version. Intensities of the highlighted white box for the result of the original problem and the rotated data are shown in $(c)$ and $(d)$, respectively. Intensity values in $[0,1]$ of the highlighted black box for the result of the original problem and the rotated data are shown in $(e)$ and $(f)$, respectively. The images in the second row and the intensity values in the third row show that the two restored images are significantly different in most parts.}\label{fig:ROT1}\end{figure}
	\begin{figure} \centering \fontsize{8}{9.5}\selectfont
		\begin{tabular}{c@{\ }c}
			\includegraphics[height=4.2cm]{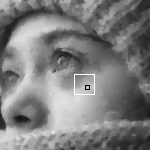} &
			\includegraphics[height=4.2cm]{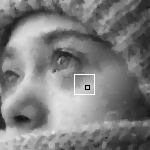} 
			\\
			$(a)$  $u^*$         & $(b)$ $\mathcal{R}^{-1}v^*$       \\[0.5em]
			\includegraphics[height=4.2cm]{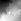}   &
			\includegraphics[height=4.2cm]{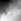}     
			\\
			$(c)$ details of the white box for $u^*$      & $(d)$ details of the white box for $\mathcal{R}^{-1}v^*$      \\[0.75em]
			\includegraphics[height=2cm]{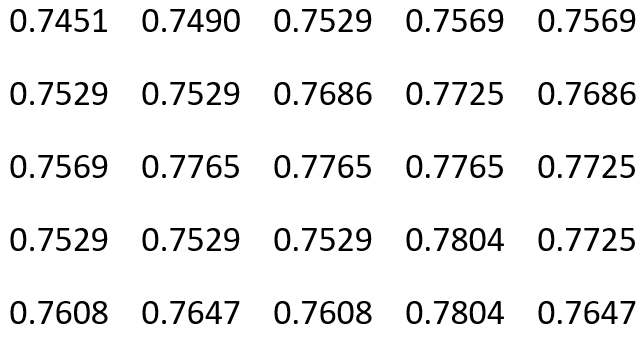}   &
			\includegraphics[height=2cm]{NewTGV-5.png}     
			\\[0.5em]
			$(e)$ details of the black box for $u^*$                & $(f)$ details of the black box for $\mathcal{R}^{-1}v^*$
		\end{tabular}
		\caption{Comparison of denoising results for the test image ``Girl'' and its $90^\circ$ rotated version using the proposed discrete TGV: $(a)$ $u^*$ is the restored image for the original problem, $(b)$ $v^*$ is the restored image for the $90^\circ$ rotated data, and $\mathcal{R}^{-1}v^*$ is its $-90^\circ$ rotated version. Intensities of the highlighted white box for the result of the original problem and the rotated data are shown in $(c)$ and $(d)$, respectively. Intensity values in $[0,1]$ of the highlighted black box for the result of the original problem and the rotated data are shown in $(e)$ and $(f)$, respectively. The images in the second row and the intensity values in the third row show that the two restored images are completely the same, indicating that the proposed model is $90^\circ$ rotationally invariant.}\label{fig:ROT2}
	\end{figure}
	\subsection{Computation of TGV; Classic TGV vs. Proposed TGV}
	In this section the rotational invariance property of the classic discrete TGV and the proposed discrete TGV are compared. For three test images (see Figure \ref{iso}), the classic discrete TGV values and the proposed discrete TGV values of the images as well as their $90^\circ$ rotated versions are calculated.		
	To compute an approximation of the classic discrete TGV and the proposed one, we solve the optimization problems (\ref{sim2}) and (\ref{sim3}), respectively, by means of the primal-dual algorithm (Algorithm \ref{algorithm1}).
	In order to solve (\ref{sim3}),
	let $\tilde{L}^*=\left(\begin{array}{cccccc}
		L_\bullet^* & 0             & 0                     & 0                  & -\mathcal{E}^{new} \\
		0           & L_{\bullet}^* & L_{\leftrightarrow}^* & L_{\updownarrow}^* & I
	\end{array}\right)$.
	Then, we can rewrite the constraint of (\ref{sim3}) as
	$$
	\tilde{L}^*(
	v_\bullet,
	{w}_{\bullet},
	{w}_{\leftrightarrow},
	{w}_{\updownarrow},
	\omega
	)=(0,\mathcal{D}^{new}u).
	$$
	Now, set
        \begin{align*}
          z&= (v_\bullet, w_{\bullet}, w_{\leftrightarrow}, w_{\updownarrow}, \omega), \\
          \mathcal{G}(z)&=\alpha_0\| v_\bullet\|_1+ \alpha_1\sum_{\star=\bullet,\leftrightarrow,\updownarrow}{\|{w}_{\star}\|_1}, \\ y&=\left(\begin{array}{c}
            v \\
            w
          \end{array}\right), \quad \mathcal{F}(y)=I_{\{(0, \mathcal{D}^{new}u)\}}(y), \quad A = \tilde L^*.
        \end{align*}
	Then, problem~\eqref{sim3} can be written in terms of~\eqref{conalg}. In order to employ Algorithm~\ref{algorithm1}, note that the proximal operators for $\sigma \mathcal{F}^*$ and $\tau \mathcal{G}$ read as:
        \begin{align*}
          \text{prox}_{\sigma \mathcal{F}^*}(v,w) &=(v,w-\sigma\mathcal{D}^{new}u), \\ \text{prox}_{\tau \mathcal{G}}(v_\bullet, w_\bullet,w_\leftrightarrow,w_\updownarrow,\omega)&=(\bar{v}_\bullet,\bar{w}_\bullet,\bar{w}_\updownarrow, \bar{w}_\leftrightarrow,\bar{\omega}),
        \end{align*}
	using the notation of~\eqref{prox}.
	As a result, the primal-dual algorithm to solve (\ref{sim3}) is outlined in Algorithm \ref{algorithm3}.
	\begin{algorithm}
          		\caption{An algorithm for solving
			problem (\ref{sim3}) and its dual form (\ref{nrt}).} \label{algorithm3}
                      \begin{algorithmic}
\Require Image $u \in \mathcal{U}_\bullet$,
                          $\alpha=(\alpha_0,\alpha_1), \alpha_0 > 0, \alpha_1 > 0$
                          \Ensure For the iteration number $N$, $z^N=(v_\bullet^N, w_\bullet^N, w_\leftrightarrow^N, w_\updownarrow^N, \omega^N)$ is an approximation of the solution of the primal problem (\ref{sim3}). Moreover, $y^N=(v^N, w^N)$ is an approximation of a solution of the dual problem (\ref{nrt}) and the approximated new discrete TGV value is $\alpha_0\| v_\bullet\|_1+ \alpha_1\sum_{\star=\bullet,\leftrightarrow,\updownarrow}{\|{w}_{\star}\|_1}$
                          \Initialization  Choose $\sigma > 0$ and $\tau > 0$ such that $\sigma\tau \|\tilde L^*\|^2 < 1$, $\tilde{L}^*=\left(\begin{array}{cccccc}
		L_\bullet^* & 0             & 0                     & 0                  & -\mathcal{E}^{new} \\
		0           & L_{\bullet}^* & L_{\leftrightarrow}^* & L_{\updownarrow}^* & I
	\end{array}\right)$ 
                        \State Choose an arbitrary initial approximation of a primal solution $v_\bullet^0 = \tilde{v}_\bullet^0 \in \mathcal{U}_\bullet \times \mathcal{U}_\bullet \times \mathcal{U}_\bullet$, $w_\star^0= \tilde{w}_{\star}^0 \in \mathcal{U}_{\star}\times \mathcal{U}_{\star}$, $\star = \bullet, {\leftrightarrow,} \updownarrow$, $\omega^0 = \tilde{\omega}^0 \in\mathcal{U}_{\leftrightarrow}\times \mathcal{U}_{\updownarrow}$
                        \State Choose an arbitrary initial approximation of a dual solution $v^0\in\bar{\mathcal{U}}_{\bullet}^x\times \bar{\mathcal{U}}_{\bullet}^y\times \mathcal{U}_\times$, $w^0\in \mathcal{U}_{\leftrightarrow}\times \mathcal{U}_{\updownarrow}$ 
                        \While{convergence criterion not met, for\ $k=0,1,\ldots$}
                          \State $v^{k+1}=v^k+\sigma(L^*_\bullet\tilde v_\bullet^k-\mathcal{E}^{new}\tilde\omega^k)$
                            \State $w^{k+1}=w^k+\sigma(L^*_\bullet \tilde w_\bullet^k + L^*_\leftrightarrow \tilde w_\leftrightarrow^k + L^*_\updownarrow \tilde w_\updownarrow^k - \mathcal{D}^{new}u +\tilde\omega^k)$
                            \ForAll{$\star = \bullet, \leftrightarrow, \updownarrow$} \State $v_\bullet^{k+1} = \text{shrink}_{\alpha_0\tau}(v_\bullet^k - \tau L_\bullet v^{k+1})$ \State $w_\star^{k+1} = \text{shrink}_{\alpha_1\tau}(w_\star^k - \tau L_\star w^{k+1})$ \EndFor
                            \State
                            $\omega^{k+1} = \omega^k - \tau(\text{Div}^{new} v^{k+1} + w^{k+1})$
                            \State $\tilde{v}_\bullet^{k+1}=2v_\bullet^{k+1}-v_\bullet^k$ %
                            \ForAll{$\star=\bullet, \leftrightarrow, \updownarrow$} \State $\tilde{w}_\star^{k+1}=2w^{k+1}_\star- w_\star^k$ \EndFor
                            \State
                            $\tilde\omega^{k+1} = 2\omega^{k+1} - \omega^k$
                          \EndWhile
                        \end{algorithmic}
                      \end{algorithm}
	We employed Algorithm \ref{algorithm3} with 1000 iterations in our simulations. Table \ref{table:2} confirms that the proposed TGV has invariant values, up to numerical precision, for the original images and their 90-degree rotated versions.
	\begin{figure} \fontsize{8}{9.5}\selectfont
               \fontsize{8}{9.5}\selectfont
		\begin{tabular}{c@{\ \ }c@{\ \ }c}
			\includegraphics[height=2.95cm]{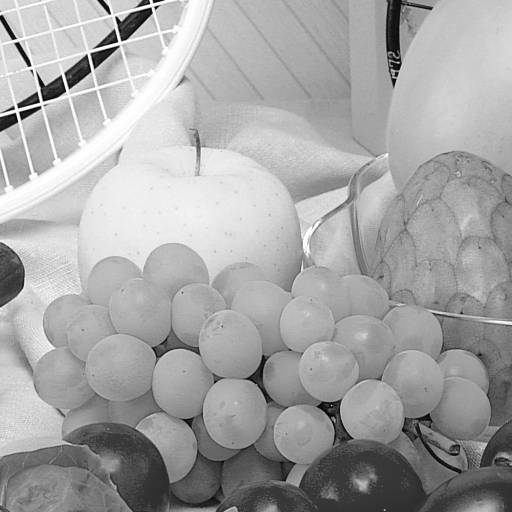} &
			\includegraphics[height=2.95cm]{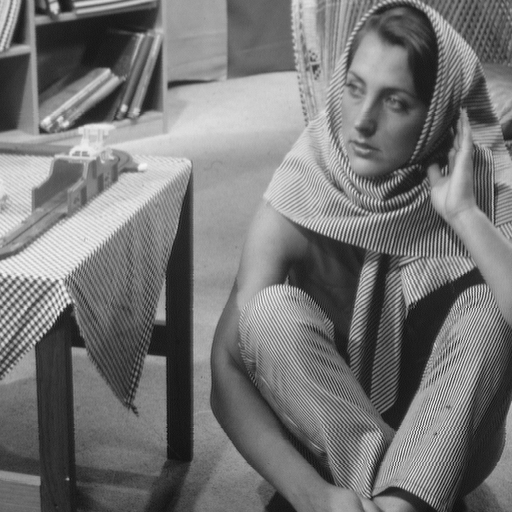}    &
			\includegraphics[height=2.95cm]{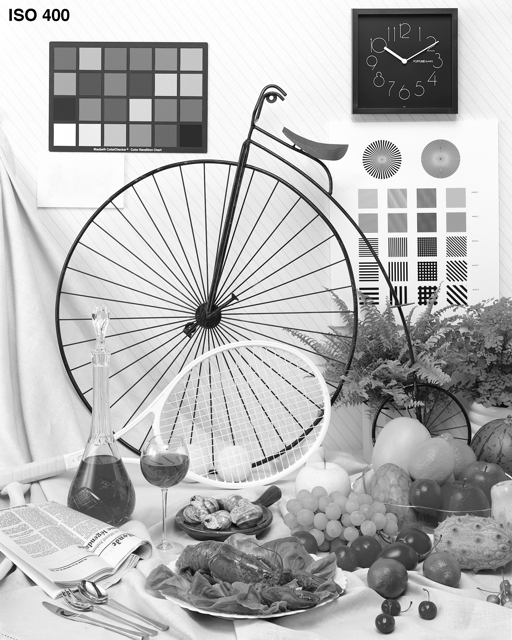}
			\\
			 $(a)$ Fruits                                 & $(b)$ Barbara & $(c)$ Bike
		\end{tabular}
		\caption{The test images used for the 90$^\circ$ isotropy testing problem. From left to right: Fruits, Barbara and Bike test image.}\label{iso}
              \end{figure}
              
	\begin{table*} \centering \tablebodyfont
          		\caption{The values of the classic discrete TGV and the proposed discrete TGV for three test images and their $90^\circ$ rotated versions {for the parameters $(\alpha_0,\alpha_1) = (0.14, 0.07)$}. Here, ``error'' is the absolute difference of the classic discrete TGV (proposed discrete TGV) value of the rotated image and the reference image.} %
		\label{table:2}
		\begin{tabular}{lrrrrrrr}
			\toprule
			Image                              &                               & \multicolumn{2}{c}{Fruits} & \multicolumn{2}{c}{Barbara}   & \multicolumn{2}{c}{Bike}                                                                                        \\
			\midrule
			\multicolumn{2}{l}{Model/Rotation} & \multicolumn{1}{c}{TGV value} & \multicolumn{1}{c}{Error}  & \multicolumn{1}{c}{TGV value} & \multicolumn{1}{c}{Error} & \multicolumn{1}{c}{TGV value} & \multicolumn{1}{c}{Error}                           \\%
			\midrule
			TGV                                & $0^\circ$                     & {587.0513}                 & --                            & 1326.9521                 & --                            & 1720.4807                 & --                      \\
			TGV                                & $90^\circ$                    & {588.2140}                 & {1.1627}                      & 1329.7168                 & 2.7647                        & 1706.1456                 & 14.3351                 \\
			New TGV                            & $0^\circ$                     & {632.2688}                 & --                            & 1421.8078                 & --                            & 1806.2271                 & --                      \\
			New TGV                            & $90^\circ$                    & {632.2688}                 & {$4.55\times 10^{-13}$}       & 1421.8078                 & {$2.27\times 10^{-13}$}       & 1806.2271                 & {$1.14\times 10^{-12}$} \\
			\bottomrule
		\end{tabular}
	\end{table*}	
	\section{Conclusion}
	In this paper, the idea of Condat's discrete total variation is transferred to the second-order TGV. A new discrete second-order TGV model is designed based on the building blocks containing the definition of suitable grids, introducing new discrete derivative and divergence operators and proposing suitable linear conversion operators to guarantee some invariance properties. The proposed model is invariant with respect to $90^\circ$ rotations {and preserves the benefits of Condat's model in reducing noise for the areas containing textures, edges and
		details.} %
	Moreover, the new discrete TGV preserves the ability of the classic discrete TGV to diminish artifacts such as staircase artifacts, which are typical for discrete %
	TV models. The same design principles can be applied for higher-order TGV or in higher dimensions to gain better results for imaging problems. %
	While this can quite easily be done for specific cases, the development of a general framework requires some effort and can thus be regarded a subject of future work.

        \backmatter

        \bmhead{Acknowledgements}

	This work is based upon research funded by Iran National Science Foundation (INSF) under project No. 4032410, and it was partially supported by IMU-CDC. {The Department of Mathematics and Scientific Computing, to which KB is affiliated, is a member of NAWI Graz (\url{https://www.nawigraz.at/en/}).}
        
        \bibliography{paper}


\begin{thebibliography}{44}
\ifx \bisbn   \undefined \def \bisbn  #1{ISBN #1}\fi
\ifx \binits  \undefined \def \binits#1{#1}\fi
\ifx \bauthor  \undefined \def \bauthor#1{#1}\fi
\ifx \batitle  \undefined \def \batitle#1{#1}\fi
\ifx \bjtitle  \undefined \def \bjtitle#1{#1}\fi
\ifx \bvolume  \undefined \def \bvolume#1{\textbf{#1}}\fi
\ifx \byear  \undefined \def \byear#1{#1}\fi
\ifx \bissue  \undefined \def \bissue#1{#1}\fi
\ifx \bfpage  \undefined \def \bfpage#1{#1}\fi
\ifx \blpage  \undefined \def \blpage #1{#1}\fi
\ifx \burl  \undefined \def \burl#1{\textsf{#1}}\fi
\ifx \doiurl  \undefined \def \doiurl#1{\url{https://doi.org/#1}}\fi
\ifx \betal  \undefined \def \betal{\textit{et al.}}\fi
\ifx \binstitute  \undefined \def \binstitute#1{#1}\fi
\ifx \binstitutionaled  \undefined \def \binstitutionaled#1{#1}\fi
\ifx \bctitle  \undefined \def \bctitle#1{#1}\fi
\ifx \beditor  \undefined \def \beditor#1{#1}\fi
\ifx \bpublisher  \undefined \def \bpublisher#1{#1}\fi
\ifx \bbtitle  \undefined \def \bbtitle#1{#1}\fi
\ifx \bedition  \undefined \def \bedition#1{#1}\fi
\ifx \bseriesno  \undefined \def \bseriesno#1{#1}\fi
\ifx \blocation  \undefined \def \blocation#1{#1}\fi
\ifx \bsertitle  \undefined \def \bsertitle#1{#1}\fi
\ifx \bsnm \undefined \def \bsnm#1{#1}\fi
\ifx \bsuffix \undefined \def \bsuffix#1{#1}\fi
\ifx \bparticle \undefined \def \bparticle#1{#1}\fi
\ifx \barticle \undefined \def \barticle#1{#1}\fi
\bibcommenthead
\ifx \bconfdate \undefined \def \bconfdate #1{#1}\fi
\ifx \botherref \undefined \def \botherref #1{#1}\fi
\ifx \url \undefined \def \url#1{\textsf{#1}}\fi
\ifx \bchapter \undefined \def \bchapter#1{#1}\fi
\ifx \bbook \undefined \def \bbook#1{#1}\fi
\ifx \bcomment \undefined \def \bcomment#1{#1}\fi
\ifx \oauthor \undefined \def \oauthor#1{#1}\fi
\ifx \citeauthoryear \undefined \def \citeauthoryear#1{#1}\fi
\ifx \endbibitem  \undefined \def \endbibitem {}\fi
\ifx \bconflocation  \undefined \def \bconflocation#1{#1}\fi
\ifx \arxivurl  \undefined \def \arxivurl#1{\textsf{#1}}\fi
\csname PreBibitemsHook\endcsname

\bibitem[\protect\citeauthoryear{Abergel and Moisan}{2017}]{Shanon}
\begin{barticle}
\bauthor{\bsnm{Abergel}, \binits{R.}},
\bauthor{\bsnm{Moisan}, \binits{L.}}:
\batitle{The {Shannon} total variation}.
\bjtitle{Journal of Mathematical Imaging and Vision}
\bvolume{59}(\bissue{2}),
\bfpage{341}--\blpage{370}
(\byear{2017})
\doiurl{10.1007/s10851-017-0733-5}
\end{barticle}
\endbibitem

\bibitem[\protect\citeauthoryear{Alter et~al.}{2005}]{Chambol3}
\begin{barticle}
\bauthor{\bsnm{Alter}, \binits{F.}},
\bauthor{\bsnm{Caselles}, \binits{V.}},
\bauthor{\bsnm{Chambolle}, \binits{A.}}:
\batitle{Evolution of characteristic functions of convex sets in the plane by
  the minimizing total variation flow}.
\bjtitle{Interfaces and Free Boundaries, Mathematical Analysis, Computation and
  Applications}
\bvolume{7}(\bissue{1}),
\bfpage{29}--\blpage{53}
(\byear{2005})
\doiurl{10.4171/ifb/112}
\end{barticle}
\endbibitem

\bibitem[\protect\citeauthoryear{Baumgärtner et~al.}{2023}]{Lukas}
\begin{barticle}
\bauthor{\bsnm{Baumgärtner}, \binits{L.}},
\bauthor{\bsnm{Bergmann}, \binits{R.}},
\bauthor{\bsnm{Herzog}, \binits{R.}},
\bauthor{\bsnm{Schmidt}, \binits{S.}},
\bauthor{\bsnm{Vidal-Núnez}, \binits{J.}}:
\batitle{Total generalized variation for piecewise constant functions on
  triangular meshes with applications in imaging}.
\bjtitle{SIAM Journal on Imaging Sciences}
\bvolume{16}(\bissue{1}),
\bfpage{313}--\blpage{339}
(\byear{2023})
\doiurl{10.1137/22m1505281}
\end{barticle}
\endbibitem

\bibitem[\protect\citeauthoryear{Bauschke and Combettes}{2017}]{Convex}
\begin{bbook}
\bauthor{\bsnm{Bauschke}, \binits{H.H.}},
\bauthor{\bsnm{Combettes}, \binits{P.L.}}:
\bbtitle{Convex Analysis and Monotone Operator Theory in Hilbert Spaces}.
\bpublisher{Springer}
(\byear{2017}).
\doiurl{10.1007/978-3-319-48311-5}
\end{bbook}
\endbibitem

\bibitem[\protect\citeauthoryear{Bogensperger et~al.}{2023}]{learnedTGV}
\begin{bbook}
\bauthor{\bsnm{Bogensperger}, \binits{L.}},
\bauthor{\bsnm{Chambolle}, \binits{A.}},
\bauthor{\bsnm{Effland}, \binits{A.}},
\bauthor{\bsnm{Pock}, \binits{T.}}:
\bbtitle{Learned Discretization Schemes for the Second-Order Total
  Generalized Variation},
pp. \bfpage{484}--\blpage{497}.
\bpublisher{Springer}
(\byear{2023}).
\doiurl{10.1007/978-3-031-31975-4_37}
\end{bbook}
\endbibitem

\bibitem[\protect\citeauthoryear{Bredies et~al.}{2018}]{brediestgv14}
\begin{barticle}
\bauthor{\bsnm{Bredies}, \binits{K.}},
\bauthor{\bsnm{Holler}, \binits{M.}},
\bauthor{\bsnm{Storath}, \binits{M.}},
\bauthor{\bsnm{Weinmann}, \binits{A.}}:
\batitle{Total generalized variation for manifold-valued data}.
\bjtitle{SIAM Journal on Imaging Sciences}
\bvolume{11}(\bissue{3}),
\bfpage{1785}--\blpage{1848}
(\byear{2018})
\doiurl{10.1137/17m1147597}
\end{barticle}
\endbibitem

\bibitem[\protect\citeauthoryear{Bredies}{2014}]{brediestgv12}
\begin{bbook}
\bauthor{\bsnm{Bredies}, \binits{K.}}:
\bbtitle{Recovering Piecewise Smooth Multichannel Images by Minimization of
  Convex Functionals with Total Generalized Variation Penalty},
pp. \bfpage{44}--\blpage{77}.
\bpublisher{Springer}
(\byear{2014}).
\doiurl{10.1007/978-3-642-54774-4_3}
\end{bbook}
\endbibitem

\bibitem[\protect\citeauthoryear{Bredies and Holler}{2014}]{brediestgv11}
\begin{barticle}
\bauthor{\bsnm{Bredies}, \binits{K.}},
\bauthor{\bsnm{Holler}, \binits{M.}}:
\batitle{Regularization of linear inverse problems with total generalized
  variation}.
\bjtitle{Journal of Inverse and Ill-posed Problems}
\bvolume{22}(\bissue{6}),
\bfpage{871}--\blpage{913}
(\byear{2014})
\doiurl{10.1515/jip-2013-0068}
\end{barticle}
\endbibitem

\bibitem[\protect\citeauthoryear{Bredies and Holler}{2015a}]{tgv2}
\begin{barticle}
\bauthor{\bsnm{Bredies}, \binits{K.}},
\bauthor{\bsnm{Holler}, \binits{M.}}:
\batitle{A {TGV}-based framework for variational image decompression, zooming,
  and reconstruction. part {I}: Analytics}.
\bjtitle{SIAM Journal on Imaging Sciences}
\bvolume{8}(\bissue{4}),
\bfpage{2814}--\blpage{2850}
(\byear{2015})
\doiurl{10.1137/15m1023865}
\end{barticle}
\endbibitem

\bibitem[\protect\citeauthoryear{Bredies and Holler}{2015b}]{tgv3}
\begin{barticle}
\bauthor{\bsnm{Bredies}, \binits{K.}},
\bauthor{\bsnm{Holler}, \binits{M.}}:
\batitle{A {TGV}-based framework for variational image decompression, zooming,
  and reconstruction. part {II}: Numerics}.
\bjtitle{SIAM Journal on Imaging Sciences}
\bvolume{8}(\bissue{4}),
\bfpage{2851}--\blpage{2886}
(\byear{2015})
\doiurl{10.1137/15m1023877}
\end{barticle}
\endbibitem

\bibitem[\protect\citeauthoryear{Bredies and Holler}{2020}]{tgv5}
\begin{barticle}
\bauthor{\bsnm{Bredies}, \binits{K.}},
\bauthor{\bsnm{Holler}, \binits{M.}}:
\batitle{Higher-order total variation approaches and generalisations}.
\bjtitle{Inverse Problems}
\bvolume{36}(\bissue{12}),
\bfpage{123001}
(\byear{2020})
\doiurl{10.1088/1361-6420/ab8f80}
\end{barticle}
\endbibitem

\bibitem[\protect\citeauthoryear{Bredies et~al.}{2010}]{bredies}
\begin{barticle}
\bauthor{\bsnm{Bredies}, \binits{K.}},
\bauthor{\bsnm{Kunisch}, \binits{K.}},
\bauthor{\bsnm{Pock}, \binits{T.}}:
\batitle{Total generalized variation}.
\bjtitle{SIAM Journal on Imaging Sciences}
\bvolume{3}(\bissue{3}),
\bfpage{492}--\blpage{526}
(\byear{2010})
\doiurl{10.1137/090769521}
\end{barticle}
\endbibitem

\bibitem[\protect\citeauthoryear{Bredies et~al.}{2020}]{brediestgv19}
\begin{barticle}
\bauthor{\bsnm{Bredies}, \binits{K.}},
\bauthor{\bsnm{Nuster}, \binits{R.}},
\bauthor{\bsnm{Watschinger}, \binits{R.}}:
\batitle{{TGV}-regularized inversion of the {Radon} transform for photoacoustic
  tomography}.
\bjtitle{Biomedical Optics Express}
\bvolume{11}(\bissue{2}),
\bfpage{994}
(\byear{2020})
\doiurl{10.1364/boe.379941}
\end{barticle}
\endbibitem

\bibitem[\protect\citeauthoryear{Buades et~al.}{2010}]{spacial2}
\begin{barticle}
\bauthor{\bsnm{Buades}, \binits{A.}},
\bauthor{\bsnm{Coll}, \binits{B.}},
\bauthor{\bsnm{Morel}, \binits{J.M.}}:
\batitle{Image denoising methods. a new nonlocal principle}.
\bjtitle{SIAM Review}
\bvolume{52}(\bissue{1}),
\bfpage{113}--\blpage{147}
(\byear{2010})
\doiurl{10.1137/090773908}
\end{barticle}
\endbibitem

\bibitem[\protect\citeauthoryear{Chambolle}{1999}]{M1}
\begin{barticle}
\bauthor{\bsnm{Chambolle}, \binits{A.}}:
\batitle{Finite-differences discretizations of the {Mumford}--{Shah}
  functional}.
\bjtitle{ESAIM: Mathematical Modelling and Numerical Analysis}
\bvolume{33}(\bissue{2}),
\bfpage{261}--\blpage{288}
(\byear{1999})
\doiurl{10.1051/m2an:1999115}
\end{barticle}
\endbibitem

\bibitem[\protect\citeauthoryear{Chambolle}{2004}]{chambollealg}
\begin{barticle}
\bauthor{\bsnm{Chambolle}, \binits{A.}}:
\batitle{An algorithm for total variation minimization and applications}.
\bjtitle{Journal of Mathematical Imaging and Vision}
\bvolume{20}(\bissue{1/2}),
\bfpage{89}--\blpage{97}
(\byear{2004})
\doiurl{10.1023/B:JMIV.0000011325.36760.1e}
\end{barticle}
\endbibitem

\bibitem[\protect\citeauthoryear{Chambolle and Pock}{2010}]{pridua}
\begin{barticle}
\bauthor{\bsnm{Chambolle}, \binits{A.}},
\bauthor{\bsnm{Pock}, \binits{T.}}:
\batitle{A first-order primal-dual algorithm for convex problems
  with applications to imaging}.
\bjtitle{Journal of Mathematical Imaging and Vision}
\bvolume{40}(\bissue{1}),
\bfpage{120}--\blpage{145}
(\byear{2010})
\doiurl{10.1007/s10851-010-0251-1}
\end{barticle}
\endbibitem

\bibitem[\protect\citeauthoryear{Chambolle and Pock}{2020}]{finit}
\begin{barticle}
\bauthor{\bsnm{Chambolle}, \binits{A.}},
\bauthor{\bsnm{Pock}, \binits{T.}}:
\batitle{{Crouzeix}–{Raviart} approximation of the total variation on
  simplicial meshes}.
\bjtitle{Journal of Mathematical Imaging and Vision}
\bvolume{62}(\bissue{6–7}),
\bfpage{872}--\blpage{899}
(\byear{2020})
\doiurl{10.1007/s10851-019-00939-3}
\end{barticle}
\endbibitem

\bibitem[\protect\citeauthoryear{Chambolle et~al.}{2011}]{upwind}
\begin{barticle}
\bauthor{\bsnm{Chambolle}, \binits{A.}},
\bauthor{\bsnm{Levine}, \binits{S.E.}},
\bauthor{\bsnm{Lucier}, \binits{B.J.}}:
\batitle{An upwind finite-difference method for total variation–based image
  smoothing}.
\bjtitle{SIAM Journal on Imaging Sciences}
\bvolume{4}(\bissue{1}),
\bfpage{277}--\blpage{299}
(\byear{2011})
\doiurl{10.1137/090752754}
\end{barticle}
\endbibitem

\bibitem[\protect\citeauthoryear{Chambolle et~al.}{2010}]{tv}
\begin{bbook}
\bauthor{\bsnm{Chambolle}, \binits{A.}},
\bauthor{\bsnm{Caselles}, \binits{V.}},
\bauthor{\bsnm{Cremers}, \binits{D.}},
\bauthor{\bsnm{Novaga}, \binits{M.}},
\bauthor{\bsnm{Pock}, \binits{T.}}:
\bbtitle{An Introduction to Total Variation for Image Analysis},
pp. \bfpage{263}--\blpage{340}.
\bpublisher{De Gruyter}
(\byear{2010}).
\doiurl{10.1515/9783110226157.263}
\end{bbook}
\endbibitem

\bibitem[\protect\citeauthoryear{Condat}{2017}]{condat1}
\begin{barticle}
\bauthor{\bsnm{Condat}, \binits{L.}}:
\batitle{Discrete total variation: New definition and minimization}.
\bjtitle{SIAM Journal on Imaging Sciences}
\bvolume{10}(\bissue{3}),
\bfpage{1258}--\blpage{1290}
(\byear{2017})
\doiurl{10.1137/16m1075247}
\end{barticle}
\endbibitem

\bibitem[\protect\citeauthoryear{Dabov et~al.}{2007}]{wav3}
\begin{barticle}
\bauthor{\bsnm{Dabov}, \binits{K.}},
\bauthor{\bsnm{Foi}, \binits{A.}},
\bauthor{\bsnm{Katkovnik}, \binits{V.}},
\bauthor{\bsnm{Egiazarian}, \binits{K.}}:
\batitle{Image denoising by sparse 3-{D} transform-domain collaborative
  filtering}.
\bjtitle{IEEE Transactions on Image Processing}
\bvolume{16}(\bissue{8}),
\bfpage{2080}--\blpage{2095}
(\byear{2007})
\doiurl{10.1109/tip.2007.901238}
\end{barticle}
\endbibitem

\bibitem[\protect\citeauthoryear{Ghazel et~al.}{2003}]{spacial}
\begin{barticle}
\bauthor{\bsnm{Ghazel}, \binits{M.}},
\bauthor{\bsnm{Freeman}, \binits{G.H.}},
\bauthor{\bsnm{Vrscay}, \binits{E.R.}}:
\batitle{Fractal image denoising}.
\bjtitle{IEEE Transactions on Image Processing}
\bvolume{12}(\bissue{12}),
\bfpage{1560}--\blpage{1578}
(\byear{2003})
\doiurl{10.1109/tip.2003.818038}
\end{barticle}
\endbibitem

\bibitem[\protect\citeauthoryear{Hohm et~al.}{2015}]{M3}
\begin{barticle}
\bauthor{\bsnm{Hohm}, \binits{K.}},
\bauthor{\bsnm{Storath}, \binits{M.}},
\bauthor{\bsnm{Weinmann}, \binits{A.}}:
\batitle{An algorithmic framework for {Mumford}–{Shah} regularization of
  inverse problems in imaging}.
\bjtitle{Inverse Problems}
\bvolume{31}(\bissue{11}),
\bfpage{115011}
(\byear{2015})
\doiurl{10.1088/0266-5611/31/11/115011}
\end{barticle}
\endbibitem

\bibitem[\protect\citeauthoryear{Hosseini}{2019}]{Hosseini2}
\begin{barticle}
\bauthor{\bsnm{Hosseini}, \binits{A.}}:
\batitle{New discretization of total variation functional for image processing
  tasks}.
\bjtitle{Signal Processing: Image Communication}
\bvolume{78},
\bfpage{62}--\blpage{76}
(\byear{2019})
\doiurl{10.1016/j.image.2019.06.005}
\end{barticle}
\endbibitem

\bibitem[\protect\citeauthoryear{Hosseini and Bazm}{2023}]{Hosseini}
\begin{barticle}
\bauthor{\bsnm{Hosseini}, \binits{A.}},
\bauthor{\bsnm{Bazm}, \binits{S.}}:
\batitle{The second-order {Shannon} total generalized variation for image
  restoration}.
\bjtitle{Signal Processing}
\bvolume{204},
\bfpage{108848}
(\byear{2023})
\doiurl{10.1016/j.sigpro.2022.108848}
\end{barticle}
\endbibitem

\bibitem[\protect\citeauthoryear{Hosseini and Bredies}{2024}]{Hosseini_data}
\begin{botherref}
\oauthor{\bsnm{Hosseini}, \binits{A.}},
\oauthor{\bsnm{Bredies}, \binits{K.}}:
A Second-Order {TGV} Discretization with 90° Rotational Invariance Property
  (Supplementary {MATLAB} files).
Mendeley Data
(2024).
\doiurl{10.17632/wbwfxht3hb}
\end{botherref}
\endbibitem

\bibitem[\protect\citeauthoryear{Hu et~al.}{2016}]{reg}
\begin{barticle}
\bauthor{\bsnm{Hu}, \binits{Y.}},
\bauthor{\bsnm{Wang}, \binits{N.}},
\bauthor{\bsnm{Tao}, \binits{D.}},
\bauthor{\bsnm{Gao}, \binits{X.}},
\bauthor{\bsnm{Li}, \binits{X.}}:
\batitle{{SERF}: A simple, effective, robust, and fast image super-resolver
  from cascaded linear regression}.
\bjtitle{IEEE Transactions on Image Processing}
\bvolume{25}(\bissue{9}),
\bfpage{4091}--\blpage{4102}
(\byear{2016})
\doiurl{10.1109/tip.2016.2580942}
\end{barticle}
\endbibitem

\bibitem[\protect\citeauthoryear{Huber et~al.}{2019}]{brediestgv18}
\begin{barticle}
\bauthor{\bsnm{Huber}, \binits{R.}},
\bauthor{\bsnm{Haberfehlner}, \binits{G.}},
\bauthor{\bsnm{Holler}, \binits{M.}},
\bauthor{\bsnm{Kothleitner}, \binits{G.}},
\bauthor{\bsnm{Bredies}, \binits{K.}}:
\batitle{Total generalized variation regularization for multi-modal electron
  tomography}.
\bjtitle{Nanoscale}
\bvolume{11}(\bissue{12}),
\bfpage{5617}--\blpage{5632}
(\byear{2019})
\doiurl{10.1039/c8nr09058k}
\end{barticle}
\endbibitem

\bibitem[\protect\citeauthoryear{Knoll et~al.}{2010}]{brediestgv15}
\begin{barticle}
\bauthor{\bsnm{Knoll}, \binits{F.}},
\bauthor{\bsnm{Bredies}, \binits{K.}},
\bauthor{\bsnm{Pock}, \binits{T.}},
\bauthor{\bsnm{Stollberger}, \binits{R.}}:
\batitle{Second order total generalized variation ({TGV}) for {MRI}}.
\bjtitle{Magnetic Resonance in Medicine}
\bvolume{65}(\bissue{2}),
\bfpage{480}--\blpage{491}
(\byear{2010})
\doiurl{10.1002/mrm.22595}
\end{barticle}
\endbibitem

\bibitem[\protect\citeauthoryear{Knoll et~al.}{2017}]{brediestgv17}
\begin{barticle}
\bauthor{\bsnm{Knoll}, \binits{F.}},
\bauthor{\bsnm{Holler}, \binits{M.}},
\bauthor{\bsnm{Koesters}, \binits{T.}},
\bauthor{\bsnm{Otazo}, \binits{R.}},
\bauthor{\bsnm{Bredies}, \binits{K.}},
\bauthor{\bsnm{Sodickson}, \binits{D.K.}}:
\batitle{Joint {MR}-{PET} reconstruction using a multi-channel image
  regularizer}.
\bjtitle{IEEE Transactions on Medical Imaging}
\bvolume{36}(\bissue{1}),
\bfpage{1}--\blpage{16}
(\byear{2017})
\doiurl{10.1109/tmi.2016.2564989}
\end{barticle}
\endbibitem

\bibitem[\protect\citeauthoryear{Langkammer et~al.}{2015}]{brediestgv16}
\begin{barticle}
\bauthor{\bsnm{Langkammer}, \binits{C.}},
\bauthor{\bsnm{Bredies}, \binits{K.}},
\bauthor{\bsnm{Poser}, \binits{B.A.}},
\bauthor{\bsnm{Barth}, \binits{M.}},
\bauthor{\bsnm{Reishofer}, \binits{G.}},
\bauthor{\bsnm{Fan}, \binits{A.P.}},
\bauthor{\bsnm{Bilgic}, \binits{B.}},
\bauthor{\bsnm{Fazekas}, \binits{F.}},
\bauthor{\bsnm{Mainero}, \binits{C.}},
\bauthor{\bsnm{Ropele}, \binits{S.}}:
\batitle{Fast quantitative susceptibility mapping using {3D} {EPI} and total
  generalized variation}.
\bjtitle{NeuroImage}
\bvolume{111},
\bfpage{622}--\blpage{630}
(\byear{2015})
\doiurl{10.1016/j.neuroimage.2015.02.041}
\end{barticle}
\endbibitem

\bibitem[\protect\citeauthoryear{Lore et~al.}{2017}]{deep1}
\begin{barticle}
\bauthor{\bsnm{Lore}, \binits{K.G.}},
\bauthor{\bsnm{Akintayo}, \binits{A.}},
\bauthor{\bsnm{Sarkar}, \binits{S.}}:
\batitle{{LLNet}: A deep autoencoder approach to natural low-light image
  enhancement}.
\bjtitle{Pattern Recognition}
\bvolume{61},
\bfpage{650}--\blpage{662}
(\byear{2017})
\doiurl{10.1016/j.patcog.2016.06.008}
\end{barticle}
\endbibitem

\bibitem[\protect\citeauthoryear{Papafitsoros and Bredies}{2015}]{tgv4}
\begin{barticle}
\bauthor{\bsnm{Papafitsoros}, \binits{K.}},
\bauthor{\bsnm{Bredies}, \binits{K.}}:
\batitle{A study of the one dimensional total generalised variation
  regularisation problem}.
\bjtitle{Inverse Problems \& Imaging}
\bvolume{9}(\bissue{2}),
\bfpage{511}--\blpage{550}
(\byear{2015})
\doiurl{10.3934/ipi.2015.9.511}
\end{barticle}
\endbibitem

\bibitem[\protect\citeauthoryear{Perona and Malik}{1990}]{dif1}
\begin{barticle}
\bauthor{\bsnm{Perona}, \binits{P.}},
\bauthor{\bsnm{Malik}, \binits{J.}}:
\batitle{Scale-space and edge detection using anisotropic diffusion}.
\bjtitle{IEEE Transactions on Pattern Analysis and Machine Intelligence}
\bvolume{12}(\bissue{7}),
\bfpage{629}--\blpage{639}
(\byear{1990})
\doiurl{10.1109/34.56205}
\end{barticle}
\endbibitem

\bibitem[\protect\citeauthoryear{Rudin et~al.}{1992}]{ROF}
\begin{barticle}
\bauthor{\bsnm{Rudin}, \binits{L.I.}},
\bauthor{\bsnm{Osher}, \binits{S.}},
\bauthor{\bsnm{Fatemi}, \binits{E.}}:
\batitle{Nonlinear total variation based noise removal algorithms}.
\bjtitle{Physica D: Nonlinear Phenomena}
\bvolume{60}(\bissue{1–4}),
\bfpage{259}--\blpage{268}
(\byear{1992})
\doiurl{10.1016/0167-2789(92)90242-f}
\end{barticle}
\endbibitem

\bibitem[\protect\citeauthoryear{Sardy et~al.}{2001}]{wav2}
\begin{barticle}
\bauthor{\bsnm{Sardy}, \binits{S.}},
\bauthor{\bsnm{Tseng}, \binits{P.}},
\bauthor{\bsnm{Bruce}, \binits{A.}}:
\batitle{Robust wavelet denoising}.
\bjtitle{IEEE Transactions on Signal Processing}
\bvolume{49}(\bissue{6}),
\bfpage{1146}--\blpage{1152}
(\byear{2001})
\doiurl{10.1109/78.923297}
\end{barticle}
\endbibitem

\bibitem[\protect\citeauthoryear{Storath and Weinmann}{2014}]{M2}
\begin{barticle}
\bauthor{\bsnm{Storath}, \binits{M.}},
\bauthor{\bsnm{Weinmann}, \binits{A.}}:
\batitle{Fast partitioning of vector-valued images}.
\bjtitle{SIAM Journal on Imaging Sciences}
\bvolume{7}(\bissue{3}),
\bfpage{1826}--\blpage{1852}
(\byear{2014})
\doiurl{10.1137/130950367}
\end{barticle}
\endbibitem

\bibitem[\protect\citeauthoryear{Storath et~al.}{2017}]{Stor1}
\begin{barticle}
\bauthor{\bsnm{Storath}, \binits{M.}},
\bauthor{\bsnm{Brandt}, \binits{C.}},
\bauthor{\bsnm{Hofmann}, \binits{M.}},
\bauthor{\bsnm{Knopp}, \binits{T.}},
\bauthor{\bsnm{Salamon}, \binits{J.}},
\bauthor{\bsnm{Weber}, \binits{A.}},
\bauthor{\bsnm{Weinmann}, \binits{A.}}:
\batitle{Edge preserving and noise reducing reconstruction for magnetic
  particle imaging}.
\bjtitle{IEEE Transactions on Medical Imaging}
\bvolume{36}(\bissue{1}),
\bfpage{74}--\blpage{85}
(\byear{2017})
\doiurl{10.1109/tmi.2016.2593954}
\end{barticle}
\endbibitem

\bibitem[\protect\citeauthoryear{Valkonen et~al.}{2013}]{brediestgv13}
\begin{barticle}
\bauthor{\bsnm{Valkonen}, \binits{T.}},
\bauthor{\bsnm{Bredies}, \binits{K.}},
\bauthor{\bsnm{Knoll}, \binits{F.}}:
\batitle{Total generalized variation in diffusion tensor imaging}.
\bjtitle{SIAM Journal on Imaging Sciences}
\bvolume{6}(\bissue{1}),
\bfpage{487}--\blpage{525}
(\byear{2013})
\doiurl{10.1137/120867172}
\end{barticle}
\endbibitem

\bibitem[\protect\citeauthoryear{Wang et~al.}{2013}]{deep2}
\begin{barticle}
\bauthor{\bsnm{Wang}, \binits{N.}},
\bauthor{\bsnm{Tao}, \binits{D.}},
\bauthor{\bsnm{Gao}, \binits{X.}},
\bauthor{\bsnm{Li}, \binits{X.}},
\bauthor{\bsnm{Li}, \binits{J.}}:
\batitle{A comprehensive survey to face hallucination}.
\bjtitle{International Journal of Computer Vision}
\bvolume{106}(\bissue{1}),
\bfpage{9}--\blpage{30}
(\byear{2013})
\doiurl{10.1007/s11263-013-0645-9}
\end{barticle}
\endbibitem

\bibitem[\protect\citeauthoryear{Wang et~al.}{2004}]{ssim}
\begin{barticle}
\bauthor{\bsnm{Wang}, \binits{Z.}},
\bauthor{\bsnm{Bovik}, \binits{A.C.}},
\bauthor{\bsnm{Sheikh}, \binits{H.R.}},
\bauthor{\bsnm{Simoncelli}, \binits{E.P.}}:
\batitle{Image quality assessment: From error visibility to structural
  similarity}.
\bjtitle{IEEE Transactions on Image Processing}
\bvolume{13}(\bissue{4}),
\bfpage{600}--\blpage{612}
(\byear{2004})
\doiurl{10.1109/tip.2003.819861}
\end{barticle}
\endbibitem

\bibitem[\protect\citeauthoryear{Weickert}{1998}]{dif2}
\begin{bbook}
\bauthor{\bsnm{Weickert}, \binits{J.}}:
\bbtitle{Anisotropic Diffusion in Image Processing}.
\bsertitle{ECMI Series}.
\bpublisher{Teubner-Verlag Stuttgart}
(\byear{1998}).
\burl{https://www.mia.uni-saarland.de/weickert/book.html}
\end{bbook}
\endbibitem

\bibitem[\protect\citeauthoryear{Wen et~al.}{2008}]{wav1}
\begin{barticle}
\bauthor{\bsnm{Wen}, \binits{Y.-W.}},
\bauthor{\bsnm{Ng}, \binits{M.K.}},
\bauthor{\bsnm{Ching}, \binits{W.-K.}}:
\batitle{Iterative algorithms based on decoupling of deblurring and denoising
  for image restoration}.
\bjtitle{SIAM Journal on Scientific Computing}
\bvolume{30}(\bissue{5}),
\bfpage{2655}--\blpage{2674}
(\byear{2008})
\doiurl{10.1137/070683374}
\end{barticle}
\endbibitem

\end{thebibliography}

	\section*{List of Symbols}\nopagebreak
        \newsavebox\WBox
         \sbox\WBox{$C_c^{k}(\Omega,\mathbb{R}^d)$}
         \begin{supertabular}{
             @{} 
             p{\wd\WBox}
             p{\dimexpr\columnwidth-\wd\WBox-2\tabcolsep\relax}}
           \textbf{Section~\ref{sec:intro}} \\
         $L^1_{loc}(\Omega)$ & space of locally integrable functions on $\Omega$ \\
         $\mathcal{F}$ & fidelity term (Tikhonov regularization)\\
         $\nabla$ & continuous gradient \\
         TV &  total variation (continuous and discrete) \\
         TGV & total generalized variation (continuous and discrete) \\
         $\text{TGV}^k$ &  $k$-th order total generalized variation \\
         $\text{TV}_c$  & Condat's TV \\
         \textbf{Section~\ref{sec:tv}} \\
         $C_c^{k}(\Omega,\mathbb{R}^d)$ & $k$-times continuously differentiable compactly supported $\mathbb{R}^d$-valued functions  \\
         $\text{div}$ & continuous divergence operator ($\text{div}: C_c^{1}(\Omega,\mathbb{R}^d)\rightarrow C^0(\Omega)$) \\
         $C^k(\Omega)$ & $k$-times continuously differentiable functions on $\Omega$ \\
         $W^{1,1}(\Omega)$ & Sobolev space of functions whose weak derivatives up to the order one, belong to $L^1(\Omega)$ \\
         \textbf{Section~\ref{sec:tgv}} \\         
         $\text{Sym}^{2}(\mathbb{R}^d)$ & space of symmetric $2$-tensors on $\mathbb{R}^d$  \\
         \rlap{$C_c^{2}(\Omega, \text{Sym}^{2}(\mathbb{R}^d))$} & \qquad\ \  two times continuously differentiable compactly supported $2$-tensor fields \\
         $\text{Div}$ & continuous divergence operator ($\text{Div}: C_c^{2}(\Omega, \text{Sym}^{2}(\mathbb{R}^d)\rightarrow C_c^1(\Omega, \mathbb{R}^d)$) \\
         $\text{div}^2$ & continuous second-order divergence operator ($\text{div}^2: C_c^{2}(\Omega, \text{Sym}^{2}(\mathbb{R}^d)\rightarrow C^0(\Omega)$) \\
         $D_{x+}$  & forward difference operator with respect to $x$-direction \\
         $D_{y+}$  & forward difference operator with respect to $y$-direction \\
         $\mathcal{D}$ & discrete gradient operator \\
         \textbf{Section~\ref{sec:iso_tv}} \\         
         $\text{TV}_i$ & isotropic TV (ROF model) \\
         \textbf{Section~\ref{sec:condat_tv}} \\         
         $L_\bullet$  & domain conversion operator to the center of a pixel for Condat's model \\
         $L_\leftrightarrow$  & domain conversion operator to the center of the horizontal edge of a pixel for Condat's model \\
         $L_\updownarrow$ & domain conversion operator to the center of the vertical edge of a pixel for Condat's model \\
         \textbf{Section~\ref{sec25}} \\         
         $D_{x-}$ & backwards difference operator with respect to $x$-direction \\
         $D_{y-}$ & backwards difference operator with respect to $y$-direction \\
         $\text{TGV}_{\alpha}^2$ & discrete classic second-order TGV with regularization parameters $\alpha=(\alpha_0,\alpha_1)$ \\
         $S(\mathbb{R}^4)^{N_1\times N_2}$ & space of discrete symmetric second-order tensor fields in $(\mathbb{R}^4)^{N_1\times N_2}$ \\
         $\text{Div}$ & discrete divergence operator ($\text{Div}: S(\mathbb{R}^4)^{N_1\times N_2}\rightarrow (\mathbb{R}^2)^{N_1\times N_2}$) \\
         $\mathcal{E}$ & adjoint of $-\text{Div}$  \\
         $\text{div}^2$ & discrete second order divergence operator ($\text{div}^2 : S(\mathbb{R}^4)^{N_1\times N_2}\rightarrow \mathbb{R}^{N_1\times N_2}$) \\
         $\mathcal{D}^2$ & adjoint of $\text{div}^2$ \\
         \textbf{Section~\ref{sec:other_tgv}} \\
         $\text{TGV}^{2(\alpha)}_{\text{SH} (n)}$ & $n$ Shannon second-order TGV \\
         $\mathbb{R}^+$ & set of positive real numbers \\
         $\mathcal{D}\mathcal{G}_r(\Omega)$ & the discontinuous Lagrange finite element spaces of order $r$ on $\Omega$ \\
         $\mathcal{R}\mathcal{T}_0(\Omega)$ & the lowest-order Raviart--Thomas finite element space on $\Omega$ \\
         \textbf{Section~\ref{sec:grids}} \\
         $A_{\bullet}$ & grid set of pixel centers \\
         $A_{\leftrightarrow}$ & grid set of horizontal edge centers of pixels \\ 
         $A_{\updownarrow}$ & grid set of vertical edge centers of pixels \\
         $\bar{A}^x_{\bullet}$ & extended grid set of pixel centers in $x$-direction \\
         $\bar{A}^y_{\bullet}$  & extended grid set of pixel centers in $y$-direction \\
         $A_{\times}$ & grid set of pixel corners \\
         $\mathcal{U}_\bullet$ &  space of real functions with domain $A_{\bullet}$ \\
         $\mathcal{U}_{\leftrightarrow}$ & space of real functions with domain $A_{\leftrightarrow}$ \\
         $\mathcal{U}_{\updownarrow}$ &  space of real functions with domain $A_{\updownarrow}$ \\
         $\bar{\mathcal{U}}_\bullet^x$ & space of real functions with domain $\bar{A}^x_{\bullet}$ \\
         $\bar{\mathcal{U}}_\bullet^y$ & space of real functions with domain $\bar{A}^y_{\bullet}$ \\
         $\bar{\mathcal{U}}_\times$ & space of real functions with domain $A_{\times}$ \\
         \textbf{Section~\ref{sec:diff_op}} \\
         $\mathcal{D}^{new}$  & the proposed discrete first-order derivative operator ($\mathcal{D}^{new} : \mathcal{U}_\bullet\rightarrow \mathcal{U}_\leftrightarrow\times\mathcal{U}_\updownarrow$) \\
         $\mathcal{E}^{new}$ & the proposed discrete first-order symmetrized derivative operator ($\mathcal{E}^{new} : \mathcal{U}_\leftrightarrow \times\mathcal{U}_\updownarrow\rightarrow \bar{\mathcal{U}}_{\bullet}^x\times \bar{\mathcal{U}}_{\bullet}^y\times \mathcal{U}_\times$) \\
         $\mathcal{D}^{2new}$ & the proposed discrete second-order derivative operator ($\mathcal{D}^{2new}:\mathcal{U}_\bullet\rightarrow \bar{\mathcal{U}}_{\bullet}^x\times \bar{\mathcal{U}}_{\bullet}^y\times \mathcal{U}_\times$) \\
         \textbf{Section~\ref{sec:div}} \\
         $\text{div}_{new}$ & the proposed discrete divergence operator (adjoint of $-\mathcal{D}^{new}$) \\
         $\text{Div}_{new}$ & the proposed discrete vector divergence operator (adjoint of $-\mathcal{E}^{new}$) \\
         $\text{div}^2_{new}$ & the proposed discrete second-order divergence operator (adjoint of $\mathcal{D}^{2new}$) \\
         \textbf{Section~\ref{sec35}} \\
       $L_\bullet$  & domain conversion operator to the pixel centers for the proposed model ($L_\bullet:{\mathcal{U}}_\leftrightarrow\times{\mathcal{U}}_\updownarrow\rightarrow {\mathcal{U}}_\bullet\times {\mathcal{U}}_\bullet$) \\
       $L_\leftrightarrow$ & domain conversion operator to the horizontal edge centers of pixels for the proposed model \\
       $L_\updownarrow$ & domain conversion operator to vertical edge centers of pixels for the proposed model \\
       $L_\bullet$ & domain conversion operator to the pixel centers for the proposed model ($L_\bullet: \bar{\mathcal{U}}_\bullet^x\times \bar{\mathcal{U}}_\bullet^y \times{\mathcal{U}}_{\times}\rightarrow {\mathcal{U}}_\bullet\times {\mathcal{U}}_\bullet\times {\mathcal{U}}_\bullet$) \\
       \textbf{Section~\ref{sec:model}} \\
       $\text{TGV}_{\alpha}^{2(new)}$ & proposed discrete second-order TGV with regularization parameters $\alpha=(\alpha_0,\alpha_1)$ \\
       $I_{K}$ & indicator function of the set $K$ \\
       \textbf{Section~\ref{subsec:extensions}} \\
       $\text{TV}_1$ & the discrete total variation introduced in \cite{Stor1} \\
       $\text{TV}_2$ & the discrete total variation introduced in \cite{Hosseini2} \\
       \textbf{Section~\ref{invariance}} \\
       $\mathcal{R}$ & $90^\circ$ rotation operator \\
       \textbf{Section~\ref{sec:algorithms}} \\
       $\text{prox}$ & proximal operator \\
       $\text{shrink}$ &shrinkage operator \\
   \end{supertabular}

        \begin{appendices}

	\section{A Higher-Order Finite-Difference Scheme for Discrete TV}
	\label{sec:central_diff}\setcounter{figure}{0}\setcounter{table}{0}		
	{%
		Recall that the basis for classic discrete TV \cite[Section 3]{tv} is the following finite-difference approximation of the derivative:
		\begin{equation}\label{Difff}
			f'(x)\approx \frac{f(x+h)-f(x)}{h},
		\end{equation}
		leading to the well-known two-point stencils of approximation order 1.
		In the following, instead of (\ref{Difff}), we use the following central differences formula for TV denoising and compare it with the classic discrete TV:
		\begin{equation}\label{Difff2}
			f'(x)\approx \frac{f(x+h) - f(x-h)}{2h}. %
		\end{equation}
		This approximation leads to a three-point finite-difference stencil which has approximation order 2.
		Using symmetric boundary conditions for $u\in\mathbb{R}^{N_1\times N_2}$, we define the central differences discrete TV as follows:
		\begin{align} \notag
				(\tilde{D}_x u)(n_1,n_2) &= \tfrac12 (u(n_1+1,n_2) - u(n_1-1,n_2)), \\ \notag
                  (\tilde{D}_y u)(n_1,n_2) &= \tfrac12 (u(n_1,n_2+1) - u(n_1,n_2-1)), \\ \notag & \qquad\qquad\qquad
				n_1 = 1,\ldots,N_1, \ n_2 = 1,\ldots,N_2,
				\\\label{testt}
				\displaystyle \widetilde{\text{TV}} (u) &= \sum_{n_1=1}^{N_1}\sum_{n_2=1}^{N_2} \sqrt{(\tilde{D}_x u)(n_1,n_2)^2+(\tilde{D}_y u)(n_1,n_2)^2}.
		\end{align}
		The $\widetilde{\text{TV}}$ model can easily be tested using a primal-dual algorithm similar to Algorithm~\ref{algorithmTV} below. Fig. \ref{test}, shows the reference image, a noisy version with zero-mean additive Gaussian noise of standard deviation $0.2$, and the restored images via classic discrete TV and central differences discrete TV according to (\ref{testt}).  One clearly recognizes that the restored image using (\ref{testt}) is inaccurate and admits undesired checkerboard artifacts. The reason could be that locally, checkerboard-like images induce a vanishing discrete gradient when using a central-difference approximation and are hence preferred by the model. A similar effect can be observed when using a 5-point finite-difference stencil and it appears that higher-order finite-difference schemes are generally not suitable for variational denoising problems. When designing finite-difference schemes for TV denoising, it seems to be essential to ensure that locally, the corresponding discrete gradient only vanishes where the image is constant. The property is certainly satisfied for the classical model that bases on~\eqref{Difff}.}		
	\begin{figure} \centering \fontsize{8}{9.5}\selectfont
		\begin{tabular}{c@{\ \ }c}
			\includegraphics[height=3.6cm]{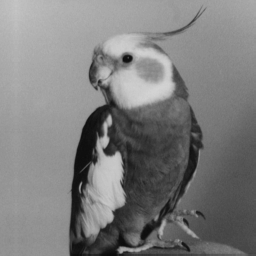}
                  & \includegraphics[height=3.6cm]{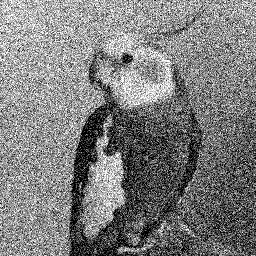} \\
                    			($a$) reference image
			& ($b$) noisy image \\[0.25em]
			\includegraphics[height=3.6cm]{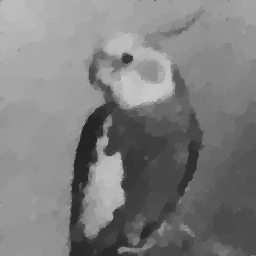}
			& \includegraphics[height=3.6cm]{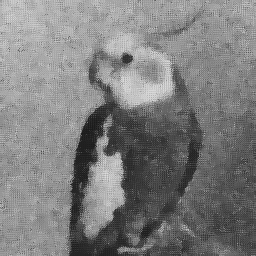}    \\
			($c$) TV-restored
			& ($d$) $\widetilde{\text{TV}}$-restored
		\end{tabular}
		\caption{{Test and comparison of the $\widetilde{\text{TV}}$ model. The chosen parameter is $\lambda^*=0.2$, the quality metrics are as follows: ($c$) PSNR=28.06, SSIM=0.8327, ($d$) PSNR=23.71, SSIM=0.3664. The change of the derivative approximation scheme leads to clearly visible artifacts and a significant drop of reconstruction quality.}}\label{test}
	\end{figure}		
	\section{Algorithms for Inverse Problems and Computational Complexities of Denoising}
	\label{sec:alg_comp}\setcounter{figure}{0}\setcounter{table}{0}		
	The primal-dual Algorithm~\ref{algorithm1} for {the inverse problems} (\ref{DPS}) $(a)$--$(c)$ could be realized by the computational schemes outlined in Algorithms \ref{algorithmTV}--\ref{algorithmTGV}. We provide them for the sake of completeness and reproducibility. In particular, they allow for a rough estimate of the computational complexity of these algorithms for denoising a given $N_1\times N_2$ gray image in terms of flops per iteration. These estimates are summarized in Table \ref{table:complexity}. In summary, we obtain the following relations of the number of the flops of the proposed algorithm with TGV and Condat-TV:
          \begin{align*}
            &\text{Number of flops for proposed TGV model} \\ & \qquad \approx 1.8\times \text{Number of flops for classic TGV model}, \\
            & \text{Number of flops for proposed TGV model} \\ & \qquad \approx 1.8\times \text{Number of flops for Condat-TV model}.
          \end{align*}
		The experimental MATLAB code in \cite{Hosseini_data} confirms that these relations also approximately transfer to the utilized CPU time.
		\begin{algorithm}[t]
                  			\caption{{An algorithm for solving the TV inverse
					problem (\ref{DPS}) $(a)$.}} \label{algorithmTV}
                                    \begin{algorithmic}
                                      \Require Degraded image $f \in Range(\mathcal{B})$, %
                                        $\lambda > 0$ %
                                      \Ensure For the iteration number $N$, $u^N$ is an approximation of the solution of the primal problem which is the reconstructed	image. Moreover, $(w^N, s^N)$ is an approximation of a solution of the dual problem
                                      \Initialization Choose $\sigma > 0$ and $\tau > 0$ such that $\sigma\tau \|{L_i}\|^2 < 1$, $L_i=\left(\begin{array}{c}\mathcal{D} \\ \mathcal{B}\end{array}\right)$ 
                                      \State Choose the initial approximation of a primal solution $u^0=\tilde{u}^0 = f \in{\mathbb{R}}^{N_1\times N_2}$ 
                                      \State Choose an arbitrary initial approximation of a dual solution  {$w^0\in (\mathbb{R}^{N_1\times N_2})^2, s^0\in Range(\mathcal{B})$} 
                                      \While{convergence criterion not met, for\ $k=0,1,\ldots$}
                                        \State $p^{k+1}=w^k+\sigma \mathcal{D}\tilde u^k$
                                         \State $\displaystyle w^{k+1} = \frac{\lambda p}{\max(|p|, \lambda)}$ \State $s^{k+1}=\frac{s^k+\sigma(\mathcal{B}\tilde{u}^k- f)}{1+\sigma}$
                                       \State   $\displaystyle u^{k+1} = {u^k - \tau (\mathcal{B}^* s^{k+1}-\text{div} \, w^{k+1})}$ 
                                          \State $\tilde{u}^{k+1} = 2u^{k+1} - u^k$
                                          \EndWhile 
                                      \end{algorithmic}
                                    \end{algorithm}
                                    \begin{algorithm}[t]
                                      			\caption{{An algorithm for solving the Condat-TV inverse
					problem (\ref{DPS}) $(b)$.}} \label{algorithmC}
                                    \begin{algorithmic}
                                      \Require Degraded image $f \in Range(\mathcal{B})$, %
                                        $\lambda > 0$ %
                                      \Ensure For the iteration number $N$, $z^N=(w_\bullet^N, w_\leftrightarrow^N, w_\updownarrow^N, u^N)$ is an approximation of the solution of the primal problem and $u^N$ is the obtained denoised			image. Moreover, $(w^N,s^N)$ is an approximation of a solution of the dual problem
                                      \Initialization Choose $\sigma > 0$ and $\tau > 0$ such that  $\sigma\tau \|L_c\|^2 < 1$, $L_c=\left(\begin{array}{cccc}L_{\bullet}^* & L_{\leftrightarrow}^* &  L_{\updownarrow} & -\mathcal{D}\\
					0 & 0 & 0 & \mathcal{B}
                                      \end{array}\right)$
                                  \State Choose the initial approximation $u^0=\tilde{u}^0 \in{\mathbb{R}}^{N_1\times N2}$
                                  \State Choose an arbitrary initial approximation of a primal solution  $w_\star^0= \tilde{w}_{\star}^0 \in (\mathbb{R}^{N_1\times N_2})^2$, $\star = \bullet, \leftrightarrow, \updownarrow$
                                  \State Choose an arbitrary initial approximation of a dual solution  {$w^0\in (\mathbb{R}^{N_1\times N_2})^2, s^0\in Range({\mathcal{B}})$} 
                                  \While{convergence criterion not met, for\ $k=0,1,\ldots$}
                                  \State $w^{k+1}=w^k+\sigma(L^*_\bullet \tilde w_\bullet^k + L^*_\leftrightarrow \tilde w_\leftrightarrow^k + L^*_\updownarrow \tilde w_\updownarrow^k - \mathcal{D}\tilde u^k)$
                                      \State $s^{k+1}=\frac{s^k+\sigma(\mathcal{B}\tilde{u}^k- f)}{1+\sigma}$ 
                                     \ForAll{$\star = \bullet, \leftrightarrow, \updownarrow$} \State $w_\star^{k+1} = \text{shrink}_{\lambda\tau}(w_\star^k - \tau L_\star w^{k+1})$ \EndFor
                                      \State $\displaystyle u^{k+1} = {u^k - \tau (\mathcal{B}^* s^{k+1}+\text{div} w^{k+1})}$ 
                                      \ForAll{$\star=\bullet, \leftrightarrow, \updownarrow$} \State $\tilde{w}_\star^{k+1}=2w^{k+1}_\star- w_\star^k$ \EndFor
                                      \State $\tilde{u}^{k+1} = 2u^{k+1} - u^k$
                                    \EndWhile
                                  \end{algorithmic}
                                \end{algorithm}
                                \begin{algorithm}[t]
                                  		\caption{{An algorithm for solving the second-order TGV inverse
					problem (\ref{DPS}) $(d)$.}} \label{algorithmTGV}
                                    \begin{algorithmic}
                                      \Require Degraded image $f \in Range({\mathcal{B}})$, %
                                        $\alpha=(\alpha_0,\alpha_1), \alpha_0 > 0, \alpha_1 > 0$ %
                                      \Ensure For the iteration number $N$, $z^N=(v^N, w^N, u^N, \omega^N)$ is an approximation of the solution of the primal problem  and $u^N$ is the obtained denoised			image. Moreover, {$y^N=(\bar{v}^N, \bar{w}^N, {s}^N)$} is an approximation of a solution of the dual problem
             \Initialization                    Choose $\sigma > 0$ and $\tau > 0$ such that $\sigma\tau \|L_{\text{TGV}}\|^2 < 1,$  $L_{\text{TGV}} = \left(\begin{array}{cccc} I & 0 & 0 & -\mathcal{E} \\ 0 & I &  -\mathcal{D} & I\\
					0 & 0 & \mathcal{B} & 0 
                                      \end{array}\right)$
                                  \State Choose the initial approximation $u^0=\tilde{u}^0  \in\mathbb{R}^{N_1\times N_2}$ 
                                  \State Choose an arbitrary initial approximation of a primal solution $v^0 = \tilde{v}^0 \in (\mathbb{R}^{N_1\times N_2})^3$, 
                                  $w^0= \tilde{w}^0 \in (\mathbb{R}^{N_1\times N_2})^2$, $\omega^0 = \tilde{\omega}^0 \in(\mathbb{R}^{N_1\times N_2})^2$
                                  \State Choose an arbitrary initial approximation of a dual solution $\bar{v}^0\in (\mathbb{R}^{N_1\times N_2})^3$, $\bar{w}^0\in (\mathbb{R}^{N_1\times N_2})^2, {s^0\in Range({\mathcal{B}})}$ 
                                  \While{convergence criterion not met, for $k=0,1,\ldots$}
                                  \State $\bar{v}^{k+1}=\bar{v}^k+\sigma(\tilde{v}^k-\mathcal{E}\tilde\omega^k)$
                                  \State $\bar{w}^{k+1}=\bar{w}^k+\sigma(\tilde {w}^k - \mathcal{D}\tilde u^k +\tilde\omega^k)$
                                  \State    $s^{k+1}=\frac{s^k+\sigma(\mathcal{B}\tilde{u}^k- f)}{1+\sigma}$
                           \State           $v^{k+1} = \text{shrink}_{\alpha_0\tau}(v^k - \tau \bar{v}^{k+1})$ \State $w^{k+1} = \text{shrink}_{\alpha_1\tau}(w^k - \tau\bar{w}^{k+1})$
            \State                          $\displaystyle u^{k+1} = {u^k - \tau (\mathcal{B}^* s^{k+1}+\text{div} \, w^{k+1})}$ \State $\omega^{k+1} = \omega^k - \tau(\text{Div} \, \bar{v}^{k+1} + \bar{w}^{k+1})$ \State
                                      $\tilde{v}^{k+1}=2v^{k+1}-v^k$ %
                                      \State $\tilde{w}^{k+1}=2w^{k+1}- w^k$ 
                                      \State $\tilde{u}^{k+1} = 2u^{k+1} - u^k$ \State $\tilde\omega^{k+1} = 2\omega^{k+1} - \omega^k$
                                    \EndWhile
                                  \end{algorithmic}
                                \end{algorithm}
	\begin{table} \centering \tablebodyfont
          		\caption{Computational complexity per iteration of the denoising algorithms ($\mathcal{B}=I$) for a given $N_1\times N_2$ noisy image, i.e., classic TV (Algorithm \ref{algorithmTV}), Condat-TV (Algorithm \ref{algorithmC}), classic second-order TGV (Algorithm \ref{algorithmTGV}), and the proposed TGV (Algorithm \ref{algorithm2}).} %
		\label{table:complexity}
		\begin{tabular}{lcccc}
			\toprule
			Operators/Model              & TV         & Condat-TV       & TGV         & Proposed     \\
			\midrule
			Number of \\	square roots        & $N_1N_2$   & $3N_1N_2$       & $2N_1N_2$   & $4 N_1N_2$   \\
			Number of \\	comparisons         & $N_1N_2$   & $3N_1N_2$       & $2N_1N_2$   & $4 N_1N_2$   \\
			\midrule
			&            & Number of flops &             &              \\
			\midrule
			Gradients and  \\ divergences    & $4N_1N_2$  & $4N_1N_2$       & $16N_1N_2$  & $ 16N_1N_2$  \\
			Grid converters \\ and adjoints & --         & $24N_1N_2$      & --          & $ 30N_1N_2$  \\
			Shrinkage \\ operator           & $7N_1N_2$  & $21N_1N_2$      & $17 N_1N_2$ & $31 N_1N_2$  \\
			Main body                    & $11N_1N_2$ & $41N_1N_2$      & $54N_1N_2$  & $82 N_1N_2$  \\
			\midrule
			Total number \\ of flops        & $22N_1N_2$ & $90N_1N_2$      & $87N_1N_2$  & $159 N_1N_2$ \\
			\botrule
		\end{tabular}
	\end{table}

      \end{appendices}
    \end{document}